\documentclass[dvipdfm]{amsart}
\usepackage{amsmath}
\usepackage{amssymb} 
\usepackage{amscd} 
\usepackage[dvipdfmx]{graphicx} 
\usepackage{epsfig}
\usepackage[dvips]{xcolor}
\usepackage{marvosym}
\allowdisplaybreaks 

\usepackage{booktabs} 

\usepackage{array}
\newcolumntype{M}[1]{>{\centering\arraybackslash}m{#1}} 

\newcommand{\Z}{\mathbb Z}
\newcommand{\N}{\mathbb N}
\newcommand{\Q}{\mathbb Q}

\newcommand{\calL}{\mathcal L}

\newcommand{\nbhd}{\mathcal{N}}
\newcommand{\bdry}{\partial}

\newcommand{\nos}{\normalsize}
\newcommand{\bpc}[1]{\begin{picture} #1 \end{picture}}

\newcommand{\pcr}[2]{\mbox{$\begin{array}{c}
   \includegraphics[scale=#2]{#1.eps}
   \end{array}$}}

\newcommand{\dis}{\displaystyle}

\newtheorem{theorem}{Theorem}[section]
\newtheorem{lemma}[theorem]{Lemma}

\newtheorem{proposition}[theorem]{Proposition}
\newtheorem{corollary}[theorem]{Corollary}
\newtheorem{claim}[theorem]{Claim}

\newtheorem{conjecture}[theorem]{Conjecture}
\newtheorem{remark}[theorem]{Remark}
\newtheorem{question}[theorem]{Question}

\newtheorem*{thm_slope conjecture Wmt}{Theorem~\ref{bothslopeconjecturesWmt}(1)}
\newtheorem*{thm_strong slope conjecture Wmt}{Theorem~\ref{bothslopeconjecturesWmt}(2)}

\newtheorem*{cor_finite-sequence}{Corollary~\ref{finite-sequence}}

\theoremstyle{definition}
\newtheorem{definition}[theorem]{Definition}
\newtheorem{example}[theorem]{Example}
\newtheorem{convention}[theorem]{Convention}

\numberwithin{equation}{section}
\numberwithin{figure}{section}
\numberwithin{table}{section}

\definecolor{dartmouthgreen}{rgb}{0.05, 0.5, 0.06}

\renewcommand{\(}{\textup{(}}
\renewcommand{\)}{\textup{)}}

\begin{document}
\baselineskip 14pt

\title{The Strong Slope Conjecture for twisted generalized Whitehead doubles}

\author[K.L. Baker]{Kenneth L. Baker}
\address{Department of Mathematics, University of Miami, 
Coral Gables, FL 33146, USA}
\email{k.baker@math.miami.edu}

\author[K. Motegi]{Kimihiko Motegi}
\address{Department of Mathematics, Nihon University, 
3-25-40 Sakurajosui, Setagaya-ku, 
Tokyo 156--8550, Japan}
\email{motegi@math.chs.nihon-u.ac.jp}

\author[T. Takata]{Toshie Takata}
\address{Graduate School of Mathematics, Kyushu University, 
744 Motooka, Nishi-ku, Fukuoka 819--0395, Japan}
\email{ttakata@math.kyushu-u.ac.jp}

\dedicatory{}

\begin{abstract}
The Slope Conjecture proposed by Garoufalidis asserts that the degree of the colored Jones polynomial determines 
a boundary slope, and its refinement, the Strong Slope Conjecture proposed by Kalfagianni and Tran asserts that the linear term in the degree determines the topology of an essential surface that satisfies the Slope Conjecture.  
Under certain hypotheses, we show that twisted, generalized Whitehead doubles of a knot satisfies the Slope Conjecture and the Strong Slope Conjecture if the original knot does.  
Additionally, we provide a proof that there are Whitehead doubles which are not adequate.
\end{abstract}

\maketitle

\renewcommand{\thefootnote}{}
\footnotetext{
}

\tableofcontents

\section{Introduction}
\label{section:Introduction}

Let $K$ be a knot in the $3$--sphere $S^3$.   
The Slope Conjecture of Garoufalidis \cite{Garoufalidis} and the Strong Slope Conjecture of Kalfagianni and Tran \cite{KT} propose relationships between a quantum knot invariant, 
the degrees of the colored Jones function of $K$, 
and a classical invariant, the boundary slope and the topology of essential surfaces in the exterior of $K$.

The \textit{colored Jones function} of $K$ is a sequence of Laurent polynomials $J_{K, n}(q) \in \mathbb{Z}[q^{\pm\frac{1}{2}}]$ for $n \in \mathbb{N}$, 
where $J_{\bigcirc, n}(q)=\frac {q^{n/2}-q^{-n/2}}{q^{1/2}-q^{-1/2}}$ for the unknot $\bigcirc$ and 
$\frac {J_{K, 2}(q)}{J_{\bigcirc, 2}(q)}$ is the ordinary Jones polynomial of $K$. 
Since the colored Jones function is $q$--holonomic \cite[Theorem~1]{GL}, 
the degrees of its terms are given by \textit{quadratic quasi-polynomials} for suitably large $n$ \cite[Theorem 1.1 \& Remark 1.1]{Gqqp}.   
For the maximum degree $d_+[J_{K,n}(q)]$, 
we set the quadratic quasi-polynomials to be 
\[ \delta_K(n) = a(n) n^2 + b(n) n+ c(n) \]
for rational valued periodic functions $a(n), b(n), c(n)$ with integral period.
Now define the sets of {\em Jones slopes} of $K$:
\[  js(K) = \{ 4a(n) \ |\ n \in \mathbb{N} \}.\]

Allowing surfaces to be disconnected, we say a properly embedded surface in a $3$--manifold is {\em essential} if each component is orientable, incompressible, and boundary-incompressible. A number $p/q \in \Q \cup \{\infty\}$ is a {\em boundary slope} of a knot $K$ if there exists an essential surface in the knot exterior $E(K)=S^3-\mathrm{int}N(K)$ with a boundary component representing $p[\mu]+q[\lambda] \in H_1(\bdry E(K))$ with respect to the standard meridian $\mu$ and longitude $\lambda$.  
Now define the set of  boundary slopes of $K$:
\[ bs(K) = \{ r \in \mathbb{Q}\cup \{\infty\}\ |\ r\ \mbox{ is a boundary slope of }\ K\}. \]
Since a Seifert surface of minimal genus is an essential surface, 
$0 \in bs(K)$ for any knot. 
Let us also remark that $bs(K)$ is always a finite set \cite[Corollary]{Hat1}. 

Garoufalidis conjectures that Jones slopes are boundary slopes.

\begin{conjecture}[\textbf{Slope Conjecture} \cite{Garoufalidis}]
\label{slope conjecture}
For any knot $K$ in $S^3$,  
every Jones slope is a boundary slope.  
That is $js(K) \subset bs(K)$. 
\end{conjecture}

Garoufalidis' Slope Conjecture concerns only the quadratic terms of $\delta_K(n)$.
Recently Kalfagianni and Tran have proposed the Strong Slope Conjecture which subsumes the Slope Conjecture and asserts that the topology of the surfaces whose boundary slopes are Jones slopes may be predicted by the linear terms of $\delta_K(n)$.   
Define
\[  jx(K) =  \{ 2b(n) \ |\ n \in \mathbb{N} \}. \]

Let $K$ be a knot in $S^3$ with $\delta_K(n) = a(n) n^2 + b(n) n+ c(n)$. 
For a given Jones slope $p/q \in js(K)$ ($q > 0$), 
we say that $p/q$ satisfies $SS(n)$ ($n \in \N$)
if there is an essential surface $F_n$ in the exterior of $K$ 
such that 
\begin{itemize}
\item $F_n$ has the boundary slope $4a(n) = p/q$, and
\item $\displaystyle 2b(n) = \frac{\chi(F_n)}{|\bdry F_n| q}$.
\end{itemize}

\begin{conjecture}[\textbf{The Strong Slope Conjecture} \cite{KT, K-AMSHartford}]
\label{YSSC2}
For any knot in $S^3$, 
every Jones slope satisfies $SS(n)$ for some $n \in \mathbb{N}$. 
\end{conjecture}

It is convenient to say that 
$K$ satisfies the 
\textit{Strong Slope Conjecture with $SS(n_0)$}, 
if a Jones slope $p/q = 4a(n_0)$ satisfies $SS(n_0)$. 
(This is weaker than the Strong Slope Conjecture in the sense that 
we do not consider if another Jones slope $p'/q'$ other than $p/q$ satisfies $SS(n_0)$ for some $n_0$. 
Actually to state our main theorems, 
we use this with $n_0 = 1$; see Theorems~\ref{bothslopeconjecturesWmt} and \ref{slopeconjectureWmtSignCond}.

\begin{remark}
\label{constant}
Let $K$ be a knot with constant $a(n)$ \(i.e. the period of $a(n)$ is $1$\).  
Then it has a single Jones slope $p/q$ and 
if it satisfies $SS(n_0)$ for some $n_0$, then $K$ satisfies the Strong Slope Conjecture. 
If $b(n)$ is also constant, then we may take $n_0 = 1$ and the Strong Slope Conjecture 
implies the Strong Slope Conjecture with $SS(1)$. 
Presently, no knots are known for which the period of $a(n)$ is not $1$. 
\end{remark}

\medskip

\begin{example}
\label{earlyexamples}
Let $K$ be a knot which appears in the following list.
\begin{enumerate}
\item Torus knots \cite{Garoufalidis}, \cite[Theorem 3.9]{KT}. 
\item Adequate knots \cite{FKP}, \cite[Lemma 3.6, 3.8]{KT}, and hence alternating knots. 
\item Non-alternating knots with up to $9$ crossings except for  $8_{20}$, $9_{43}$, $9_{44}$ \cite{Garoufalidis}, \cite{KT,Ho}.   
($8_{20}$, $9_{43}$, $9_{44}$ satisfy the Strong Slope Conjecture, 
but for these knots the coefficient $b(n)$ has period $3$.)
\item Graph knots \cite{MT,BMT_graph}. 
\end{enumerate}
Then, writing 
$\delta_K(n) = a(n) n^2 + b(n) n + c(n)$,
we have that  
$a(n), b(n)$ are constant, 
and $c(n)$ has period at most two. 
Moreover, $K$ satisfies the Slope Conjecture, the unique Jones slope satisfies $SS(1)$, 
and hence $K$ satisfies the Strong Slope Conjecture \(Remark~\ref{constant}\).  
Note also that if $K$ is nontrivial, then $b(n) = b \le 0$. 
\end{example}

\subsection{Main Results}

In this article we give further supporting evidence for the Strong Slope Conjecture and by 
examining them for the Whitehead doubles of a knot $K$, 
and more generally for its twisted generalized Whitehead doubles $W_{\omega}^{\tau}(K)$ defined below. 

Let $V$ be a standardly embedded solid torus in $S^3$ with a preferred meridian-longitude 
$(\mu_V, \lambda_V)$, and take a pattern $(V,  k_{\omega}^{\tau})$ where $k_{\omega}^{\tau}$ is a knot in the interior of $V$ 
illustrated by Figure~\ref{fig:satellite}.  
We always assume $\omega \ne 0$, for otherwise, 
$k_{\omega}^{\tau}$ is the unknot contained in a $3$--ball in $V$. 
Given a knot $K$ in $S^3$ with a preferred meridian-longitude $(\mu_K, \lambda_K)$, 
let $f \colon V \to S^3$ an orientation preserving embedding which sends the core of $V$ to the knot $K \subset S^3$ such that 
$f(\mu_V) = \mu_K$ and $f(\lambda_V) = \lambda_K$.
Then the image $f(k_{\omega}^{\tau})$ is called a {\em $\tau$--twisted, $\omega$--generalized Whitehead double of $K$} 
and is denoted by $W_{\omega}^{\tau}(K)$. 
When $\omega = 1$, $\tau = 0$, $W_{1}^0$ is the (untwisted) negative Whitehead double of $K$.
Note that the mirror image $\overline{W_{\omega}^{\tau}(K)}$ of $W_{\omega}^{\tau}(K)$ is $W_{-\omega}^{-\tau}(\overline{K})$. 

\begin{figure}[!ht]
\includegraphics[width=0.85\linewidth]{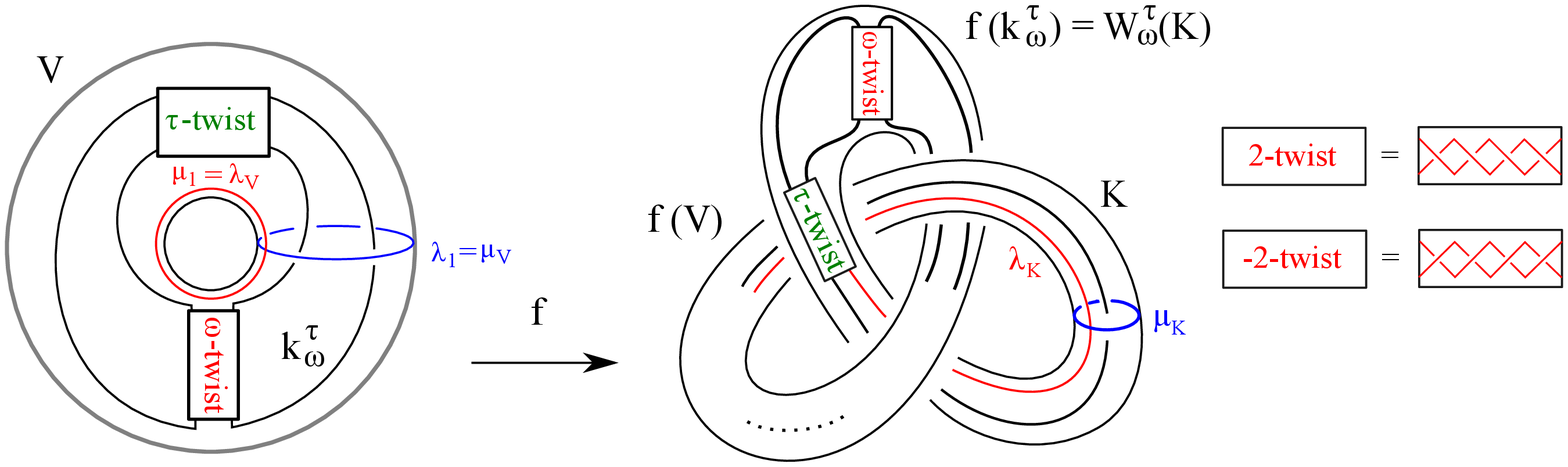}
\caption{Twisted generalized negative Whitehead double of $K$;
$f \colon V \to S^3$ is a faithful embedding and it maps the core of $V$ to $K$.}
\label{fig:satellite}
\end{figure}

For notational simplicity, in what follows, 
we use the following notation. 

\begin{convention}
\label{quadratic quasi-polynomial}
For a given knot $K$, 
let $N_K$ be the smallest nonnegative integer such that 
$d_+[J_{K,n}(q)]$ is a quadratic quasi-polynomial $\delta_K(n)=a(n) n^2 +b(n) n +c(n)$ 
for $n \ge 2N_K+1$. 
We put $a_1:=a (2N_K + 1)$, $b_1:=b (2N_K + 1)$.   
\end{convention}

The aim of this paper is to establish the Slope Conjecture and 
the Strong Slope Conjecture for twisted generalized Whitehead doubles 
in the following form. 

\begin{theorem}
\label{bothslopeconjecturesWmt}
Let $K$ be a knot in $S^3$ where $d_+[J_{K,n}(q)]$ is the quadratic quasi-polynomial 
$\delta_K(n)=a(n) n^2 +b(n) n+c(n)$ for all $n \ge 0$.
We put $a_1:=a(1)$, $b_1:=b(1)$, and $c_1:=c(1)$. 
We assume that the period of $d_+[J_{K,n}(q)]=\delta_K(n)$ is less than or equal to  $2$ and that 
$b_1 \le 0$.
Assume further that if $b_1 = 0$, 
then $a_1 \ne \frac{\tau}{4}$. 
\begin{enumerate}
\item
If $K$ satisfies the Slope Conjecture, 
then all of its twisted generalized Whitehead doubles also satisfy the Slope Conjecture.
\item
If $K$ satisfies the Strong Slope Conjecture with $SS(1)$,  
then all of its twisted generalized Whitehead doubles satisfy the Strong Slope Conjecture. 
In particular, each of these doubles has a unique Jones slope, and that slope satisfies $SS(1)$. 
\end{enumerate}
\end{theorem}

In the next theorem we requires that $K$ satisfies the Sign Condition; see Definition~\ref{sign}. 
At the moment we have no examples of knots which do not satisfy the Sign Condition.

\begin{theorem}
\label{slopeconjectureWmtSignCond}
Let $K$ be a knot that satisfies the Sign Condition.
Assume that the period of $\delta_K(n)$ is less than or equal to  $2$ and  $b_1 \le 0$
where $b_1 = 0$ implies $a_1 \ne \frac{\tau}{4}$. 
If $K$ satisfies the Strong Slope Conjecture with $SS(1)$, 
then all of its twisted generalized Whitehead doubles with $\omega >0$ satisfy the Strong Slope Conjecture. 
In particular, each of these doubles has a unique Jones slope, and that slope satisfies $SS(1)$.  
\end{theorem}

\begin{remark}\
\begin{enumerate}
\item The hypotheses that $b_1 \le 0$ is actually implied by $K$ being a non-trivial knot that satisfies the Strong Slope Conjecture since the only essential surface with boundary and positive Euler characteristic is the disk. 
See also \cite[Conjecture 5.1]{KT}.
\item The hypothesis that the quasi-polynomials have period $\leq 2$ allows simplifications in the proofs that lead to Theorem~\ref{bothslopeconjecturesWmt}.  
This is in part due to the pattern for a twisted generalized Whitehead double having wrapping number $2$ and its effect upon the colored Jones polynomial for the satellite, 
see Proposition~\ref{prop:equationsplitting}.  
Indeed, allowing periods $>2$ significantly complicates Propositions~\ref{maxdeg-signcond} and \ref{maxdeg_KT_Wmt-allw}.
\item The twisted generalized Whitehead doubles considered in Theorems~\ref{bothslopeconjecturesWmt} and \ref{slopeconjectureWmtSignCond} have quasi-polynomials that are actually just polynomials; 
see Remark~\ref{no_period}. 
\item As mentioned in Remark~\ref{constant}, the coefficients $a(n)$ is constant for all known examples. 
Furthermore, it appears that for any knot known to satisfy the Strong Slope Conjecture, its unique Jones slope satisfies $SS(1)$. Hence the hypothesis that ``$K$ satisfies the Strong Slope Conjecture with $SS(1)$'' holds in Theorems~\ref{bothslopeconjecturesWmt}(2) and \ref{slopeconjectureWmtSignCond} for knots known to satisfy the Strong Slope Conjecture. 
\end{enumerate}
\end{remark}

Theorem~\ref{slopeconjectureWmtSignCond} (and Theorem~\ref{bothslopeconjecturesWmt}) has the hypothesis ``if $b_1=0$ then $a _1 \ne \frac{\tau}{4}$''. 
Since knots appearing in the next corollary satisfy 
``if $b_1=0$ then $a _1 \ne 0$'',
  taking $\tau = 0$ ensures the condition ``if $b_1=0$ then $a _1 \ne \frac{\tau}{4}$'' is always satisfied. 
So assuming $\tau = 0$, 
we have the following corollary which is proven in Section~\ref{corollary}. 
Let us note that the corollary uses Theorem~\ref{slopeconjectureWmtSignCond} because the iterated cablings and Whitehead doublings may only have $\delta=d_+$ for suitably large $n$.

\begin{corollary}
\label{finite-sequence}
Any knot obtained by a finite sequence of cabling, 
untwisted $\omega$--generalized Whitehead doublings with $\omega > 0$ 
and connected sums of $B$--adequate knots or torus knots satisfies the Slope Conjecture and 
the Strong Slope Conjecture. 
\end{corollary}

\medskip

\section{Colored Jones polynomials of generalized Whitehead doubles and their degrees} 
\label{Jones_generalized_W}

The primary goal of this section is to prove Proposition~\ref{maxdeg-signcond} below for $\tau$--twisted $\omega$--generalized Whitehead doubles with $\omega>0$.  
As a warm up, we will first prove Proposition~\ref{maxdeg_KT_Wmt-allw} for $\tau$--twisted $\omega$--generalized Whitehead doubles for all $\omega \neq 0$, but under different and stronger hypotheses.  
(However, such hypotheses are known to  not be satisfied by all knots.)

Proposition~\ref{maxdeg-signcond} requires knots to satisfy the Sign Condition.  

\begin{definition}
\label{sign}
Let $\varepsilon_n(K)$  be the sign of the coefficient of the term of the maximum degree  of $J_{K, n}(q)$. 
A knot $K$ satisfies the \textit{Sign Condition} if 
$\varepsilon_m(K) = \varepsilon_n(K)$ for $m \equiv n\ \mathrm{mod}\ 2$. 
\end{definition}

In \cite{BMT_graph}, 
we demonstrate that torus knots and $B$--adequate knots satisfy the Sign Condition. 
This fact will be used in the proof of Corollary~\ref{finite-sequence}. 
See also Definition~\ref{kappa} and Proposition~\ref{induction} which provide a broader class of knots with the Sign Condition. 

\begin{question}
Does every knot satisfy the Sign Condition?
\end{question}

\begin{proposition}[\textbf{Sign Condition}]
\label{maxdeg-signcond} 
Let $K$ be a knot in $S^3$ that satisfies the Sign Condition.    
Let $N_K$ the smallest nonnegative integer such that 
$d_+[J_{K,n}(q)]$ is a quadratic quasi-polynomial $\delta_K(n)=a(n) n^2 +b(n) n +c(n)$ 
for $n \ge 2N_K+1$. 
We put $a_1:=a (2N_K + 1)$, $b_1:=b (2N_K + 1)$, and $c_1:=c (2N_K + 1)$. 
We assume that the period of $\delta_K(n)$ is less than or equal to  $2$ and that 
$b_1 \le 0$.
Assume further that if $b_1 = 0$, 
then $a_1 \ne \frac{\tau}{4}$. 
Then the maximum degree of the colored Jones polynomial 
of its $\tau$--twisted $\omega$--generalized Whitehead double with $\omega>0$ 
is given by the quadratic polynomial

\begin{equation}
\delta_{W_{\omega}^{\tau}(K)}(n)=
\left\{ \begin{array}{ll} 
( 4a_1-\tau) n^2 +(- 4a_1+2b_1+\tau - \frac{1}{2}) n+a_1-b_1+c_1+\frac 1 2 & (a_1> \frac \tau 4), \\
-\frac 1 2 n+C_+(K,\tau)+\frac 1 2  &  (a_1 < \frac \tau 4), \\
-\frac{1}{2}n + C_+(K, \tau) + \frac{1}{2} &  (a_1 = \frac{\tau}{4}, b_1 \ne 0),
\end{array} \right. 
\end{equation}
where 
$ C_+(K, \tau)$ is a number that only depends on the knot $K$ and the number $\tau$.
Furthermore $W_{\omega}^{\tau}(K)$ satisfies the Sign Condition.
\end{proposition}

\begin{proposition}[\textbf{$d_+=\delta$}]
\label{maxdeg_KT_Wmt-allw}
Let $K$ be a knot in $S^3$ and $d_+[J_{K,n}(q)]$ is the quadratic quasi-polynomial 
$\delta_K(n)=a(n) n^2 +b(n) n+c(n)$ for all $n \ge 0$.
We put $a_1:=a(1)$, $b_1:=b(1)$, and $c_1:=c(1)$. 
We assume that the period of $d_+[J_{K,n}(q)]=\delta_K(n)$ is less than or equal to  $2$ and that 
$b_1 \le 0$.
Assume further that if $b_1 = 0$, 
then $a_1 \ne \frac{\tau}{4}$. 
Then, for suitably large $n$, the maximum degree of the colored Jones polynomial 
of its $\tau$--twisted $\omega$--generalized Whitehead double is given by one of the following quadratic polynomials:

If $\omega > 0$, then 
\begin{equation}
  \delta_{W_{\omega}^{\tau}(K)}(n)=
	\begin{cases}
            ( 4a_1-\tau) n^2 +(- 4a_1+2b_1+\tau - \frac{1}{2}) n+a_1-b_1+c_1+\frac 1 2 & (a_1> \frac \tau 4), \\
             -\frac 1 2 n+c_1+\frac 1 2  &  (a_1 < \frac \tau 4), \\
             -\frac{1}{2}n + c_1 + \frac{1}{2} &  (a_1 = \frac{\tau}{4}, b_1 \ne 0).
        \end{cases}
\end{equation}
In fact, if either $a_1 < \frac \tau 4$ or $a_1 = \frac{\tau}{4}$ with $b_1 \ne 0$, then $d_+[J_{W_{\omega}^{\tau}(K),n}(q)] =  \delta_{W_{\omega}^{\tau}(K)}(n)$ for all $n \geq 1$.

If $\omega < 0$, then
\begin{equation}
  \delta_{W_{\omega}^{\tau}(K)}(n)=
\begin{cases}
(4a_1- \tau-\omega-\frac{1}{2}) n^2+(-4a_1+2b_1 +\omega+ \tau + 1)n+a_1-b_1+c_1-\frac 1 2
      &(a_1 > \frac{\tau}{4}+\frac{1}{8}) \\
- \omega n^2 + (\omega+\frac 1 2) n  + a_1+b_1+c_1-\frac 1 2 & (a_1 < \frac{\tau}{4}+\frac{1}{8})\\
- \omega n^2 + (\omega+\frac 1 2) n  + a_1+b_1+c_1-\frac 1 2 & (a_1 = \frac{\tau}{4}+\frac{1}{8}, b_1 \neq 0).
\end{cases}
\end{equation}
In fact, if either $a_1< \frac{\tau}{4} +\frac{1}{8}$ or $a_1 = \frac{\tau}{4} +\frac{1}{8}$ with $b_1 \ne 0$, 
then $d_+[J_{W_{\omega}^{\tau}(K),n}(q)] = {\delta}_{W_{\omega}^{\tau}(K)}(n)$ for all $n \geq 1$.
\end{proposition}

\begin{remark}
\label{no_period}
Propositions~\ref{maxdeg-signcond} and \ref{maxdeg_KT_Wmt-allw} say that 
even when $\delta_K(n)$ has period $2$, 
$\delta_{W_{\omega}^{\tau}(K)}(n)$  is a usual polynomial rather than a
quasi-polynomial.  
\end{remark}

In Subsection~\ref{computation_Jones} we first derive a formula of the colored Jones function of $W_{\omega}^{\tau}(K)$ 
using a slightly different normalization that simplifies calculation.
For knot $K$ and a nonnegative integer $n$, set
\[J'_{K, n}(q):=\frac{J_{K, n+1}(q)}{J_{\bigcirc, n+1}(q)}\] 
so that $J'_{\bigcirc, n}(q)=1$ for the unknot $\bigcirc$ and 
$J'_{K, 1}(q)$ is the ordinary Jones polynomial of a knot $K$. 

For our proof of Proposition~\ref{maxdeg_KT_Wmt-allw},  
we study the behavior of the maximum degrees of the colored Jones functions of $\tau$-twisted  $\omega$-generalized Whitehead doubles with $\omega > 0$ and 
$\omega < 0$, respectively. 
Let $\delta'_{K}(n)$ be the maximum degrees of this normalized colored Jones function for large $n$. 

Let us introduce the Normalized Sign Condition which is analogous to the Sign Condition for $K$.

\begin{definition}
\label{sign'}
Let $\varepsilon'_n(K)$  be the sign of the coefficient of the term of the maximum degree  of $J'_{K, n}(q)$. 
A knot $K$ satisfies the \textit{Normalized Sign Condition} if 
$\varepsilon'_m(K) = \varepsilon'_n(K)$ for $m \equiv n\ \mathrm{mod}\ 2$. 
\end{definition}

\begin{proof}[Proof of Propositions~\ref{maxdeg-signcond} and \ref{maxdeg_KT_Wmt-allw}]
By the above definition of $J'_{K, n}(q)$, 
if $K$ satisfies the Sign Condition, 
then it also satisfies the Normalized Sign Condition.  
With respect to this normalization we first establish 
Propositions~\ref{maxdeg_J'_omega>0}, \ref{maxdeg_J'_omega<0} and \ref{maxdeg_J'_signcond} in Subsection~\ref{computation_degrees}, 
which derive Propositions~\ref{maxdeg_KT_Wmt-allw} and \ref{maxdeg-signcond} by using the transformation described below. 

To derive Propositions~\ref{maxdeg-signcond} and \ref{maxdeg_KT_Wmt-allw} from 
Propositions~\ref{maxdeg_J'_signcond}, \ref{maxdeg_J'_omega>0} and \ref{maxdeg_J'_omega<0}, 
we apply the transformation with respect to normalization. 
Recall from the Introduction that $J_{\bigcirc, n+1}(q) = [n+1] = (-1)^{n}\langle n \rangle$ (where $[n+1]$ and $\langle n \rangle$ are defined in Equation~(\ref{def:<n>})). 
Then, in general, we have 
\begin{equation}
\label{KT-G-relation.formula}
\langle n \rangle J'_{K, n}(q) 
= \langle n \rangle \frac{J_{K, n+1}(q)}{J_{\bigcirc, n+1}(q)}
= \langle n \rangle  \frac{J_{K, n+1}(q)}{(-1)^n \langle n \rangle}
= (-1)^n J_{K, n+1}(q),\ J'_{K, 0}(q) =1,\ J_{K, 1}(q) =1. 
\end{equation}

This implies: 
\[
d_{+}\left[ \frac {q^{(n+1)/2}-q^{-(n+1)n/2}}{q^{1/2}-q^{-1/2}}\right] d_{+} [J'_{K, n}(q)] = 
d_{+} [ J_{K, n+1}(q)]. 
\]
Since $\displaystyle d_{+}\left[ \frac {q^{(n+1)/2}-q^{-(n+1)n/2}}{q^{1/2}-q^{-1/2}}\right] = \frac{n}{2}$, 
we have
\[\delta'_K(n) = \delta_K(n+1) - \frac{1}{2}n\]
and 
\[\delta_{W_{\omega}^{\tau}(K)}(n)=\delta'_{W_{\omega}^{\tau}(K)}(n-1)+\frac 1 2 n-\frac 1 2. 
\]

Note also that 
\begin{eqnarray*}
 & & \alpha_0=a_1, \ \beta_0=2a_1+b_1-\frac 1 2, \ \gamma_0=a_1+b_1+c_1.
\end{eqnarray*}
\end{proof}

\subsection{Computations of colored Jones polynomials of $W_{\omega}^{\tau}(K)$}
\label{computation_Jones}
 In this subsection we will compute the normalized colored Jones polynomial $J'_{W_{\omega}^{\tau}(K), n}(q)$ instead of $J_{W_{\omega}^{\tau}(K), n}(q)$. 
We begin by recalling the following functions with respect to $q$ for  non-negative integers $s,t,u$. 
See \cite{MV}.

\begin{eqnarray}
\label{def:<n>}
&& \langle s \rangle  := (-1)^s [s+1], \quad [s]= \frac{q^{s/2} - q^{-s/2}}{q^{1/2} - q^{-1/2}},
\quad [s]! = \prod_{t=1}^s[t] 
\end{eqnarray}

\begin{eqnarray}
&& \langle s,t,u \rangle := (-1)^{i+j+k} 
\frac {[i+j+k+1]![i]![j]![k]!}{[s]![t]![u]!}, 
\end{eqnarray}
where $i=\frac {t+u-s}{2}$, $j=\frac {u+s-t}{2}$, and $k=\frac {s+t-u}{2}$,
\begin{eqnarray}
&& \delta(u;s,t):=(-1)^{\frac {s+t+u} 2} q^{-\frac{1}{8}(u^2-s^2-t^2+2u-2s-2t)},
\end{eqnarray}
and

\begin{equation}\label{tetsymb}
\left\langle \begin{array}{ccc} A & B & E \\ D & C & F \end{array}\right\rangle
=\frac {\prod_{i=1}^3 \prod_{j=1}^4 [b_i-a_j]!}
  {[A]![B]![C]![D]![E]![F]!}
  \sum_{\max\{a_j\} \le s \le \min \{b_i \}}
  \frac {(-1)^s [s+1]!} {\prod_{i=1}^3 [b_i-s]!\prod_{j=1}^4 [s-a_j]!}, 
\end{equation}
where  
$a_1=\frac {A+B+E} 2$, $a_2=\frac {B+D+F} 2$, $a_3=\frac {C+D+E} 2$, 
$a_4=\frac {A+C+F} 2$, $\Sigma=A+B+C+D+E+F$, 
$b_1=\frac {\Sigma-A-D} 2$, $b_2=\frac {\Sigma-E-F} 2$, and  
$b_3=\frac {\Sigma-B-C} 2$. 

We will also use the following equalities introduced by Masbaum and Vogel \cite{MV}.  

\begin{equation}
\label{pall}
\bpc{(82,25)
	\put(-10,38){\pcr{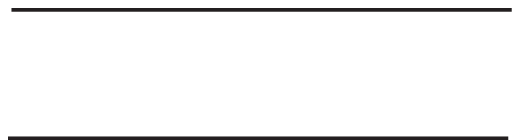}{0.4}}
	\put(42,16){\nos $s$}
	\put(42,-17){\nos $t$} } 
=\sum_u \frac {\langle u \rangle}{\langle s,t,u \rangle} 
\bpc{(82,25)
	\put(0,-1){\pcr{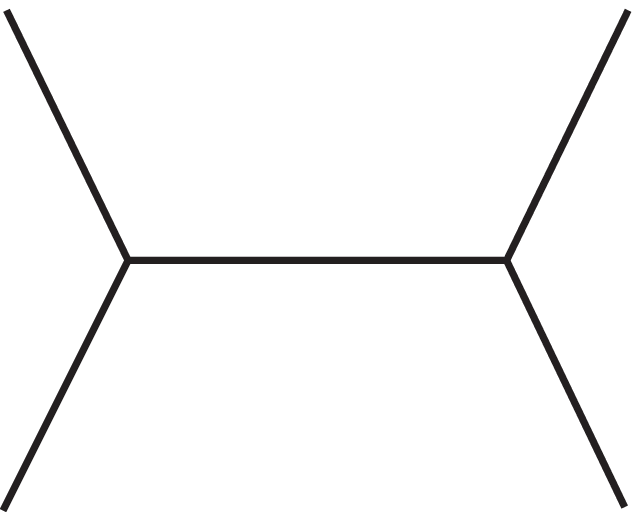}{0.4}}
	\put(11,29){\nos $s$}
	\put(11,-26){\nos $t$} 
	\put(38,10){\nos $u$}
	\put(65,28){\nos $s$}
	\put(64,-26){\nos $t$}} .
\end{equation}

\vskip 5truemm
\noindent
Here the sum is over those colors $u$ such that the triple $(s,t,u)$ satisfies 
$s+t+u \equiv 0 \pmod 2$ and $|s-t| \le u \le s+t$. 

\begin{equation}
\label{twist-tri}
\hskip 5truemm 
\bpc{(62,25)
	\put(3,-2){\pcr{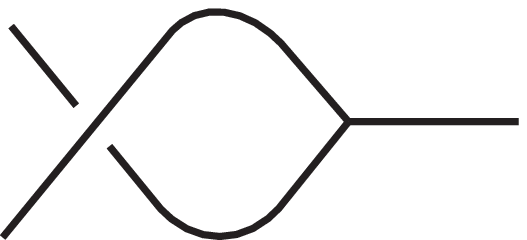}{0.4}}
	\put(4,20){\nos $s$}  
	\put(4,-22){\nos $t$}
	\put(55,10){\nos $u$} } \quad 
=\delta (u; s, t)^{-1} 
\bpc{(62,25)
	\put(0,-1){\pcr{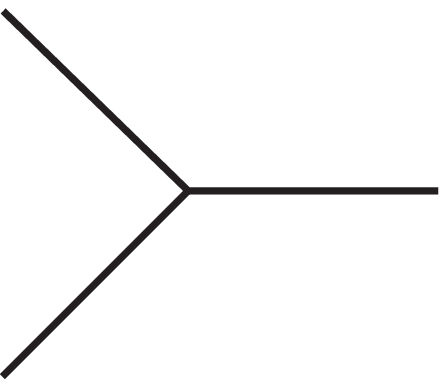}{0.4}}
	\put(6,28){\nos $s$}
	\put(6,-28){\nos $t$} 
	\put(45,10){\nos $u$}}, \quad
\bpc{(52,25)
	\put(3,-2){\pcr{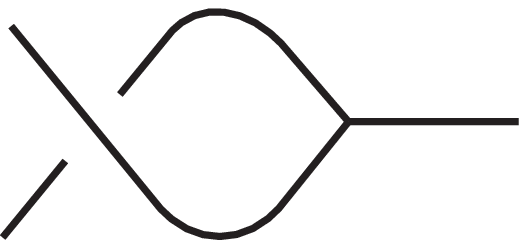}{0.4}}
	\put(5,20){\nos $s$}  
	\put(5,-22){\nos $t$}
	\put(55,10){\nos $u$} } \quad \quad
=\delta (u; s, t) 
\bpc{(62,25)
	\put(0,-1){\pcr{tri}{0.4}}
	\put(6,28){\nos $s$}
	\put(6,-28){\nos $t$} 
	\put(45,10){\nos $u$}}
\end{equation}  

\vskip 5truemm

\begin{equation}
\label{curl}
\hskip 5truemm 
\bpc{(45,25)
	\put(3,-2){\pcr{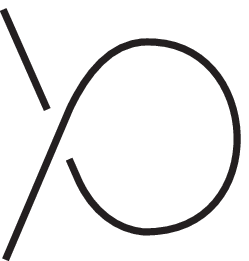}{0.5}}
	\put(30,20){\nos $s$}} 
=\delta (0; s, s)^{-1} 
\bpc{(22,25)
	\put(0,-1){\pcr{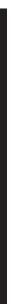}{0.5}}
	\put(10,3){\nos $s$}}, \quad \quad 
\bpc{(45,25)
	\put(3,-2){\pcr{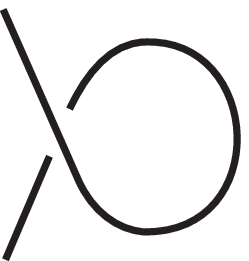}{0.5}}
	\put(30,20){\nos $s$}} 
=\delta (0; s, s) 
\bpc{(82,25)
	\put(0,-1){\pcr{1-string}{0.5}}
	\put(10,3){\nos $s$}} 
\end{equation}

\vskip 0.5truecm

\begin{equation}\label{tri-open}
\bpc{(52,25)
	\put(5,-2){\pcr{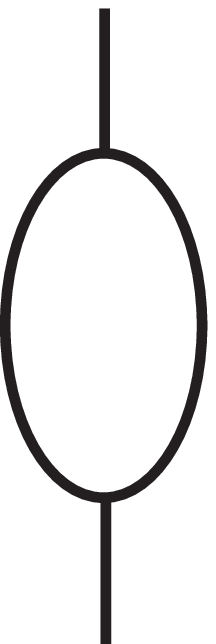}{0.3}}
	\put(22,22){\nos $u$}
	\put(3,0){\nos $s$}
	\put(32,0){\nos $t$}
	\put(22,-22){\nos $u$}} 
=\frac {\langle s,t,u \rangle} {\langle u \rangle} 
\bpc{(82,25)
	\put(0,-1){\pcr{1-string}{0.5}}
	\put(10,3){\nos $u$}} 
\end{equation}

\vskip 1truecm

\begin{equation}\label{tetra-s}
\bpc{(42,25)
	\put(-20,15){\pcr{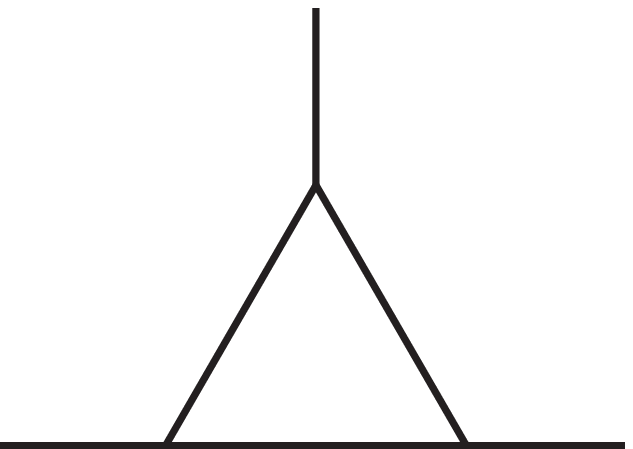}{0.3}}
	\put(-5,10){\nos $B$}
	\put(-10,-9){\nos $F$} 
	\put(15,26){\nos $A$}
	\put(23,10){\nos $E$}
	\put(10,-9){\nos $D$} 
	\put(30,-9){\nos $C$}} 
=\frac {\left\langle \begin{array}{ccc} A & B & E \\ D & C & F \end{array}\right\rangle}
{\langle A,F,C \rangle} \quad
\bpc{(42,25)
	\put(-10,12){\pcr{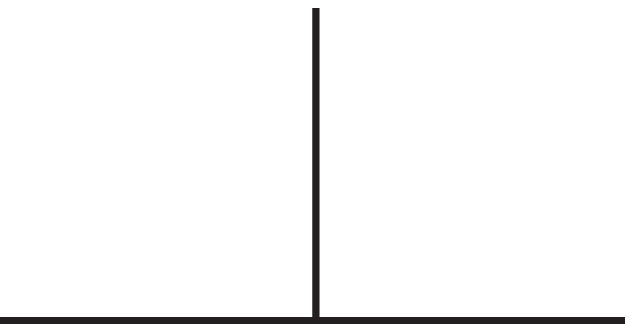}{0.3}}
	\put(5,-7){\nos $F$} 
	\put(25,13){\nos $A$}
	\put(35,-7){\nos $C$} } 
\end{equation}

\medskip

In the following we will use the following symbols. 

\begin{figure}[!ht]
\includegraphics[width=0.5\linewidth]{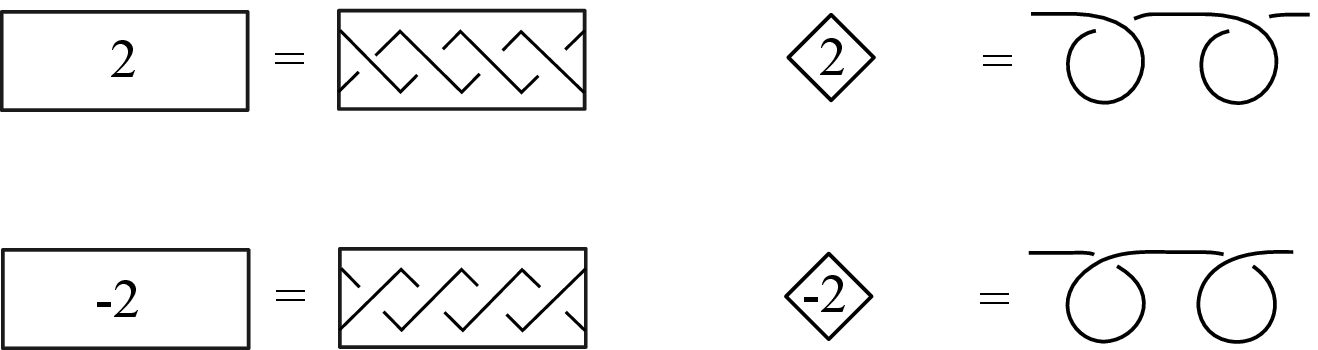}
\caption{}
\label{twist_curl}
\end{figure}

\begin{lemma}
\label{lem:twistreduction}	
\begin{multline*}
		\pcr{}{0.4} 
		=  \pcr{}{0.4} 
		=  \sum_{i=0}^n \frac{\langle 2i \rangle}{\langle n,n,2i \rangle} \pcr{}{0.4}\\
		= \sum_{i=0}^n \frac{\langle 2i \rangle}{\langle n,n,2i \rangle} q^{-\alpha i (i+1)} \pcr{}{0.4}
\end{multline*}

\end{lemma}
\begin{proof}
The first equality is due to the framed isotopy shown here for each handedness:
\[\pcr{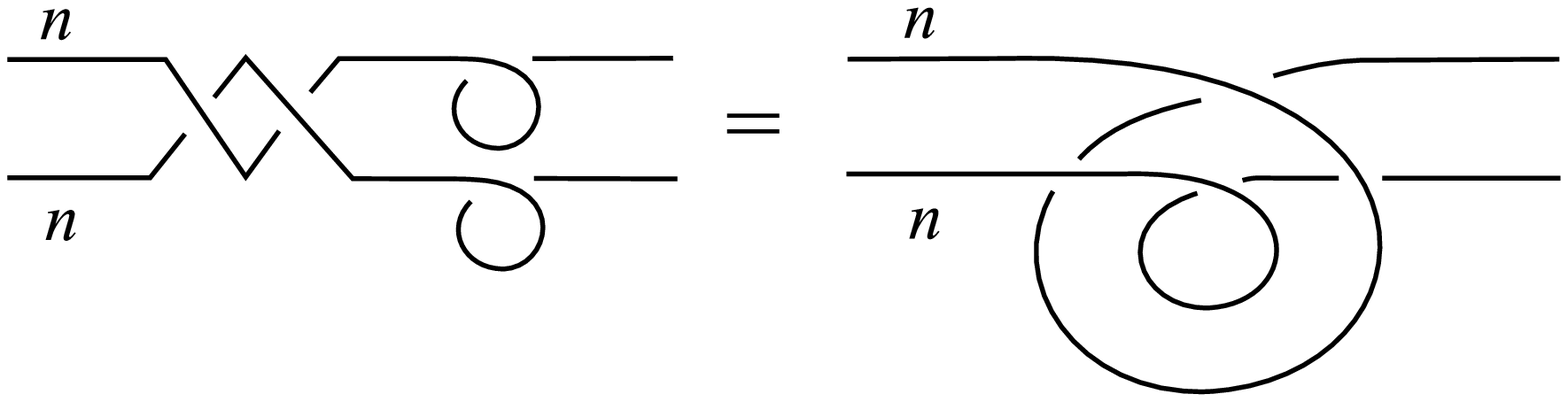}{0.35}\]
Next we exploit the formula (\ref{pall}) and then formula (\ref{curl}).
\end{proof}

\medskip

\begin{lemma}\label{lem:collapse}
\[
		\hskip 1.75truecm 
		\bpc{(30,30)
			\put(0,0){\pcr{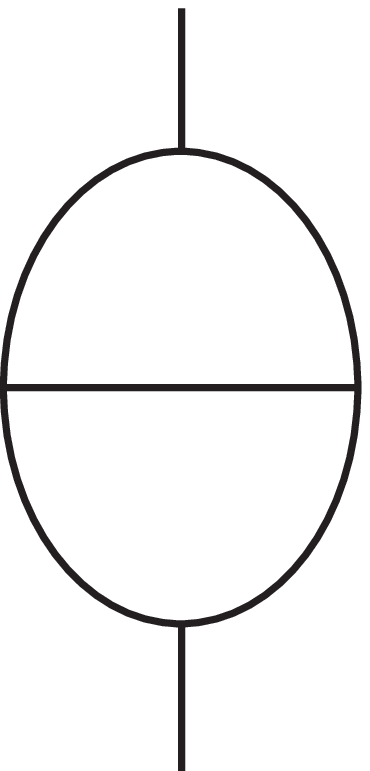}{0.4}}
			\put(12,37){\nos $2k$}
			\put(11,15){\nos $n$}
			\put(11,-12){\nos $n$}
			\put(18,9){ $2j$}
			\put(35,15){\nos $n$}
			\put(35,-12){\nos $n$}
			\put(12,-34){\nos $2k$}
		}   
	\quad \quad = \quad 
		\frac {\left\langle \begin{array}{ccc} n & n & 2j \\ n & n & 2k \end{array}\right\rangle}
		{\langle n,2k,n \rangle} \quad 
		\bpc{(52,25)
			\put(5,-2){\pcr{tri-open}{0.3}}
			\put(22,22){\nos $2k$}
			\put(2,0){\nos $n$}
			\put(30,0){\nos $n$}
			\put(22,-22){\nos $2k$}}
	= \quad 
		\frac {\left\langle \begin{array}{ccc} n & n & 2j \\ n & n & 2k \end{array}\right\rangle}
		{\langle 2k \rangle} 
		\bpc{(82,25)
			\put(0,-1){\pcr{1-string}{0.5}}
			\put(10,3){\nos $2k$}} 
\]
\end{lemma}
\vskip 1truecm
\begin{proof}
	Apply formula (\ref{tetra-s}) and then (\ref{tri-open}).
\end{proof}

\begin{lemma}
\label{graphicalcoloredjones}
For a $0$ framed diagram of any knot $K$:
 \[ 
 \langle n \rangle J'_{K,n}(q) = 
 \bpc{(30,30)
  \put(0,0){\pcr{}{0.4}}
  \put(12,0){ $K$}
   \put(9,37){\nos $n$}
     }
     \]
\end{lemma}
\vskip 1truecm
\begin{proof}
It follows from \cite[Section~5]{Mas} that the right hand side describes 
$(-1)^{n}J_{K, n+1}(q)$.  
On the other hand, 
(\ref{KT-G-relation.formula}) shows that 
\[
\langle n \rangle  J'_{K, n}(q) 
= (-1)^n J_{K, n+1}(q).
\]
Thus we obtain the desired equality.
\end{proof}

\begin{proposition}  
\label{prop:equationsplitting} 
\[ J'_{W_{\omega}^{\tau}(K), n}(q)  = \frac 1 {\langle n \rangle}\sum_{j,k=0}^n  
	\frac{\langle 2j \rangle}{\langle n,n,2j \rangle}\frac{\langle 2k \rangle}{\langle n,n,2k \rangle} \left\langle \begin{array}{ccc} n & n & 2j \\ n & n & 2k \end{array}\right\rangle
	q^{-\omega j(j+1) -\tau k(k+1)} J'_{K,2k}(q).
\] 
\end{proposition}

\begin{proof} 

We will compute $J'_{W_{\omega}^{\tau}(K), n}(q)$ using the graphical calculus of Masbaum and Vogel \cite{MV,Mas} following the method of Tanaka \cite{Tan}. 
To this end 
we need a diagram of $W_{\omega}^{\tau}(K)$ whose blackboard framing is $0$. 
Note that since two strands in the $\omega$--twist region and the $\tau$--twist region run in opposite directions, 
to obtain a correct $0$--framing of $W_{\omega}^{\tau}(K)$ we need to add some curls indicated in Figure~\ref{fig:correct_framing}.  
In Figure~\ref{fig:Wwtwrithe0}, 
we use $D(K)$ to mean a double of $K$, 
i.e.\ two parallel copies of $K$ whose blackboard framing is $0$. 

\begin{figure}[!ht]
\includegraphics[width=0.63\linewidth]{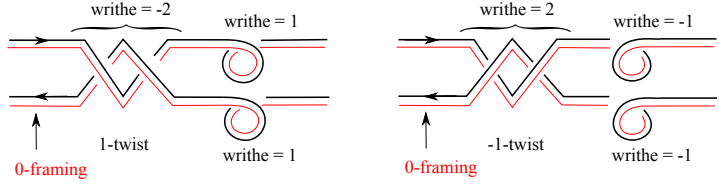}
\caption{Addition of curls to get a $0$--framing to twist regions}
\label{fig:correct_framing}
\end{figure}

\begin{figure}[!ht]
\includegraphics[width=0.28\linewidth]{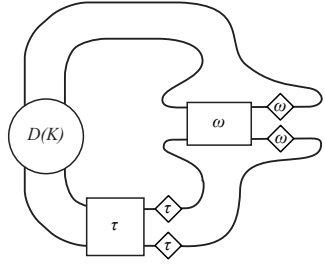}
\caption{A diagram of $W_{\omega}^{\tau}(K)$ with trivial writhe}
\label{fig:Wwtwrithe0}
\end{figure}

Using the above formulas, we compute $J'_{W_{\omega}^{\tau}(K), n}(q)$ graphically in the manner of \cite{Mas} and \cite{Tan}.  
As shown below, 
we begin by expressing $\langle n \rangle J'_{W_{\omega}^{\tau}(K), n}(q)$ diagrammatically with Lemma~\ref{graphicalcoloredjones}.  
Then we apply Lemma~\ref{lem:twistreduction} twice, 
once for each of the $\tau$ and $\omega$ twist regions.  
Next we apply Lemma~\ref{lem:collapse}.  
Finally, we again apply Lemma~\ref{graphicalcoloredjones} for the diagrammatic expression of $\langle 2k \rangle J'_{K,2k}(q)$.

\begin{align*}
\hskip -1truecm 
\langle n \rangle J'_{W_{\omega}^{\tau}(K), n}(q) 
&= \pcr{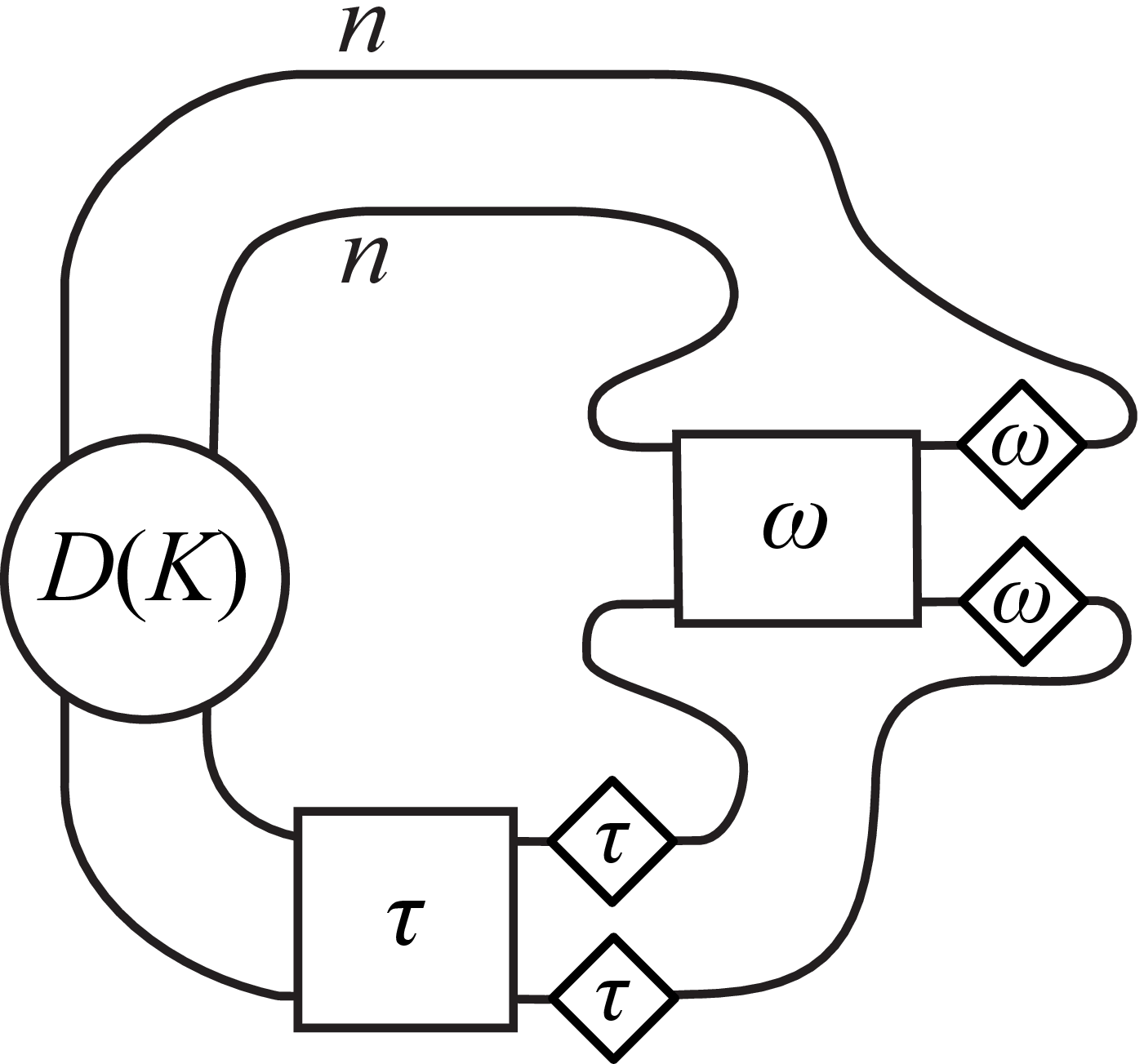}{0.25}\\
&= \displaystyle
\sum_{j,k=0}^n 
\frac{\langle 2j \rangle}{\langle n,n,2j \rangle}
\frac{\langle 2k\rangle}{\langle n,n,2k\rangle}
q^{-\omega j (j+1)-\tau k (k+1)} \pcr{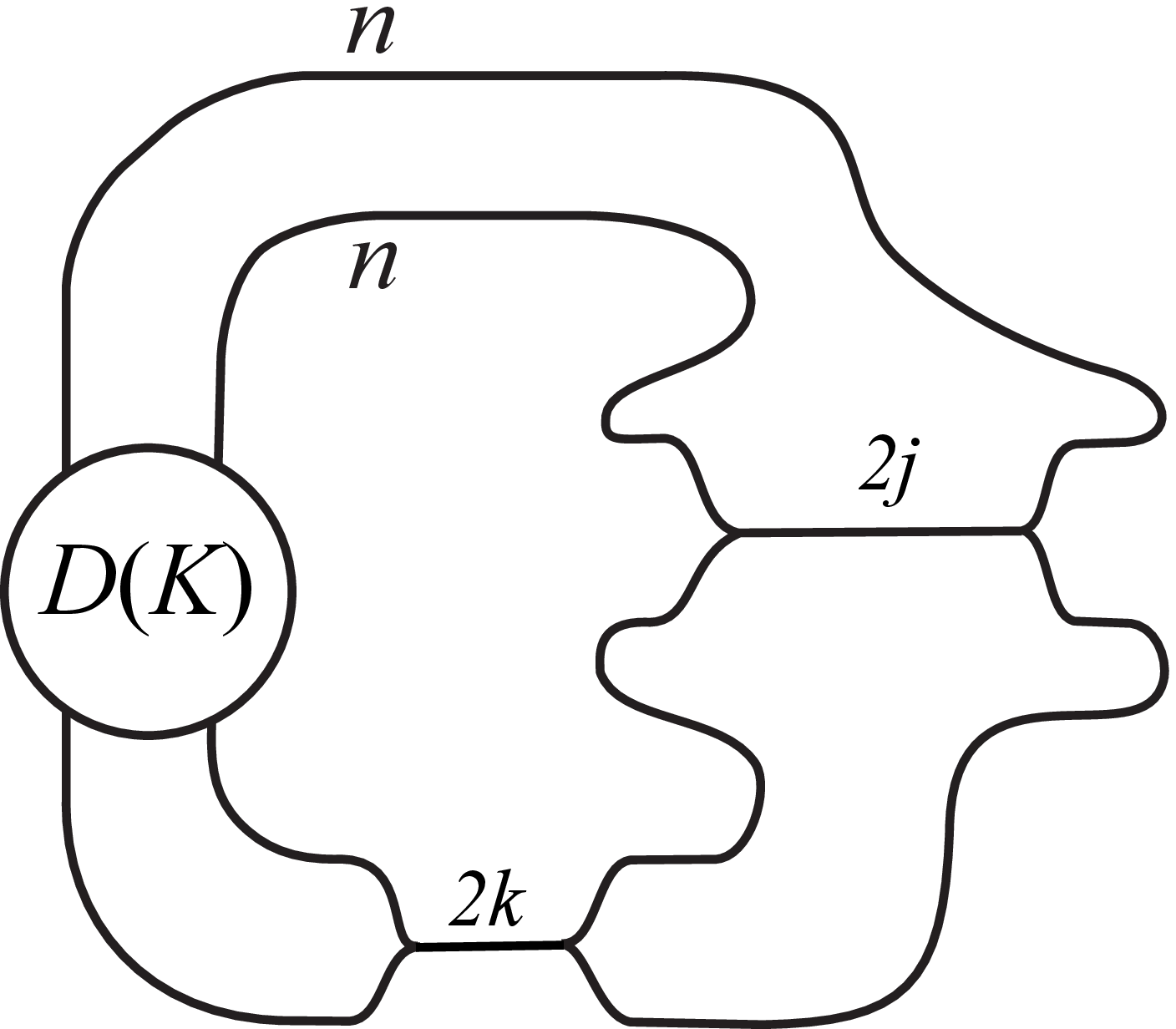}{0.25}
&\\
&\\
& = \displaystyle \sum_{j,k=0}^n 
\frac{\langle 2j \rangle}{\langle n,n,2j \rangle}
\frac{\langle 2k\rangle}{\langle n,n,2k\rangle}
 q^{-\omega j (j+1)-\tau k (k+1)}\pcr{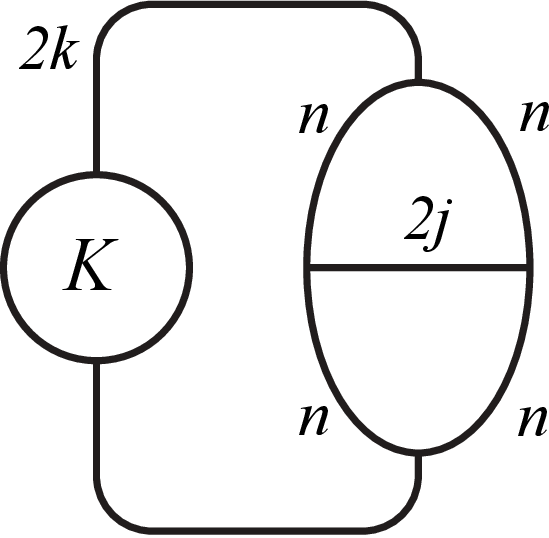}{0.5} 
& \\
& \\
&= \displaystyle
\sum_{j,k=0}^n 
\frac{\langle 2j \rangle}{\langle n,n,2j \rangle}
\frac{\langle 2k\rangle}{\langle n,n,2k\rangle}
 q^{-\omega j (j+1)-\tau k (k+1)} 		
\frac {\left\langle \begin{array}{ccc} n & n & 2j \\ n & n & 2k \end{array}\right\rangle}{<2k>}  \pcr{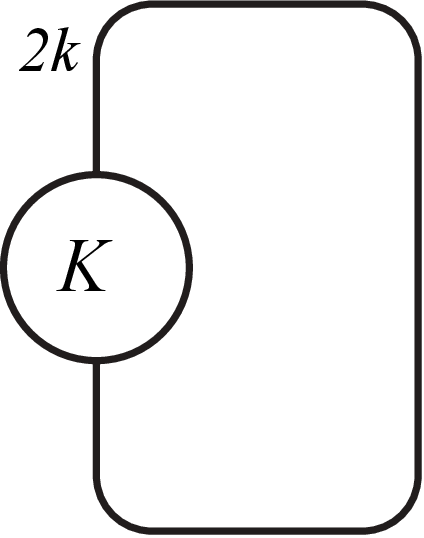}{0.5} 
 & \\
 & \\
 &= \displaystyle
\sum_{j,k=0}^n 
\frac{\langle 2j \rangle}{\langle n,n,2j \rangle}
\frac{\langle 2k\rangle}{\langle n,n,2k\rangle}
 q^{-\omega j (j+1)-\tau k (k+1)} 		
\left\langle \begin{array}{ccc} n & n & 2j \\ n & n & 2k \end{array}\right\rangle
 J'_{K, 2k}(q)
\end{align*}

Therefore we have: 
\begin{eqnarray*}
& & J'_{W_{\omega}^{\tau}(K), n}(q) \\
& & = \frac 1 {<n>}\sum_{j,k=0}^n  
\frac{\langle 2j \rangle}{\langle n,n,2j \rangle}
\frac{\langle 2k\rangle}{\langle n,n,2k\rangle}
\left\langle \begin{array}{ccc} n & n & 2j \\ n & n & 2k \end{array}\right\rangle q^{-\omega j(j+1) -\tau k(k+1)} J'_{K,2k}(q).
\end{eqnarray*} 

\end{proof}

\begin{remark}
Other formulas of the colored Jones polynomial  of the twisted Whitehead double of a knot $K$,
are given  by Tanaka \cite{Tan} and Zheng \cite{Zheng}.  
\end{remark}

\subsection{Computations of the maximum degrees}
\label{computation_degrees}

The maximum degree of a polynomial  $f(q)\in \Q [q^{\pm \frac{1}{4}}]$ is denoted $d_+[f(q)]$. 
For a rational function $f(q) = \frac{f_1(q)}{f_2(q)}$ with $f_1(q), f_2(q)\in \Q [q^{\pm \frac{1}{4}}]$ and $f_2(q)\ne 0$, 
we extend the maximum  degree of $f(q)$ as 
$d_+[f_1(q)]-d_+[f_2(q)]$. 

Propositions~\ref{maxdeg_J'_omega>0}, \ref{maxdeg_J'_omega<0} and \ref{maxdeg_J'_signcond} 
give the maximum degree of 
$J'_{K, n}(q) \in \mathbb{Z}[q^{\pm 1}]$ under various hypotheses.   
For convenience we recall the maximum degree of the functions which appear in the expression 
of $J'_{W_{\omega}^{\tau}(K), n}(q)$ in Proposition~\ref{prop:equationsplitting}. 

By definition \ref{def:<n>} we have: 

\begin{lemma}
\label{<n>}
$d_{+}[\langle n \rangle] = \dfrac{1}{2}n$.
\end{lemma}

\begin{lemma}[\cite{GV}] 
\label{deglem1}
The maximum degree of $\langle s,t,u \rangle$ are  
 given by 
\[d_{+}[\langle s,t,u \rangle]= \frac {s+t+u} 4.\]
\end{lemma}

\begin{lemma}[\cite{GV}] 
\label{deglem2} 
The maximum degree of 
$\dis{
\left\langle \begin{array}{ccc} A & B & E \\ D & C & F \end{array}\right\rangle
}$ are given by 
\begin{eqnarray*}
&& d_{+} \left[\left\langle \begin{array}{ccc} A & B & E \\ D & C & F \end{array}\right\rangle \right]\\
&&= \frac 1 2[-\Sigma^2-\frac 1 2 (A^2+B^2+C^2+D^2+E^2+F^2-\Sigma )
+\frac 3 2 \sum_{i=1}^3 b_i(b_i-1) +\sum_{j=1}^4 a_j(a_j+1)\\
&& \quad -3M^2+M(1+2\Sigma)], 
\end{eqnarray*}
where $\Sigma, a_j, b_i$ are as in (\ref{tetsymb}) and  $M=\min b_i$. 
\end{lemma}

To ease our computations, 
we define
\begin{equation}
\label{fsummand}
f(j,k;q)  =  
\frac{\langle 2j \rangle}{\langle n,n,2j \rangle}
\frac{\langle 2k\rangle}{\langle n,n,2k\rangle}
\left\langle \begin{array}{ccc} n & n & 2j \\ n & n & 2k \end{array}\right\rangle 
q^{-\omega j(j+1) -\tau k(k+1)} J'_{K,2k}(q).
\end{equation}
Then by 
Proposition~\ref{prop:equationsplitting} we have 
\begin{equation}
\label{normalizedsum}
\langle n \rangle J'_{W_{\omega}^{\tau}(K), n}(q) =\sum_{j,k=0}^n f(j,k;q)
\end{equation}
so that 
\begin{equation}
\label{normalizeddegreesum}
d_{+}[J'_{W_{\omega}^{\tau}(K), n}(q)] 
= d_{+}[\sum_{j,k=0}^n f(j,k;q)] - d_{+}[\langle n \rangle] 
= d_{+}[\sum_{j,k=0}^n f(j,k;q)] - \frac{n}{2}.
\end{equation}
Hence our computations of $d_{+}[J'_{W_{\omega}^{\tau}(K), n}(q)]$ reduce to an understanding the extrema of $d_{+}[f(j,k;q)]$ for $0\leq j,k \leq n$.
 
To that end we first record the following computations of degrees.  
From Lemma \ref{deglem1}, 
we have that 
\begin{equation}
\label{deg2j/nn2j}
d_{+}\left[ \frac {\langle 2j \rangle}{\langle n,n,2j \rangle} \right]
=-\frac{n}{2}+\frac{j}{2}.
\end{equation}
From  Lemma \ref{deglem2}, we have that 
\begin{equation}
\label{deg6j}
d_{+}\left[ \left\langle \begin{array}{ccc} n & n & 2j \\ n & n & 2k \end{array}\right\rangle \right]
= \left\{\begin{array}{ll}
\frac 1 2 (j+k+n) & (j+k\le n), \\
\frac 1 2 (-j^2-2jk-k^2+2n+2jn+2kn-n^2) & (j+k\ge n).
\end{array}
\right.
\end{equation}

\begin{proposition}[$J'$, $d_+ = \delta$, $\omega>0$]
\label{maxdeg_J'_omega>0} 
Let $K$ be a knot in $S^3$ and $d_+[J'_{K,n}(q)]$ is the quadratic quasi-polynomial 
$\delta'_K(n)=\alpha(n) n^2 +\beta(n) n+\gamma(n)$ for all $n \ge 0$.
We put $\alpha_0:=\alpha (0)$, $\beta_0:=\beta (0)$, and $\gamma_0:=\gamma (0)$. 
We assume that the period of $d_+[J'_{K,n}(q)]=\delta'_K(n)$ is less than or equal to $2$ and that 
$-2\alpha_0+\beta_0+\frac 1 2 \le 0$. 
Assume further that if $-2\alpha_0+\beta_0+\frac{1}{2} = 0$, 
then $\tau \ne 4 \alpha_0$. 
Then, for suitably large $n$, 
the maximum degree of the normalized colored Jones polynomial 
of its $\tau$--twisted  $\omega$--generalized  Whitehead double with $\omega>0$ is given by 
\begin{equation}
\label{tgWmax_omega>0}
\delta'_{W_{\omega}^{\tau}(K)}(n)=
\begin{cases} 
(4\alpha_0- \tau) n^2 +(2\beta_0-\tau) n+\gamma_0 & (\alpha_0> \frac{\tau}{4}) \\
             -n +\gamma_0  &  (\alpha_0 < \frac{\tau}{4}) \\
             -n+ \gamma_0& (\alpha_0 = \frac{\tau}{4}, -2\alpha_0+\beta_0+\frac{1}{2} \ne 0).
\end{cases} 
\end{equation}
In fact, if either $\alpha_0 < \frac{\tau}{4}$ or $\alpha_0 = \frac{\tau}{4}$ with $-2\alpha_0+\beta_0+\frac{1}{2} \ne 0$, 
then $d_+[J'_{W_{\omega}^{\tau}(K),n}(q)] =  \delta'_{W_{\omega}^{\tau}(K)}(n)$ for all $n \geq 0$.
\end{proposition}

\begin{proof}
In light of \eqref{normalizedsum}, 
to determine the maximum degree of $J'_{W_{\omega}^{\tau}(K), n}(q)$, 
we need to understand  the maximum degrees of the functions $f(j,k;q)$ for $0\leq j,k \leq n$. 

Since the period of $d_+[J'_{K,n}(q)]=\delta'_K(n)$ is less than or equal to $2$, 
it follows that 
$\alpha(2k)=\alpha_0$, $\beta(2k)=\beta_0$, and $\gamma (2k)=\gamma_0$ and so 
\begin{equation}
\label{pdegK_omega>0}
d_+[J'_{K,2k}(q)]=\alpha(2k)(2k)^2+\beta(2k)2k+\gamma(2k)
=4\alpha_0 k^2+2\beta_0 k+\gamma_0
\end{equation}
for all $k\geq 0$.

Due to (\ref{deg6j}) the argument splits into two cases.

\medskip

\noindent 
\textbf{Case $1$.\ $j+k \ge n$.}

Applying the equalities (\ref{deg2j/nn2j}) and (\ref{deg6j}) to (\ref{fsummand}), 
we see that 
\begin{align*}
d_+[f(j,k;q)]
&= (-\frac n 2+\frac j 2) + (-\frac n 2 + \frac k 2) + \frac 1 2 (-j^2-2jk-k^2+2n+2jn+2kn-n^2) \\
& \qquad\qquad\qquad\qquad -\omega j(j+1) -\tau k(k+1) + d_+[J'_{K,2k}(q)]\\
&= -\frac{2 \omega+1}{2} j^2+(n-k- \omega+\frac 1 2)j \\
&\qquad\qquad\qquad\qquad +(- \tau -\frac 1 2) k^2+( \frac 1 2+n- \tau) k-\frac {n^2} 2+d_+[J'_{K, 2k}(q)] \\
&= 
-\left(\frac{2 \omega+1}{2} \right)\left(j-\frac{n-k- \omega+\frac{1}{2}}{2 \omega +1}\right)^2 + \frac{(n-k- \omega+\frac{1}{2})^2}{2(2\omega+1)}\\
& \qquad\qquad\qquad\qquad
+(- \tau -\frac 1 2) k^2+( \frac 1 2+n- \tau) k-\frac {n^2} 2+d_+[J'_{K, 2k}(q)]. 
\end{align*}
Since $\frac{n-k- \omega+\frac{1}{2}}{2 \omega +1}<n-k$ and $j\geq n-k$, this is maximized at $j=n-k$ for a fixed $k$. 
Therefore this case is included in the next case. 

\medskip

\noindent 
\textbf{Case $2$.\ $j+k \le n$.}

Applying the equalities (\ref{deg2j/nn2j}) and (\ref{deg6j}) to (\ref{fsummand}), 
we see that 
\begin{align*}
d_+[f(j,k;q)]
&= (-\frac n 2+\frac j 2) + (-\frac n 2 + \frac k 2) + \frac 1 2 (j+k+n) \\
& \qquad\qquad\qquad -\omega j(j+1) -\tau k(k+1) + d_+[J'_{K,2k}(q)]\\
&= - \omega j^2-( \omega-1) j -\tau k^2 - (\tau-1) k -\frac{n}{2}+d_+[J'_{K,2k}(q)]\\
& = -\omega\left(j + \frac{\omega - 1}{2\omega}\right)^2 + \omega\left(\frac{\omega -1}{2\omega}\right)^2 - \tau k^2 - (\tau -1)k - \frac{n}{2}+d_+[J'_{K,2k}(q)].
\end{align*}
Since $\omega > 0$ so that $\frac{\omega-1}{2\omega} \ge 0 $, 
this is maximized uniquely at $j=0$ for a fixed $k$. Thus, 
\begin{align*}
\max_{0\le j \le n-k}d_+[f(j,k;q)] 
& = d_+[f(0,k;q)] \\
&=  -\tau k^2-(\tau-1) k -\frac{n}{2} + d_+[J'_{K,2k}(q)] \\
& =  -\tau k^2-(\tau-1) k -\frac{n}{2} + 4\alpha_0 k^2 + 2\beta_0 k + \gamma_0 \\
& = (4\alpha_0- \tau)k^2 + (2\beta_0 - \tau +1)k-\frac{n}{2} + \gamma_0
\end{align*}
since we are assuming $d_+[J'_{K,2k}(q)] = \delta'_K(2k)$. 

\medskip

We now consider the three cases of when $\alpha_0 > \frac{\tau}{4}$, $\alpha_0 < \frac{\tau}{4}$, or $\alpha_0 = \frac{\tau}{4}$.
First note that when $\alpha_0 \ne \frac{\tau}{4}$, 
we may write
\begin{equation}
\label{eqn:d+bigk_omega>0}
d_+[f(0,k;q)] =(4\alpha_0- \tau) \left(k+\frac {2\beta_0- \tau +1}{2(4\alpha_0-\tau)} \right)^2
- \frac {(2\beta_0- \tau +1 )^2}{4(4\alpha_0- \tau)}-\frac n 2+\gamma_0
\end{equation}
for $n\geq k\geq 0$. 
 Therefore, for $k >  -\frac {2\beta_0- \tau +1}{2(4\alpha_0-\tau)}$, 
 we see that $d_+[f(0,k;q)]$ is monotonically increasing or decreasing depending on whether $4\alpha_0- \tau$ is positive or negative respectively.

\smallskip

\noindent 
\textbf{Case $2.1$.\ $\alpha_0 > \frac{\tau}{4}$.}

Note that $-\frac {2\beta_0-\tau+1}{2(4\alpha_0- \tau)} < \frac n 2 $ 
for sufficiently large $n\geq 0$. 
Hence equation (\ref{eqn:d+bigk_omega>0}) shows that $d_+[f(0,k;q)]$ is monotonically increasing with respect to $k$ for $k  > -\frac {2\beta_0-\tau+1}{2(4\alpha_0- \tau)} $ when $n$ is sufficiently large.  
Therefore, for sufficiently large $n$, $\displaystyle \max_{0 \le k \le n} d_+[f(0,k;q)]$ is uniquely realized at $k = n$. 
Hence 
\begin{align}
\label{eqn:a0big_omega>0}
d_+[ \sum_{j,k=0}^n f(j,k;q)]
&=d_+[ \sum_{k=0}^n f(0,k;q)]\\
&=d_+[f(0,n;q)] \nonumber \\
&=(4\alpha_0- \tau) n^2+(2\beta_0- \tau + \frac 1 2)n+\gamma_0 \nonumber
\end{align}
for sufficiently large $n$. 

\smallskip

\noindent 
\textbf{Case $2.2$.\ $\alpha_0 < \frac{\tau}{4}$.}

Since 
$0 \ge -2\alpha_0+\beta_0+\frac{1}{2}
=-2(\alpha_0-\frac{\tau}{4})+\beta_0-\frac{\tau}{2}+\frac{1}{2}$ by assumption,  
then we have $2\beta_0 - \tau + 1 < 0$. 
Therefore, 
if $\alpha_0 <\frac{\tau}{4}$, 
so that $\frac {2\beta_0- \tau+1}{2(4\alpha_0- \tau)} > 0$, 
(\ref{eqn:d+bigk_omega>0}) shows that $d_+[f(0,k;q)]$ is monotonically decreasing with respect to $k$ for $k  \ge 0$.  
Thus $ \displaystyle \max_{0 \le k \le n} d_+[f(0,k;q)]$ is uniquely realized at $k = 0$. 
Hence 
\begin{align}
\label{eqn:a0small_omega>0}
d_+[ \sum_{j,k=0}^n f(j,k;q)]
&=d_+[ \sum_{k=0}^n f(0,k;q)]\\
&= d_+[f(0,0;q)] \nonumber \\
&= -\frac{n}{2} + \gamma_0. \nonumber
\end{align}

\smallskip

\noindent 
\textbf{Case $2.3$.\ $\alpha_0 = \frac{\tau}{4}$.}

By the assumption that $\alpha_0 = \frac{\tau}{4}$, 
our hypotheses of this proposition imply that $-2\alpha_0 + \beta_0 +\frac{1}{2} < 0$.    
Since $\beta_0- \frac{\tau}{2} + \frac{1}{2} 
= -2(\alpha_0 - \frac{\tau}{4}) + \beta_0 - \frac{\tau}{2}+ \frac{1}{2}
= -2 \alpha_0 + \beta_0 + \frac{1}{2}$, we now have that $2\beta_0 - \tau + 1 < 0$. 
Therefore 
\[
d_+[f(0,k;q)] = (2\beta_0 - \tau +1)k-\frac{n}{2} + \gamma_0
\]
is monotonically decreasing with respect to $k$ for $0 \le k \le n$. 
Thus $\displaystyle \max_{0 \le k \le n}d_+[f(0,k;q)] $ is uniquely realized at $k=0$. 
Hence 
\begin{align}
\label{eqn:a0small2_omega>0}
d_+[ \sum_{j,k=0}^n f(j,k;q)]
&=d_+[ \sum_{k=0}^n f(0,k;q)]\\
&= d_+[f(0,0;q)] \nonumber \\
&= -\frac{n}{2} + \gamma_0. \nonumber
\end{align}

\medskip

Finally we determine the maximum degree of 
$J'_{W_{\omega}^{\tau}(K), n}(q)$ via
\[d_+[J'_{W_{\omega}^{\tau}(K), n}(q)] = d_+[\sum_{j, k=0}^n f(j,k;q)] -  \frac{n}{2}\]
using equation (\ref{normalizeddegreesum}).
By equations (\ref{eqn:a0big_omega>0}), (\ref{eqn:a0small_omega>0}), and (\ref{eqn:a0small2_omega>0}) we have 
\[
\delta'_{W_{\omega}^{\tau}(K)}(n) = 
\begin{cases}
(4\alpha_0- \tau) n^2+(2\beta_0- \tau) n + \gamma_0 &(\alpha_0 > \frac{\tau}{4}) \\
-n + \gamma_0 & (\alpha_0 < \frac{\tau}{4})\\
-n + \gamma_0 & (\alpha_0 = \frac{\tau}{4}, -2\alpha_0 + \beta_0 +\frac{1}{2} \neq 0).
\end{cases}
\]
Notice that only equation (\ref{eqn:a0big_omega>0}) requires that $n$ be suitably large.  Equations (\ref{eqn:a0small_omega>0}) and (\ref{eqn:a0small2_omega>0}) hold for all $n \geq 0$.
\end{proof}

\begin{proposition}[$J'$, $d_+ = \delta$, $\omega <0$]
\label{maxdeg_J'_omega<0} 
Let $K$ be a knot in $S^3$ and $d_+[J'_{K,n}(q)]$ is the quadratic quasi-polynomial 
$\delta'_K(n)=\alpha(n) n^2 +\beta(n) n+\gamma(n)$ for all $n \ge 0$.
We put $\alpha_0:=\alpha (0)$, $\beta_0:=\beta (0)$, and $\gamma_0:=\gamma (0)$. 
We assume that the period of $d_+[J'_{K,n}(q)]=\delta'_K(n)$ is less than or equal to $2$ and that 
$-2\alpha_0+\beta_0+\frac 1 2 \le 0$. 
Assume further that if $-2\alpha_0+\beta_0+\frac{1}{2} = 0$, 
then $\tau \ne 4 \alpha_0$. 
Then, for suitably large $n$, 
the maximum degree of the normalized colored Jones polynomial 
of its $\tau$--twisted  $\omega$--generalized Whitehead double with $\omega<0$ is given by 
\begin{equation}
\label{tgWmax_omega<0}
\delta'_{W_{\omega}^{\tau}(K)}(n)=
\begin{cases}
(4\alpha_0- \tau-\omega-\frac{1}{2}) n^2+(2\beta_0- \tau -\omega + \frac{1}{2})n+\gamma_0 &(\alpha_0 > \frac{\tau}{4}+\frac{1}{8}) \\
- \omega n^2 - \omega n  + \gamma_0 & (\alpha_0 < \frac{\tau}{4}+\frac{1}{8})\\
- \omega n^2 - 	\omega n  + \gamma_0 & (\alpha_0 = \frac{\tau}{4}+\frac{1}{8}, -2\alpha_0 + \beta_0 +\frac{1}{2} \neq 0).
\end{cases}
\end{equation}
In fact, if either $\alpha_0 < \frac{\tau}{4}+\frac{1}{8}$ or $\alpha_0 = \frac{\tau}{4}+\frac{1}{8}$ with $-2\alpha_0+\beta_0+\frac{1}{2} \ne 0$, 
then $d_+[J'_{W_{\omega}^{\tau}(K),n}(q)] =  \delta'_{W_{\omega}^{\tau}(K)}(n)$ for all $n \geq 0$.
\end{proposition}

\begin{proof}
In light of \eqref{normalizedsum}, 
to determine the maximum degree of $J'_{W_{\omega}^{\tau}(K), n}(q)$, 
we need to understand  the maximum degrees of the functions $f(j,k;q)$ for $0\leq j,k \leq n$. 

Since the period of $d_+[J'_{K,n}(q)]=\delta'_K(n)$ is less than or equal to $2$, 
it follows that 
$\alpha(2k)=\alpha_0$, $\beta(2k)=\beta_0$, and $\gamma (2k)=\gamma_0$ and so 
\begin{equation}
\label{pdegK_omega<0}
d_+[J'_{K,2k}(q)]=\alpha(2k)(2k)^2+\beta(2k)2k+\gamma(2k)
=4\alpha_0 k^2+2\beta_0 k+\gamma_0
\end{equation}
for all $k\geq 0$.

Due to (\ref{deg6j}) the argument splits into two cases.

\medskip

\noindent 
\textbf{Case $1$.\ $j+k \le n$.} 

Applying the equalities (\ref{deg2j/nn2j}) and (\ref{deg6j}) to (\ref{fsummand}), 
we see that 
\begin{align*}
d_+[f(j,k;q)]
&= (-\frac n 2+\frac j 2) + (-\frac n 2 + \frac k 2) + \frac 1 2 (j+k+n) \\
& \qquad\qquad\qquad -\omega j(j+1) -\tau k(k+1) + d_+[J'_{K,2k}(q)]\\
&= - \omega j^2-( \omega-1) j -\tau k^2 - (\tau-1) k -\frac{n}{2}+d_+[J'_{K,2k}(q)]\\
& = -\omega\left(j + \frac{\omega - 1}{2\omega}\right)^2 + \omega\left(\frac{\omega -1}{2\omega}\right)^2 - \tau k^2 - (\tau -1)k - \frac{n}{2}+d_+[J'_{K,2k}(q)].
\end{align*}
Since  $\frac{\omega-1}{2\omega} \ge 0$, 
this is maximized uniquely at $j=n-k$ for a fixed $k$. Thus, 
this case is included in the next case.

\medskip

\noindent 
\textbf{Case $2$.\ $j+k \ge n$.} 

Applying the equalities (\ref{deg2j/nn2j}) and (\ref{deg6j}) to (\ref{fsummand}), 
we see that 
\begin{align*}
d_+[f(j,k;q)]
&= (-\frac n 2+\frac j 2) + (-\frac n 2 + \frac k 2) + \frac 1 2 (-j^2-2jk-k^2+2n+2jn+2kn-n^2) \\
& \qquad\qquad\qquad\qquad -\omega j(j+1) -\tau k(k+1) + d_+[J'_{K,2k}(q)]\\
&= -\frac{2 \omega+1}{2} j^2+(n-k- \omega+\frac 1 2)j \\
&\qquad\qquad\qquad\qquad +(- \tau -\frac 1 2) k^2+( \frac 1 2+n- \tau) k-\frac {n^2} 2+d_+[J'_{K, 2k}(q)] \\
&= 
-\left(\frac{2 \omega+1}{2} \right)\left(j-\frac{n-k- \omega+\frac{1}{2}}{2 \omega +1}\right)^2 + \frac{(n-k- \omega+\frac{1}{2})^2}{2(2\omega+1)}\\
& \qquad\qquad\qquad\qquad
+(- \tau -\frac 1 2) k^2+( \frac 1 2+n- \tau) k-\frac {n^2} 2+d_+[J'_{K, 2k}(q)]. 
\end{align*}
Since $\frac{n-k- \omega+\frac{1}{2}}{2 \omega +1}<0\leq n-k$ and $j\geq n-k$, 
this is maximized at $j=n$ for a fixed $k$.  
Thus, 
\begin{align*}
\max_{0\le j \le n-k}d_+[f(j,k;q)] 
& = d_+[f(n,k;q)] \\
&= -(\tau+\frac{1}{2})k^2 - (\tau - \frac{1}{2})k - \omega n^2 - (\omega-\frac{1}{2})n+ d_+[J'_{K,2k}(q)] \\
& =   -(\tau+\frac{1}{2})k^2 - (\tau - \frac{1}{2})k - \omega n^2 - (\omega-\frac{1}{2})n+ 4\alpha_0 k^2 + 2\beta_0 k + \gamma_0 \\
& =  (4\alpha_0-\tau-\frac{1}{2})k^2 + (2\beta_0-\tau + \frac{1}{2})k - \omega n^2 - (\omega-\frac{1}{2})n  + \gamma_0
\end{align*}
since we are assuming $d_+[J'_{K,2k}(q)] = \delta'_K(2k)$.
\medskip

We now consider the three cases of when $\alpha_0 > \frac{\tau}{4} + \frac{1}{8}$, 
$\alpha_0 < \frac{\tau}{4}+ \frac{1}{8}$, or $\alpha_0 = \frac{\tau}{4}+ \frac{1}{8}$.
First note that when $\alpha_0 \ne \frac{\tau}{4}+ \frac{1}{8}$, 
we may write
\begin{equation}
\label{eqn:d+bigk_omega<0}
d_+[f(n,k;q)] =
(4\alpha_0-\tau-\frac{1}{2}) \left(k+\frac {2\beta_0-\tau + \frac{1}{2}}{2(4\alpha_0-\tau-\frac{1}{2})} \right)^2
- \frac {(2\beta_0-\tau + \frac{1}{2})^2}{2 (4\alpha_0-\tau-\frac{1}{2})}
- \omega n^2 - (\omega-\frac{1}{2})n  + \gamma_0                
\end{equation}
for $n\geq k\geq 0$. 
Therefore, for $k >  -\frac {2\beta_0-\tau + \frac{1}{2}}{2(4\alpha_0-\tau-\frac{1}{2})}$, 
we see that $d_+[f(n,k;q)]$ is monotonically increasing or decreasing depending on whether $4\alpha_0-\tau-\frac{1}{2}$ is positive or negative respectively.

\smallskip

\noindent 
\textbf{Case $2.1$.\ $\alpha_0 > \frac{\tau}{4}+ \frac{1}{8}$.}

Note that $-\frac {2\beta_0-\tau + \frac{1}{2}}{2(4\alpha_0-\tau-\frac{1}{2})} < \frac n 2 $ 
for sufficiently large $n\geq 0$. 
Hence equation (\ref{eqn:d+bigk_omega<0}) shows that $d_+[f(n,k;q)]$ is monotonically increasing with respect to $k$ for $k  > -\frac {2\beta_0-\tau + \frac{1}{2}}{2(4\alpha_0-\tau-\frac{1}{2})} $ when $n$ is sufficiently large.  
Therefore, for sufficiently large $n$, 
$\displaystyle \max_{0 \le k \le n} d_+[f(n,k;q)]$ is uniquely realized at $k = n$. 
Hence 
\begin{align}
\label{eqn:a0big_omega<0}
d_+[ \sum_{j,k=0}^n f(j,k;q)]
&=d_+[ \sum_{k=0}^n f(n,k;q)]\\
&=d_+[f(n,n;q)] \nonumber \\
&=(4\alpha_0- \tau-\omega-\frac{1}{2}) n^2+(2\beta_0- \tau -\omega + 1)n+\gamma_0 \nonumber
\end{align}
for sufficiently large $n$. 

\smallskip

\noindent 
\textbf{Case $2.2$.\ $\alpha_0 < \frac{\tau}{4}+\frac{1}{8}$.}
 
Since 
$0 \ge -2\alpha_0+\beta_0+\frac{1}{2}
=-2(\alpha_0-\frac{\tau}{4}-\frac{1}{8})+\beta_0-\frac{\tau}{2}-\frac{1}{4}+\frac{1}{2}$ by hypothesis,  
then we have $2\beta_0 - \tau + \frac{1}{2} < 0$. 
Therefore, 
since $\alpha_0 <\frac{\tau}{4}+\frac{1}{8}$ by assumption
so that $\frac {2\beta_0-\tau + \frac{1}{2}}{2(4\alpha_0-\tau-\frac{1}{2})}> 0$, 
(\ref{eqn:d+bigk_omega<0}) shows that $d_+[f(n,k;q)]$ is monotonically decreasing with respect to $k$ for $k  \ge 0$.  
Thus $ \displaystyle \max_{0 \le k \le n} d_+[f(n,k;q)]$ is uniquely realized at $k = 0$. 
Hence 
\begin{align}
\label{eqn:a0small_omega<0}
d_+[ \sum_{j,k=0}^n f(j,k;q)]
&=d_+[ \sum_{k=0}^n f(n,k;q)]\\
&= d_+[f(n,0;q)] \nonumber \\
&=  - \omega n^2 - (\omega-\frac{1}{2})n  + \gamma_0. \nonumber
\end{align}

\smallskip

\noindent 
\textbf{Case $2.3$.\ $\alpha_0 = \frac{\tau}{4}-\frac{1}{8}$.}

By the assumption that $\alpha_0 = \frac{\tau}{4}-\frac{1}{8}$, our hypotheses of this proposition imply that $-2\alpha_0 + \beta_0 +\frac{1}{2} < 0$.    
Since $\beta_0- \frac{\tau}{2} + \frac{1}{4} 
= -2(\alpha_0 - \frac{\tau}{4}-\frac{1}{8}) + \beta_0 - \frac{\tau}{2}+ \frac{1}{2}
= -2 \alpha_0 + \beta_0 + \frac{1}{2}$, we now have that $2\beta_0 - \tau + \frac{1}{2} < 0$. 
Therefore 
\[
d_+[f(n,k;q)] = (2\beta_0-\tau + \frac{1}{2})k - \omega n^2 - (\omega-\frac{1}{2})n  + \gamma_0
\]
is monotonically decreasing with respect to $k$ for $0 \le k \le n$. 
Thus $\displaystyle \max_{0 \le k \le n}d_+[f(n,k;q)] $ is uniquely realized at $k=0$. 
Hence 
\begin{align}
\label{eqn:a0small2_omega<0}
d_+[ \sum_{j,k=0}^n f(j,k;q)]
&=d_+[ \sum_{k=0}^n f(n,k;q)]\\
&= d_+[f(n,0;q)] \nonumber \\
&=- \omega n^2 - (\omega-\frac{1}{2})n  + \gamma_0. \nonumber
\end{align}

\medskip

Finally we determine the maximum degree of 
$J'_{W_{\omega}^{\tau}(K), n}(q)$ via
\[d_+[J'_{W_{\omega}^{\tau}(K), n}(q)] = d_+[\sum_{j, k=0}^n f(j,k;q)] -  \frac{n}{2}\]
using equation (\ref{normalizeddegreesum}).
By equations (\ref{eqn:a0big_omega<0}), (\ref{eqn:a0small_omega<0}), 
and (\ref{eqn:a0small2_omega<0}) we have 
\[
\delta'_{W_{\omega}^{\tau}(K)}(n) = 
\begin{cases}
(4\alpha_0- \tau-\omega-\frac{1}{2}) n^2+(2\beta_0- \tau -\omega + \frac{1}{2})n+\gamma_0 &(\alpha_0 > \frac{\tau}{4}+\frac{1}{8}) \\
- \omega n^2 - \omega n  + \gamma_0 & (\alpha_0 < \frac{\tau}{4}+\frac{1}{8})\\
- \omega n^2 - 	\omega n  + \gamma_0 & (\alpha_0 = \frac{\tau}{4}+\frac{1}{8}, -2\alpha_0 + \beta_0 +\frac{1}{2} \neq 0).
\end{cases}
\]
Notice that only equation (\ref{eqn:a0big_omega<0}) requires that $n$ be suitably large.  
Equations (\ref{eqn:a0small_omega<0}) and (\ref{eqn:a0small2_omega<0}) hold for all $n \geq 0$.
\end{proof}

\begin{proposition}[$J'$, Normalized Sign Condition, $\omega>0$]
\label{maxdeg_J'_signcond} 
Let $K$ be a knot in $S^3$ that satisfies the Normalized Sign Condition.  Let 
$N'_K$ the smallest nonnegative integer such that $d_+[J'_{K,n}(q)]$ is a quadratic quasi-polynomial 
$\delta'_K(n)=\alpha(n) n^2 +\beta(n) n+\gamma(n)$ for $n \ge 2N'_K$. 
We put $\alpha_0:=\alpha (2N'_K)$, $\beta_0:=\beta (2N'_K)$, and $\gamma_0:=\gamma (2N'_K)$. 
We assume that the period of $\delta'_K(n)$ is less than or equal to $2$ and that 
$-2\alpha_0+\beta_0+\frac 1 2 \le 0$. 
Assume further that if $-2\alpha_0+\beta_0+\frac{1}{2} = 0$, 
then $\tau \ne 4 \alpha_0$. 
Then, for suitably large $n$, 
the maximum degree of the colored Jones polynomial 
of its $\tau$-twisted  $\omega$-generalized Whitehead double with $\omega>0$ is given by 
\begin{equation}
\label{tgWmax_signcond}
\delta'_{W_{\omega}^{\tau}(K)}(n)=
\begin{cases}
(4\alpha_0- \tau) n^2 +(2\beta_0-\tau) n+\gamma_0 & (\alpha_0> \frac{\tau}{4}) \\
-n + C_+(K, \tau)  &  (\alpha_0 < \frac{\tau}{4}) \\
-n+ C_+(K, \tau) & (\alpha_0 = \frac{\tau}{4}, -2\alpha_0+\beta_0+\frac{1}{2} \ne 0)
\end{cases}
\end{equation}
where $C_+(K, \tau)$ is a number that only depends on the knot $K$ and the number $\tau$. 

Furthermore $W_{\omega}^{\tau}(K)$ satisfies the Normalized Sign Condition.
\end{proposition}

\begin{remark}
The proof below of Proposition~\ref{maxdeg_J'_signcond} is very similar to the proof of Proposition~\ref{maxdeg_J'_omega>0}.  
The main difference is that Proposition~\ref{maxdeg_J'_omega>0} assumes $N_K'=0$ but not the Sign Condition.  
Without assuming that $N'_K=0$, the behavior of $d_+[J'_{K,n}(q)]$ is uncontrolled for integers $n<2N_K'$.   
This uncontrolled behavior potentially leads to cancellations that complicate the calculation of $d_+[J'_{W_{\omega}^{\tau}(K),n}(q)]$, even for arbitrarily large $n$.  
Hence we introduce the Sign Condition on  $J'_{K,n}(q)$ which prevents such problematic cancellations.  
This is used explicitly in Cases 2.2 and 2.3 below.
\end{remark}

\begin{proof}
In light of \eqref{normalizedsum}, to determine the maximum degree of $J'_{W_{\omega}^{\tau}(K), n}(q)$, 
we need to understand  the maximum degrees of the functions $f(j,k;q)$ for $0\leq j,k \leq n$. 

Since the period of $\delta'_K(n)$ is less than or equal to $2$, it follows that 
$\alpha(2k)=\alpha_0$, $\beta(2k)=\beta_0$, and $\gamma (2k)=\gamma_0$ and so 
\begin{equation}\label{pdegK_signcondition}
d_+[J'_{K,2k}(q)]=\alpha(2k)(2k)^2+\beta(2k)2k+\gamma(2k)
=4\alpha_0 k^2+2\beta_0 k+\gamma_0
\end{equation}
for $k\geq N'_K$ (so that $2k \geq 2N'_K$).

Due to (\ref{deg6j}) the argument splits into two cases.

\medskip

\noindent
\textbf{Case $1$.\ $j+k \ge n$.} 

Applying the equalities (\ref{deg2j/nn2j}) and (\ref{deg6j}) to (\ref{fsummand}), 
we see that 
\begin{align*}
d_+[f(j,k;q)]
&= (-\frac n 2+\frac j 2) + (-\frac n 2 + \frac k 2) + \frac 1 2 (-j^2-2jk-k^2+2n+2jn+2kn-n^2) \\
& \qquad\qquad\qquad\qquad -\omega j(j+1) -\tau k(k+1) + d_+[J'_{K,2k}(q)]\\
&= -\frac{2 \omega+1}{2} j^2+(n-k- \omega+\frac 1 2)j \\
&\qquad\qquad\qquad\qquad +(- \tau -\frac 1 2) k^2+( \frac 1 2+n- \tau) k-\frac {n^2} 2+d_+[J'_{K, 2k}(q)] \\
&= 
-\left(\frac{2 \omega+1}{2} \right)\left(j-\frac{n-k- \omega+\frac{1}{2}}{2 \omega +1}\right)^2 + \frac{(n-k- \omega+\frac{1}{2})^2}{2(2\omega+1)}\\
& \qquad\qquad\qquad\qquad
+(- \tau -\frac 1 2) k^2+( \frac 1 2+n- \tau) k-\frac {n^2} 2+d_+[J'_{K, 2k}(q)]. 
\end{align*}
Since $\frac{n-k- \omega+\frac{1}{2}}{2 \omega +1}<n-k$ and $j\geq n-k$, 
this is maximized at $j=n-k$ for a fixed $k$. 
Therefore this case is included in the next case. 

\medskip

\noindent 
\textbf{Case $2$.\ $j+k \le n$.} 

Applying the equalities (\ref{deg2j/nn2j}) and (\ref{deg6j}) to (\ref{fsummand}), 
we see that 
\begin{align*}
d_+[f(j,k;q)]
&= (-\frac n 2+\frac j 2) + (-\frac n 2 + \frac k 2) + \frac 1 2 (j+k+n) \\
& \qquad\qquad\qquad -\omega j(j+1) -\tau k(k+1) + d_+[J'_{K,2k}(q)]\\
&= - \omega j^2-( \omega-1) j -\tau k^2 - (\tau-1) k -\frac{n}{2}+d_+[J'_{K,2k}(q)]\\
& = -\omega(j + \frac{\omega - 1}{2\omega})^2 + \omega(\frac{\omega -1}{2\omega})^2 - \tau k^2 - (\tau -1)k - \frac{n}{2}+d_+[J'_{K,2k}(q)].
\end{align*}
Since $\omega > 0$ so that $\omega-1 \ge 0 $, 
this is maximized uniquely at $j=0$ for a fixed $k$. Thus, 
\begin{align*}
\max_{0\le j \le n-k}d_+[f(j,k;q)] 
& = d_+[f(0,k;q)] \\
&=  -\tau k^2-(\tau-1) k -\frac{n}{2} + d_+[J'_{K,2k}(q)] 
\end{align*}
Furthermore, note that by (\ref{pdegK_signcondition}), 
 this becomes
\begin{align}
\label{signcond.j=0}
d_+[f(0,k;q)] & =  -\tau k^2-(\tau-1) k -\frac{n}{2} + 4\alpha_0 k^2 + 2\beta_0 k + \gamma_0 \\
 & = (4\alpha_0- \tau)k^2 + (2\beta_0 - \tau +1)k-\frac{n}{2} + \gamma_0 \nonumber
\end{align}
for $k \geq {N'}_K$.

\medskip

We now consider the three cases of when $\alpha_0 > \frac{\tau}{4}$, $\alpha_0 < \frac{\tau}{4}$, 
or $\alpha_0 = \frac{\tau}{4}$.
First note that when $\alpha_0 \ne \frac{\tau}{4}$, 
we may write
\begin{equation}
\label{eqn:d+bigk_signcondition}
  d_+[f(0,k;q)] =(4\alpha_0- \tau) \left(k+\frac {2\beta_0- \tau +1}{2(4\alpha_0-\tau)} \right)^2
                  - \frac {(2\beta_0- \tau +1 )^2}{4(4\alpha_0- \tau)}-\frac n 2+\gamma_0
\end{equation}
for $n\geq k\geq N'_K$. 
 Therefore, for $k >  -\frac {2\beta_0- \tau +1}{2(4\alpha_0-\tau)}$ and $k \geq N'_K$, 
 we see that $d_+[f(0,k;q)]$ is monotonically increasing or decreasing depending on whether $4\alpha_0- \tau$ is positive or negative respectively.

In the cases $\alpha_0 < \frac{\tau}{4}$ and $\alpha_0 = \frac{\tau}{4}$, 
as well as in the proof that $J'_{W_{\omega}^{\tau}(K),n}(q)$ satisfies the Normalized Sign Condition, we will need the following claim.
\begin{claim}\label{claimsigncond}
Both the denominator and the sign of the leading term of the numerator of the rational function $f(0,k;q)$ are independent of $k\geq0$.
\end{claim}

\begin{proof}
First observe that 
\begin{align*}
\label{f0ksummand}
f(0,k;q)  &=  
\frac{\langle 0\rangle}{\langle n,n,0 \rangle}\frac{\langle 2k \rangle}{\langle n,n,2k \rangle} \left\langle \begin{array}{ccc} n & n & 0 \\ n & n & 2k \end{array}\right\rangle 
q^{-\tau k(k+1)} J'_{K,2k}(q)\\
&= \frac{1}{\langle n \rangle}[2k+1] q^{-\tau k (k+1)} J'_{K,2k}(q)\\
&= \frac{1}{(-1)^n [n+1]}[2k+1] q^{-\tau k (k+1)} J'_{K,2k}(q)
\end{align*}
has numerator $[2k+1] q^{-\tau k (k+1)} J'_{K,2k}(q)$ and denominator $(-1)^n [n+1]$ that are polynomials in $\Q[q^{\pm1/4}]$. So the denominator is independent of $k$.
Since $K$ satisfies the Normalized Sign Condition by assumption, 
the sign of the leading term of the polynomial $J'_{K,2k}(q)$ is independent of $k$. 
Because the leading term of $[2k+1] q^{-\tau k (k+1)}$ is $1$, 
the sign of the leading term of the numerator $[2k+1] q^{-\tau k (k+1)} J'_{K,2k}(q)$ is also independent of $k$.
\end{proof}

\smallskip

\noindent 
\textbf{Case $2.1$.\ $\alpha_0 > \frac{\tau}{4}$.}

Note that $-\frac {2\beta_0-\tau+1}{2(4\alpha_0- \tau)} < \frac n 2 $ 
for sufficiently large $n\geq N'_K$. 
Hence equation (\ref{eqn:d+bigk_signcondition}) shows that $d_+[f(0,k;q)]$ is monotonically increasing with respect to $k$ for $k  > -\frac {2\beta_0-\tau+1}{2(4\alpha_0- \tau)}$ and $k>N'_K$ when $n$ is sufficiently large. 
Therefore, for sufficiently large $n$, $\displaystyle \max_{N'_K \le k \le n} d_+[f(0,k;q)]$ is uniquely realized at $k = n$. 
Moreover, since $d_+[f(0,n;q)]$ increases with $n$ to $\infty$ once $n$ is sufficiently large, 
it follows that
\[
\max_{0 \le k \le N'_K} d_+[f(0,k;q)] < d_+[f(0,n;q)]
\]
for suitably large $n$.
Therefore we have
\begin{equation}
\label{eqn:a0big_signcond}
  \max_{0\le k\le n} d_+[f(0,k;q)]=d_+[f(0,n;q)]=(4\alpha_0- \tau) n^2+(2\beta_0- \tau + \frac 1 2)n+\gamma_0
\end{equation} 
for suitably large $n$.

\smallskip

\noindent 
\textbf{Case $2.2$.\ $\alpha_0 < \frac{\tau}{4}$.}
 
Since 
$0 \ge -2\alpha_0+\beta_0+\frac{1}{2}
=-2(\alpha_0-\frac{\tau}{4})+\beta_0-\frac{\tau}{2}+\frac{1}{2}$ by assumption,  
then we have $2\beta_0 - \tau + 1 < 0$. 
Therefore, 
if $\alpha_0 <\frac{\tau}{4}$, 
so that $\frac {2\beta_0- \tau+1}{2(4\alpha_0- \tau)} > 0$, 
(\ref{eqn:d+bigk_signcondition}) shows that $d_+[f(0,k;q)]$ is monotonically decreasing with respect to $k$ for $N'_K \leq k \leq n$ when $n \geq N'_K$.  
Thus, assuming $n \geq N'_K$,  $ \displaystyle \max_{N'_K \le k \le n} d_+[f(0,k;q)]$ is uniquely realized at $k = N'_k$.

On the other hand,  
$\displaystyle \max_{0\le k \le N'_K} d_+[f(0,k;q)]$ may be realized at multiple values of $k$.  
This could potentially lead to cancellations among terms of maximum degree in the sum $\displaystyle \sum_{k=0}^n f(0,k;q)$. 
Since by Claim~\ref{claimsigncond} each term $f(0,k;q)$ is a rational function with the same denominator, 
any cancellations will be among the numerators.  
However Claim~\ref{claimsigncond} also shows the leading terms of the numerators of the terms $f(0,k;q)$ all have the same sign; hence there can be no cancellations.   
Thus there are no cancellations among terms of maximal degree in the sum $\displaystyle \sum_{k=0}^n f(0,k;q)$.

Consequently, there exists a smallest $k_0$ with $0 \le k_0 \le n$ such that $d_+[f(0,k_0;q)]$ gives 
the maximum degree of the sum $\displaystyle \sum_{k=0}^n f(0,k;q)$.   
In particular, since $\displaystyle \max_{N'_K \le k \le n} d_+[f(0,k;q)]$ is uniquely realized at $k = N'_K$ whenever $n \geq N'_K$, we know $0 \leq k_0 \leq N'_K$.
Hence, 
\begin{align*}
d_+[ \sum_{j,k=0}^n f(j,k;q)]
&=d_+[ \sum_{k=0}^n f(0,k;q)]\\
&= d_+[f(0,k_0;q)]  \\
&=-\tau k_0^2-(\tau-1) k_0 -\frac{n}{2} + d_+[J'_{K, 2k_0}(q)].  
\end{align*}
Furthermore when $n\geq N'_K$, the value of $k_0$ is fixed so that
\[
C_+(K, \tau)=-\tau k_0^2-(\tau-1) k_0 + d_+[J'_{K, 2k_0}(q)],
\]
 is a constant that only depends on the knot $K$ and the number $\tau$, 
 and we may write
\begin{equation}
\label{eqn:a0small_signcondition}
d_+[ \sum_{j,k=0}^n f(j,k;q)] = C_+(K,\tau) -\frac{n}{2}.
\end{equation}

\smallskip

\noindent 
\textbf{Case $2.3$.\ $\alpha_0 = \frac{\tau}{4}$.}

By the assumption that $\alpha_0 = \frac{\tau}{4}$, 
our hypotheses of this proposition imply that $-2\alpha_0 + \beta_0 +\frac{1}{2} < 0$.    
Since $\beta_0- \frac{\tau}{2} + \frac{1}{2} 
= -2(\alpha_0 - \frac{\tau}{4}) + \beta_0 - \frac{\tau}{2}+ \frac{1}{2}
= -2 \alpha_0 + \beta_0 + \frac{1}{2}$, we now have that $2\beta_0 - \tau + 1 < 0$. 
Therefore, 
\[
d_+[f(0,k;q)] = (2\beta_0 - \tau +1)k-\frac{n}{2} + \gamma_0
\]
is monotonically decreasing with respect to $k$ for $N'_K \le k \le n$ when $n \geq N'_K$. 
Thus, assuming $n \geq N'_K$,  
$\displaystyle \max_{N'_K \le k \le n}d_+[f(0,k;q)] $ is uniquely realized at $k=N'_K$. 

As in Case 2.2, $\displaystyle \max_{0 \le k \le N'_K}d_+[f(0,k;q)]$ may be realized at multiple values of $k$, 
potentially leading to cancellations among terms of maximum degree in the sum $\displaystyle \sum_{k=0}^n f(0,k;q)$.  
However, as discussed in Case 2.2, Claim~\ref{claimsigncond} implies there can be no such cancellations.  

The remainder of this case follows Case 2.2 verbatim to show that for $n \geq N'_k$
\begin{equation}
\label{eqn:a0small2_signcondition} 
d_+[ \sum_{j,k=0}^n f(j,k;q)] = C_+(K,\tau) -\frac{n}{2}.
\end{equation}

\medskip

With these cases complete, we determine the maximum degree of 
$J'_{W_{\omega}^{\tau}(K), n}(q)$ via
\[d_+[J'_{W_{\omega}^{\tau}(K), n}(q)] = d_+[\sum_{j, k=0}^n f(j,k;q)] -  \frac{n}{2}\]
using equation (\ref{normalizeddegreesum}).
Therefore, by equations (\ref{eqn:a0big_signcond}), (\ref{eqn:a0small_signcondition}), 
and (\ref{eqn:a0small2_signcondition}) we have for suitably large $n \geq N'_k$
\[
d_+[J'_{W_{\omega}^{\tau}(K), n}(q)]  = \delta'_{W_{\omega}^{\tau}(K)}(n) = 
\begin{cases}
(4\alpha_0- \tau) n^2+(2\beta_0- \tau) n + \gamma_0 &(\alpha_0 > \frac{\tau}{4}) \\
-n + C_+(K,\tau) & (\alpha_0 < \frac{\tau}{4})\\
-n + C_+(K,\tau)  & (\alpha_0 = \frac{\tau}{4}, -2\alpha_0 + \beta_0 +\frac{1}{2} \neq 0).
\end{cases}
\]

\medskip

Finally we show that $W_{\omega}^{\tau}(K)$ satisfies the Normalized Sign Condition.  
Given a polynomial $f \in \Q[q^{\pm 1/4}]$, let $\ell_+[f]$ be the coefficient of the term of highest degree.
Since $\ell_+[\langle n \rangle^2] = 1$ and 
\[ \langle n \rangle^2 J'_{W_{\omega}^{\tau}(K), n}(q) = \langle n \rangle \sum_{j,k=0}^n f(j,k;q) = \sum_{j,k=0}^n  \langle n \rangle f(j,k;q),\]
we have that 
\[\ell_+[  J'_{W_{\omega}^{\tau}(K), n}(q) ] = \ell_+\left[ \langle n \rangle^2 J'_{W_{\omega}^{\tau}(K), n}(q) \right] = \ell_+\left[ \sum_{j,k=0}^n  \langle n \rangle f(j,k;q)\right].\]
Next, note that 
\[
d_+[\langle n \rangle J'_{W_{\omega}^{\tau}(K), n}(q)] = d_+\left[\sum_{k=0}^n \langle n \rangle f(0,k;q)\right]
= d_+[\sum_{k_0 \in \mathcal{M}_n} \langle n \rangle f(0,k_0;q)]
\]
where 
\[\mathcal{M}_n = \left\{k_0  : 0 \leq k_0 \leq n, d_+[\langle n \rangle f(0,k_0;q)] = d_+\left[\sum_{k=0}^n \langle n \rangle f(0,k;q)\right]\right\}\]
 is a non-empty set due to Claim~\ref{claimsigncond}.
Then it follows that
\begin{align*}
 \ell_+[ J'_{W_{\omega}^{\tau}(K), n}(q) ]  &= \ell_+\left[ \sum_{k_0 \in \mathcal{M}_n}  \langle n \rangle f(0,k_0;q)\right] \\
 &= \sum_{k_0 \in \mathcal{M}_n} \ell_+[\langle n \rangle f(0,k_0;q)] \\
 &= \sum_{k_0 \in \mathcal{M}_n} \ell_+[(q^{k_0} + \dots + q^{-k_0}) q^{-\tau k_0(k_0+1)}J'_{K,2k_0}(q)] \\
 &= \sum_{k_0 \in \mathcal{M}_n} \ell_+[J'_{K,2k_0}(q)].
 \end{align*}
By the Normalized Sign Condition for $K$, the sign $\varepsilon'_{2k_0}(K)$ of $\ell_+[J'_{K,2k_0}(q)]$ is the same as the sign $\varepsilon'_{0}(K)$ for every $k_0\in \mathcal{M}_n$ across every $n$.  
Hence their sum always has this sign too.  
Thus the sign $\varepsilon'_n(W_{\omega}^{\tau}(K))$ of  $\ell_+[ J'_{W_{\omega}^{\tau}(K), n}(q) ]$ equals the sign $\varepsilon'_{0}(K)$ for all $n$ as claimed.
\end{proof}

\begin{remark}
\label{b=0}
In Propositions~\ref{maxdeg-signcond} and \ref{maxdeg_KT_Wmt-allw}, 
for technical reason, 
if $b_1 = 0$, then we assume that $a_1 \ne \frac{\tau}{4}$. 
When $b_1 = 0$, 
it is conjectured that $K$ is cabled \cite[Conjecture~5.1]{KT}. 
For example, 
if $K$ is an $(a, b)$-torus knot $T_{a, b}$, 
then we may have $\delta_{T_{a, b}}(n)$ with $4a_1 = ab$ and $b_1 = 0$ \cite{Garoufalidis}. 
\end{remark}

We close this section by computing $\delta_{W_{\omega}^{ab}(T_{a, b})}(n)$. 

\begin{proposition}
\label{ab-twisted-ab-torus} 
Let $a$ and $b$ be integers with $a>b>1$. 
Then the maximum degree of the colored Jones polynomial ${J}_{W_{\omega}^{ab}(K),n}(q)$ 
of $ab$-twisted $\omega$-generalized Whitehead double 
of $K=T_{a.b}$ 
is given by $-\frac{1}{2}n + \frac{1}{2}$.
\end{proposition} 

\begin{proof}
The colored Jones polynomial of $K = T_{a, b}$ is explicitly computed in \cite{Mor}: 
\begin{equation}
J'_{K,n}(q)=\frac {q^{\frac 1 4 abn(n+2)}}{q^{\frac {n+1}2}-q^{-\frac {n+1}2}}
  \sum_{k=-\frac n 2}^{\frac n 2} 
    (q^{-abk^2+(a-b)k+\frac 1 2}-q^{-abk^2+(a+b)k-\frac 1 2}). 
\end{equation}
We note that if $n$ is even, 
then $k$ is an integer in the summand. 
We define the functions $f_\pm (\ell)$ on $\Z$ by 
$$
 f_{\pm}(\ell):=-ab \ell^2+(a\mp b)\ell \pm \frac 1 2. 
$$
Since 
$$
f_{\pm}(\ell)=-ab(\ell- \frac {a\mp b}{2ab})^2+\frac {(a\mp b)^2}{4ab}\pm \frac 1 2
$$
and $0<\frac {a\mp b}{2ab}<\frac 1 2$, $f_\pm (\ell)$ is maximized at $\ell=0$ and 
$f_- (0)<f_+ (0)=\frac 1 2$. 
Hence the maximum degree of $J'_{K,n}(q)$ for even $n$ is calculated  by 
$$
\frac 1 4 abn(n+2)-\frac {n+1} 2+\frac 1 2= \frac {ab} 4 n^2+\frac {ab-1} 2 n, 
$$
and the term of the maximum degree is $q^{\frac {ab} 4 n^2+\frac {ab-1} 2 n}$. 
Therefore, the term of the maximum degree of $J'_{K,2k}(q) $ is given by 
\begin{equation}
\label{ab-maxterm}
 q^{ab k^2+(ab-1)k}. 
\end{equation}

Recall from (\ref{fsummand}) that 
\[
f(j,k;q)  =  
\frac{\langle 2j \rangle}{\langle n,n,2j \rangle}\frac{\langle 2k \rangle}{\langle n,n,2k \rangle} \left\langle \begin{array}{ccc} n & n & 2j \\ n & n & 2k \end{array}\right\rangle 
q^{-\omega j(j+1) -ab k(k+1)} J'_{K,2k}(q), 
\]
and then by
Proposition~\ref{prop:equationsplitting} we have 
\[
\langle n \rangle J'_{W_{\omega}^{ab}(K), n}(q) =\sum_{j,k=0}^n f(j,k;q).
\]
 From the proof of Proposition~\ref{maxdeg_J'_omega>0}, 
 we have that 
 $d_+[f(j,k;q)]$ is maximized uniquely at $j=0$ for a fixed $k$. 
 Moreover, Since we have that 
$\langle n,n,0 \rangle = \langle n\rangle$ and 
$
\left\langle \begin{array}{ccc} n & n & 0 \\ n & n & 2k \end{array}\right\rangle 
=\langle n,n,2k\rangle,
$
we calculate 
\begin{eqnarray*}
 f(0,k;q) & =  & \frac 1 {\langle n\rangle}
 \frac{\langle 2k\rangle}{\langle n,n,2k \rangle} \left\langle \begin{array}{ccc} n & n & 0 \\ n & n & 2k \end{array}\right\rangle 
q^{ -ab k(k+1)} J'_{K,2k}(q)\\
&= & \frac 1 {\langle n\rangle} \langle 2k\rangle q^{ -ab k(k+1)} J'_{K,2k}(q). 
\end{eqnarray*} 
From (\ref{ab-maxterm}), 
the term of the maximum degree of $f(0,k;q)$ is  
calculated as 
\[
(-1)^n q^{-\frac n 2} q^k q^{-ab k(k+1)} q^{abk^2+(ab-1)k}=(-1)^n q^{-\frac n 2}. 
\]
Hence the term of the maximum degree of ${J'}_{W_{\omega}^{ab}(K),n}(q)$ is given by 
\[
(-1)^n q^{-\frac n 2} \sum_{k=0}^n (-1)^n q^{-\frac n 2}=(n+1) q^{-n}. 
\]

Hence ${\delta'}_{W_{\omega}^{ab}(K)}(n) = -n$.
Apply the transformation given in the proof of Propositions~\ref{maxdeg-signcond} and \ref{maxdeg_KT_Wmt-allw} (after Definition~\ref{sign'}), 
we have 
\[\delta_{W_{\omega}^{ab}(K)}(n) = \delta'_{W_{\omega}^{ab}(K)}(n-1)+\frac {1}{2} n-\frac{1}{2} = -\frac{1}{2}n + \frac{1}{2}.\]
\end{proof}

\section{Computations of slopes and Euler characteristics for generalized Whitehead doubles}
\label{essential_surfaces}

\subsection{Exteriors of twisted, generalized Whitehead doubles and those of two-bridge links}
\label{exterior}

We start with a $2$--bridge link $k_1 \cup k_2$, which is expressed as 
$[2, 2\omega, -2]$ with $\omega \ge 1$ depicted in Figure~\ref{fig:2_2m_-2} below. 
Then $k_2$ lies in an unknotted solid torus $V = S^3 - \mathrm{int}N(k_1)$. 
Let us perform $\tau$ twist along $k_1$ to obtain a knot $k_{\omega}^{\tau}$ ($\tau \in \mathbb{Z}$), 
which is embedded in $V$. 
Note that $k_1 \cup k_{\omega}^{\tau}$ does not form a $2$--bridge link in general, 
but its exterior is orientation preservingly homeomorphic to the exterior of the $2$--bridge link $k_1 \cup k_2$. 
If $(\omega, \tau) = (1, 0)$, then $k_1 \cup k_1^0$ is the negative Whitehead link.  

Let us take  preferred meridian-longitude pairs 
$(\mu_1, \lambda_1)$, $(\mu, \lambda)$ of $k_1$, $k_{\omega}^{\tau}$, respectively. 
Then take an orientation preserving embedding 
$f : V \to S^3$ which sends the core of $V$ to a knot $K$ and 
$f(\mu_1) = \lambda_K$ and $f(\lambda_1) = \mu_K$, 
where $(\mu_K, \lambda_K)$ is a preferred meridian-longitude pair of $K$. 
The image $f(k_{\omega}^{\tau})$ is $W_{\omega}^{\tau}(K)$, 
a $\tau$--twisted, $\omega$--generalized Whitehead double of $K$. 

\begin{figure}[!ht]
\includegraphics[width=0.6\linewidth]{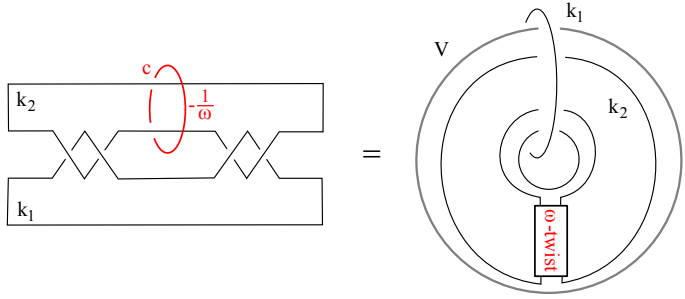}
\caption{$k_1 \cup k_2$ is a two bridge link $[2, 2\omega, -2] = \mathcal{L}_{\frac{4\omega-1}{8\omega}}$.}
\label{fig:2_2m_-2}
\end{figure}

This observation shows that the exterior of $W_{\omega}^{\tau}(K)$ is the union of 
the exterior $E(K)$ and $V - \mathrm{int}N(k_2)$; 
the latter is the exterior of the two-bridge link $k_1 \cup k_2$, which is expressed as 
$[2, 2\omega, -2]$.

Since $k_2 = k_{\omega}^{\tau}$ has winding number $0$ in $V$, 
and is therefore null-homologous in $V$, 
we have: 

\begin{lemma}
\label{twist}
Let $F$ be an essential surface in $V - \mathrm{int}N(k_2) = V - \mathrm{int}N(k_\omega^0)$  
such that $F \cap \partial V$ has slope $r_V$ and 
$F \cap \partial N(k_2)$ has slope $r$. 
Apply $\tau$ twist to obtain $k_{\omega}^{\tau}$ and an essential surface $F_{\tau}$, 
the image of $F$,  
in $ V - \mathrm{int}N(k_{\omega}^{\tau})$. 
Then $F_{\tau} \cap \partial V$ has slope $r_V + \tau$ and 
$F \cap \partial N(k_\omega^{\tau})$ has slope $r$.  
\end{lemma}
 
In the following subsections~\ref{sec:two-bridge_link}--\ref{sec:genaralizedwhiteheadalgorithm}, 
we will investigate essential surfaces in the exterior $E(k_1 \cup k_2)$ of a two-bridge link 
$[2, \omega, -2]$ in details. 

\subsection{Essential surfaces in two-bridge link exteriors}
\label{sec:two-bridge_link}
Here we extend the work in \cite[Section 5]{HS} to catalogue all the properly embedded essential surfaces in the exterior of the two-bridge link $\calL_{(4\omega-1)/8\omega}$ for integers $\omega \neq 0$.

Hatcher-Thurston show how a certain collection of ``minimal edge paths'' in the Farey diagram from $1/0$ to $p/q$ are in correspondence with the properly embedded incompressible and $\bdry$--incompressible surfaces with boundary in the exterior of the two bridge knot $\calL_{p/q}$ \cite{HT}.
Floyd-Hatcher extend this to two-bridge links of two components \cite{FH} from which Hoste-Shanahan discern the boundary slopes of such surfaces \cite{HS}, building upon work of Lash \cite{Lash}.  

Here, for use with satellite constructions, 
we use the works of Floyd-Hatcher \cite{FH} and Hoste-Shanahan\cite{HS} to catalog all the properly embedded essential surfaces in the exterior of the $\omega$ generalized Whitehead link $\calL_{(4\omega-1)/8\omega}$, 
their Euler characteristics, their boundary slopes, 
and number of boundary components; 
if $\omega > 0$, 
put $k = \omega$ to obtain $\calL_{(4k-1)/8k}$ ($k \ge 1$) and 
if $\omega < 0$, 
then put $k = - \omega$ to obtain $\calL_{(-4k-1)/ (-8k)} = \calL_{(4k+1)/8k}$ ($k \ge 1$).

\begin{remark}
While \cite{FH} uses the continued fraction convention $[x_1, x_2, \dots, x_n] = 1/(x_1 + 1/(x_2 + \dots 1/x_n))$, \cite{HS} appears to use the convention $[x_0, x_1, x_2, \dots, x_n] = x_0 + 1/(x_1 + 1/(x_2 + \dots 1/x_n))$.  
To remain consistent with this notation and the depiction of $\calL_{3/8}$ in \cite[Figure 1]{HS},  
the link $\calL_{(4k-1)/8k}$ is actually obtained by $-1/k$ surgery on the middle circle of \cite[Figure 9]{HS} which produces $2k$ right-handed crossings. 
\end{remark}

We refer the reader to both the original paper \cite{FH} and Hoste-Shanahan's recounting of it \cite[Section 2]{HS} for details on the Floyd-Hatcher algorithm.  
Here we briefly recall the algorithm and quickly work through the application of it for the Whitehead link $\calL_{3/8}$ based on the more general treatment for the links $\calL_{(4\omega-1)/8\omega}$ given in \cite[Section 5]{HS}.

\subsection{The Algorithm}
\label{algorithm}
Figure~\ref{fig:FareyDiagrams} shows three diagrams. 
The diagram $D_1$ is the common Farey diagram.  
Pair adjacent triangles into quadrilaterals containing a diagonal so that a vertex is an endpoint of either all or none of the diagonals of the incident quadrilaterals.  
The diagram $D_0$ is obtained by switching the diagonal in each of the quadrilaterals.   
The diagram $D_t$ is obtained by replacing these diagonals with inscribed quadrilaterals.  
Actually, $D_t$ represents a parameterized family of diagrams for $t \in [0,\infty]$: 
with appropriate parameterizations of the edges of the quadrilaterals by $[0,1]$ the vertices of the inscribed quadrilaterals in $D_t$ are located at either $t$ or $1/t$.  
The diagrams $D_0=D_\infty$ and $D_1$ arise as limits where the inscribed quadrilaterals degenerate to diagonals.  
The edges of $D_1$ are labeled $A$ and $C$, 
the edges of $D_0=D_\infty$ are labeled $B$ and $D$, 
and these induce labels on $D_t$.   
Orientations are chosen on a basic set of edges in $D_t$ and passed to the rest of the edges of $D_t$ by the action of the M\"obius transformations in which the ideal triangle with vertices $\{1/0,0/1,1/1\}$ is a fundamental domain.  
We omit the orientations in Figure~\ref{fig:FareyDiagrams}; see \cite{HS} for details.

\begin{figure}
	\centering
	\includegraphics[width=\textwidth]{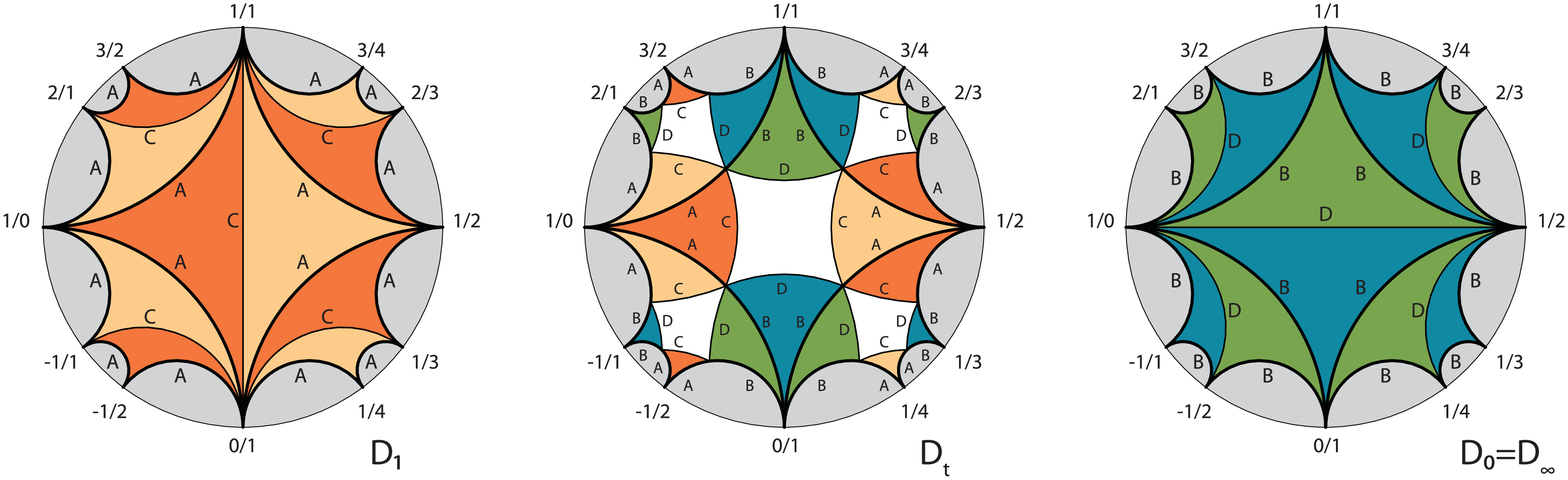}
	\caption{The diagrams $D_1$, $D_t$, and $D_0=D_\infty$, cf. \cite[Figures 3 and 4]{HS}. }
	\label{fig:FareyDiagrams}
\end{figure}

For a two bridge link $\calL_{p/q}$ (where $q$ is even), 
Floyd-Hatcher show that a properly embedded essential surface in the exterior of the link is carried by one of finitely many branched surfaces associated to ``minimal edge paths'' in $D_t$ from $1/0$ to $p/q$.  
A {\em minimal edge path} in $D_t$ is a consecutive sequence of edges of $D_t$ 
(ignoring their orientations) such that the boundary of any face of $D_t$ contains at most one edge of the path. 
Then for each minimal edge path, a branched surface is assembled from the sequence of edges by stacking four blocks of basic branched surface $\Sigma_A, \Sigma_B, \Sigma_C, \Sigma_D$ corresponding to the labels $A,B,C,D$ that are positioned according to the endpoints and orientation of its edge and whether $t< 1$ or $t> 1$.  
These blocks of basic branched surfaces are illustrated in Figure~\ref{fig:branchtypes} for $t>1$ 
(cf.\ \cite[Figure 2]{HS} and \cite[Figure 3.1]{FH}) and are weighted in terms of the parameters $\alpha>\beta>0$ where  $t = \alpha/\beta$
and the extra integral parameter $n$ between $0$ and $\beta$ for $\Sigma_A$ or between $0$ and $\alpha-\beta$ for $\Sigma_D$. 
(This extra parameter $n$ allows for the construction of homeomorphic but non-isotopic surfaces with the same boundary slopes, see \cite{HS,FH}.) 
For $t<1$, the blocks are rotated $180^\circ$ corresponding to an exchange of the components of $\calL_{p/q}$ and the parameters $\alpha$ and $\beta$ are swapped in the figure.  

\begin{figure}
\centering
\includegraphics[width=5in]{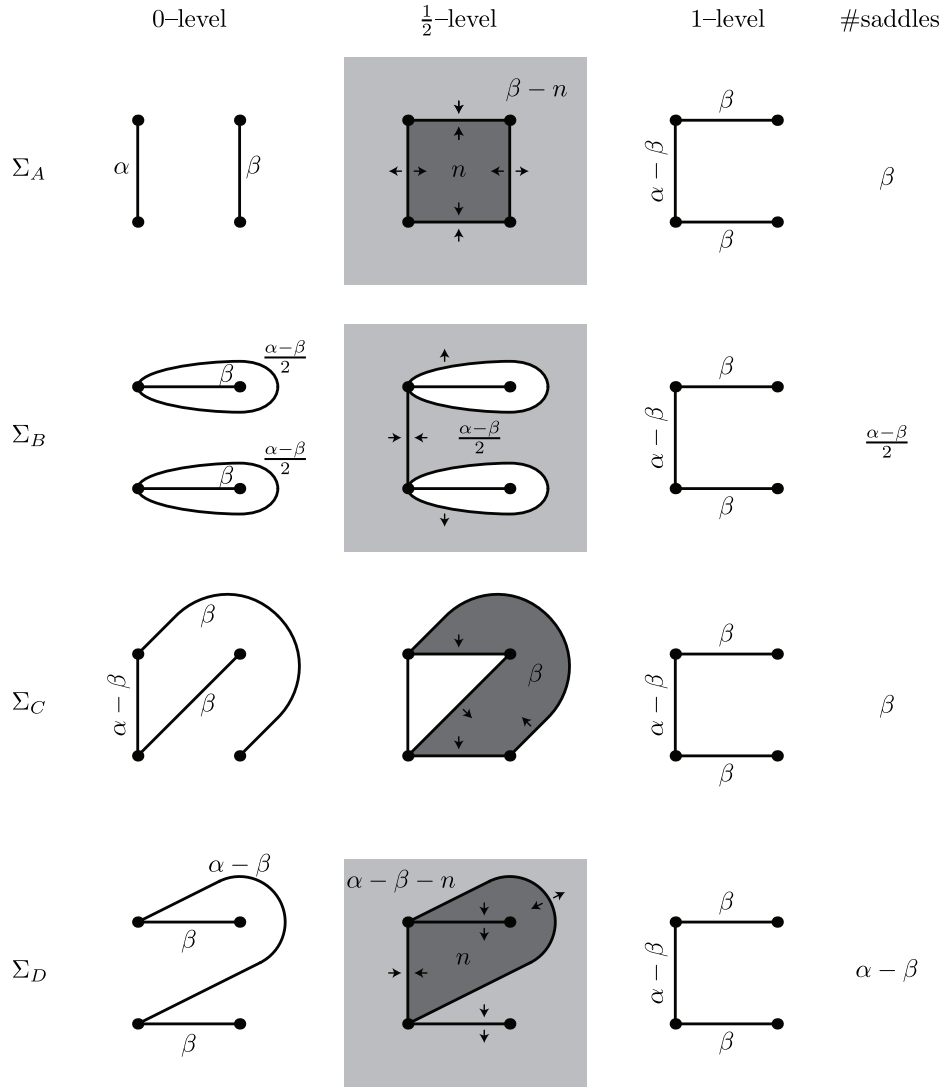}
\caption{The four basic weighted branched surfaces (reproduced from \cite[Figure 2]{HS}, see also \cite[Figure 3.1]{FH}) along with the corresponding number of saddles for the carried surface, when $\alpha\geq \beta$.  
When $\alpha < \beta$, rotate the images $180^\circ$ and swap $\alpha$ and $\beta$.}
\label{fig:branchtypes}
\end{figure}

 In this manner, every minimal edge path in $D_t$ for $t \in (0,1) \cup (1, \infty)$ produces a weighted branched surface, 
 with weights in terms of the parameters $\alpha$ and $\beta$ such that $t=\alpha/\beta$ 
 (along with auxiliary parameters for instances of the blocks $\Sigma_A$ and $\Sigma_D$). 
These minimal edge paths $\gamma$ in $D_t$ with their parameters $\alpha, \beta$ describe specific surfaces $F_{\gamma, \alpha, \beta}$ 
which may have multiple components and may be non-orientable. 
If it is non-orientable, 
then we may replace $F_{\gamma, \alpha, \beta}$ by the boundary of a tubular neighborhood (a twisted $I$--bundle over $F_{\gamma, \alpha, \beta}$), 
which is orientable and associated with parameters $2\alpha, 2\beta$; 
so the resulting orientable essential surface is associated with $F_{\gamma, 2\alpha, 2\beta}$. 
In the following we omit parameters $\alpha, \beta$ and assume that $F_{\gamma}$ is orientable, 
but it may have multiple components.

Taking the limits $t \to 0$ or $t \to \infty$ so that $\alpha=0$ or $\beta=0$ produces surfaces associated to minimal edges paths in $D_0 = D_\infty$.   
Taking the limits $t \to 1$ so that $\alpha = \beta$ also produces surfaces associated to minimal edge paths in $D_1$.  
However, since $\alpha-\beta=0$ in this case, 
the basic surface $\Sigma_A$ with its extra parameter $n$ may be used in place of $\Sigma_D$ to produce more surfaces.   

\medskip

Floyd and Hatcher \cite{FH} establish the following classification of essential surfaces in the exterior of two-bridge links. 

\begin{theorem}[\cite{FH}]
\label{FloydHatcher}
Let $\calL_{p/q}$ be a two-bridge link (with $q$ even). 
The orientable incompressible and meridionally incompressible surfaces in $S^3 - \nbhd(\calL_{p/q})$ without peripheral components are (up to isotopy) exactly the orientable surfaces carried by the collection of branched surfaces associated to minimal edge paths in $D_t$ from $1/0$ to $p/q$ for $t \in [0,\infty]$. 
\end{theorem}

\begin{remark}  
Let $S$ be a properly embedded surface in the exterior of a link $L$ in $S^3$.
Then $S$ is {\em meridionally incompressible}  if for any embedded disk $D$ in $S^3$ with $D \cap S = \bdry D$ such that $L$ intersects $D$ transversally in a single interior point, 
there is an annulus embedded in $S$ whose boundary is $\bdry D$ and a component of $\bdry S$ that is a meridian of $L$.
A component of $S$ is {\em peripheral} if it is isotopic through the exterior of $L$ into $\bdry \nbhd(L)$.
If $S$ has a $\bdry$--compressing disk, 
then either $L$ is a split link or the $\bdry$-compressible component of $S$ is either compressible or peripheral.  
Hence the surfaces in Theorem~\ref{FloydHatcher} are also $\bdry$--incompressible.
\end{remark}

\subsection{Euler characteristics of carried surfaces}

The Euler characteristic of a surface carried by one of these weighted branched surfaces associated to an edge path in $D_t$ may be calculated from the branch pattern associated to the edge path and the weights $\alpha$ and $\beta$. 

\begin{lemma}
Let $S$ be the surface carried by the weighted branched surface associated to an edge path $\gamma$ in $D_t$ where $t = \alpha/\beta$. 
If $\alpha \geq \beta$, then  \[\chi(S) = (\alpha+\beta) - \sum s_i(\alpha,\beta)\] where $s_i(\alpha,\beta)$ is the number of saddles of the surface carried by the basic branched surface associated to the label of the $i$th edge of $\gamma$ and weighted by $\alpha$ and $\beta$ as shown in Figure~\ref{fig:branchtypes}.  
If $\alpha<\beta$, exchange $\alpha$ and $\beta$.
\end{lemma}

\begin{proof}
 As shown in Figure~\ref{fig:branchtypes}, when $\alpha \geq \beta$, 
 each basic weighted branched surface of type $\Sigma_A, \Sigma_B, \Sigma_C, \Sigma_D$ carries $\beta$, $\tfrac{\alpha-\beta}{2}$, $\beta$, $\alpha-\beta$ saddles respectively and a number of vertical disks that together meet each of the upper and lower levels in a total of $\alpha+\beta$ arcs.  
 Since the weighted branched surfaces are assembled from a stack of copies of these basic weighted branched surfaces, the height function induces a Morse function on a carried surface $S$ whose singularities correspond to the saddles on the interior of the surface and the $\alpha+\beta$ half-maxima and $\alpha+\beta$ half-minima on the boundary of the surface at the extrema of the link. 
 Hence $\chi(S)$ is calculated as $\alpha+\beta$ minus the total number of saddles.  
 When $\alpha<\beta$, the blocks are rotated and $\alpha$ and $\beta$ are swapped.
 \end{proof}

\subsection{Boundaries slopes and count of boundary components}
Note that surfaces carried by these branched surfaces are given by non-negative integral weights $\alpha$ and $\beta$ (with the auxiliary integral parameters $n$ as needed), 
and these weights indicate the algebraic (and geometric) intersection numbers of the surface with the meridians $\mu_1, \mu_2$ of the two components of $\calL_{p/q}$.

Hoste and Shanahan use a certain blackboard framing $\lambda_1, \lambda_2$ of the two components of $\calL_{p/q}$ to further keep track of how the branched surfaces associated to minimal edge paths in $D_t$ intersect this framing.  
They then determine how to correct this framing to the canonical framings $\lambda_1^0, \lambda_2^0$ of the individual unknot components of the two-bridge link. 
From this, one then obtains the boundary slopes of the carried surfaces in terms of the canonical framings of the components.

Furthermore, by a calculation in the homology of a torus, the greatest common divisor ($\gcd$) 
of the algebraic intersection numbers of the boundary of a surface with the meridian and longitudinal framing of a component of $\calL_{p/q}$ produces the number of boundary components of the surface meeting that component of $\calL_{p/q}$.

\subsection{Applying the Algorithm to the Whitehead Link-- $2$--bridge link $[2, 2, -2]$}
\label{sec:whiteheadalgorithm}
As a warm-up example, in this subsection, 
we apply the algorithm to the Whitehead Link, which is the $2$--bridge link $[2, 2, -2]$.
Figure~\ref{fig:DtDiagram3-8} shows the portions of the diagrams $D_0=D_\infty$, $D_t$, and $D_1$ that carry the minimal edge paths from $1/0$ to $3/8$. 
Table~\ref{table:3/8paths} lists these minimal edge paths with their names as given in each \cite{HS} and \cite{FH}, 
the branch pattern of the induced branched surface (i.e.\ the sequence of edge labels), 
and the Euler characteristic of the carried surface corresponding to weights $\alpha\geq \beta$.  
Table~\ref{table:3/8slopes} lists for each of these paths the boundary slopes of the carried surfaces relative to the canonical meridian-longitude framings of the two unknot components of the two-bridge links and the count of the number of  boundary components on each link component.  
These are also calculated from the given preliminary data of algebraic intersections of the boundary components with the meridians and blackboard framed longitudes and the boundary slopes in terms of the blackboard framing; 
refer to \cite{HS} for details.  
Note that for each of the paths $\gamma_i$, $i \in \{1,2,3,5,6\}$, 
when $\beta=0$ so that $\alpha/\beta=\infty$ the associated essential surface is disjoint from the 2nd link component. 
Table~\ref{table:3/8slopes} summarizes the relevant data. When $\alpha < \beta$ we may continue to use the two tables, 
but with $\alpha$ and $\beta$ swapped and with the two link components swapped.

\begin{figure}
	\centering
	\includegraphics[width=5in]{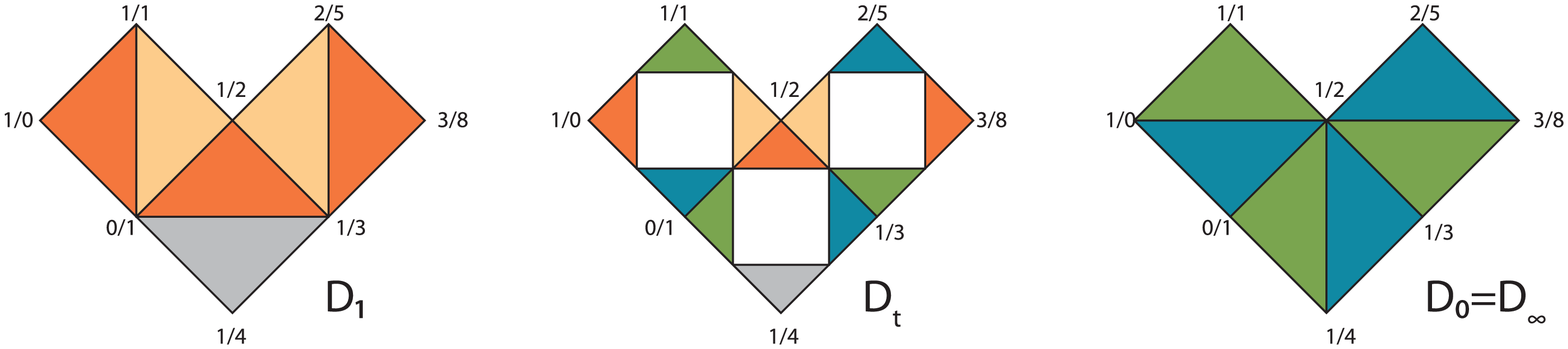}
	\caption{The diagrams $D_1$, $D_t$, and $D_0=D_\infty$ that carry the minimal edge paths from $\frac{1}{0}$ to $\frac{3}{8}$.}
	\label{fig:DtDiagram3-8}
\end{figure}

\begin{table*}
	\centering
	\caption{
	The minimal edge paths in $D_t$ from $0/1$ to $3/8$, the branch patterns of their supporting branched surfaces, and the Euler characteristics for the surfaces when $t = \alpha/\beta > 1$ are shown. 
The HS path name is established in \cite[Table 2]{HS}.  
}
	\label{table:3/8paths}
	\begin{tabular}{@{}M{20mm}M{35mm}M{25mm}M{20mm}@{}}
		\toprule
		HS path    & path picture & branch pattern  & $\chi$ \\
		\midrule
		$\gamma_1$  & \includegraphics[width=1in]{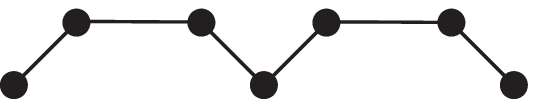} & $ADAADA$ & $-\alpha - \beta$ \\ \addlinespace[0.5em]
		$\gamma_2$  & \includegraphics[width=1in]{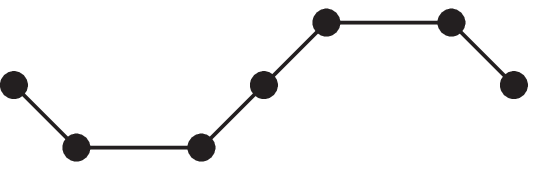} & $ADAADA$ & $-\alpha - \beta$ \\ \addlinespace[0.5em]
		$\gamma_3$  & \includegraphics[width=1in]{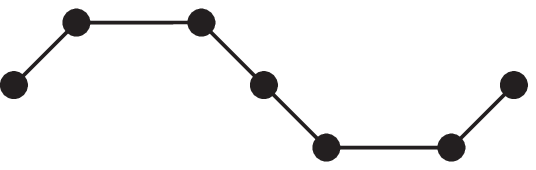} & $ADAADA$ & $-\alpha - \beta$ \\ \addlinespace[0.5em]
		$\gamma_5$  & \includegraphics[width=1in]{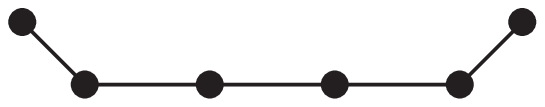} & $ADCDA$  & $-\alpha$ \\ \addlinespace[0.5em]
		$\gamma_6$ &  \includegraphics[width=1in]{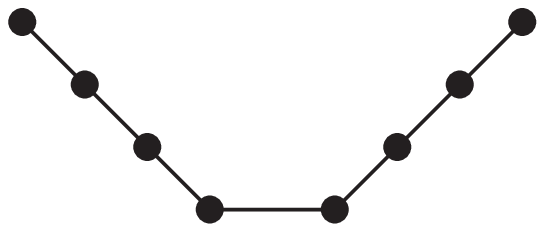} & $ABBCBBA$ & $-\alpha$ \\ \addlinespace[0.5em]
		\bottomrule
	\end{tabular}
\end{table*}

\begin{table}
	\centering
	\caption{The two boundary slopes and number of boundary components of each surface carried by a branched surface associated to a minimal edge path from $0/1$ to $3/8$ for the first and second link component are presented as a pair.  
The table shows the case $\alpha > \beta > 0$ so that $t = \alpha/\beta \in (1, \infty)$. 
For $t=\infty$ when $\alpha>\beta = 0$, 
the surface is disjoint from the second component so the second coordinate in the last three columns are $\emptyset, \emptyset, 0$ respectively. 
For $t \in [0,1)$, 
apply the homeomorphism of the two-bridge link that swaps its two components, 
i.e.\ exchange coordinates and swap $\alpha$ and $\beta$.
Actually there is also a path $\gamma'_5$ for $t = \alpha / \beta = 1$, but we omit it.}
	\label{table:3/8slopes}
	{\small
	\begin{tabular}{@{}lllll@{}}
		\toprule
		 HS & alg.\ int.\ with & slopes with  &  slopes with & number of \\
		path & $(\lambda_1, \mu_1, \lambda_2, \mu_2)$ &  blackboard framing  & canonical framing & boundary components  \\
		\midrule
		$\gamma_1$ &  $(\alpha+2 \beta, \alpha, 2\alpha+\beta, \beta)$
		& $(1+2 \tfrac{\beta}{\alpha}, 2\tfrac{\alpha}{\beta}+1)$ & $(2 \tfrac{\beta}{\alpha}, 2\tfrac{\alpha}{\beta})$ & $(\gcd(2\beta,\alpha), \gcd(2\alpha,\beta))$ 
		\\
		$\gamma_2$ &  $(\alpha, \alpha, \beta, \beta)$ 
		& $(1,1)$ & $(0,0)$ & $(\alpha,\beta)$ 
		\\
		$\gamma_3$ &  $(\alpha, \alpha, \beta, \beta)$ 
		& $(1,1)$ & $(0,0)$ & $(\alpha,\beta)$ 
		\\
		$\gamma_5$ &  $(\alpha-2 \beta, \alpha, -2\alpha-\beta, \beta)$  
		& $(1-2 \tfrac{\beta}{\alpha}, -2\tfrac{\alpha}{\beta}-1)$ & $(-2 \tfrac{\beta}{\alpha}, -2\tfrac{\alpha}{\beta}-2)$ & $(\gcd(2\beta,\alpha), \gcd(2\alpha,\beta))$ 
		\\
		$\gamma_6$ &  $(-3\alpha, \alpha, -\beta, \beta)$ 
		&   $(-3,-1)$& $(-4,-2)$ & $(\alpha,\beta)$ 
		\\
		\bottomrule
	\end{tabular}
	}
\end{table}

\begin{table*}
	\centering
	\caption{
	Summary of data in terms of canonical framing.}
	\label{table:3/8summary}
	{\small
	\begin{tabular}{@{}llllllll@{}}
		\toprule
	HS   & branch   & $\chi$ & \multicolumn{2}{l}{boundary slopes} & \phantom{x}&  \multicolumn{2}{l}{number of boundary components} \\
		path	&	pattern		&		& $_{\beta>0}$ & $_{\beta=0}$		&	& $_{\beta>0}$ & $_{\beta=0}$\\
	\midrule
	$\gamma_1$ &$ADAADA$ & $-\alpha - \beta$ & $(2 \tfrac{\beta}{\alpha}, 2\tfrac{\alpha}{\beta})$ & $(2 \tfrac{\beta}{\alpha}, \emptyset)$ && $(\gcd(2\beta, \alpha), \gcd(2\alpha, \beta))$ & $(\gcd(2\beta, \alpha), 0)$ \\
	$\gamma_2$ &  $ADAADA$ & $-\alpha - \beta$ & $(0,0)$ & $(0,\emptyset)$& & $(\alpha,\beta)$& $(\alpha,0)$\\
	$\gamma_3$ &  $ADAADA$ & $-\alpha - \beta$ & $(0,0)$ & $(0,\emptyset)$& & $(\alpha,\beta)$& $(\alpha,0)$\\
	$\gamma_5$ &  $ADCDA$  & $-\alpha$ & $(-2 \tfrac{\beta}{\alpha}, -2\tfrac{\alpha}{\beta}-2)$ & $(-2 \tfrac{\beta}{\alpha}, \emptyset)$&& $(\gcd(2\beta, \alpha), \gcd(2\alpha, \beta))$& $(\gcd(2\beta, \alpha), 0)$ \\
	$\gamma_6$ & $ABBCBBA$ & $-\alpha$ & $(-4,-2)$ & $(-4, \emptyset)$ && $(\alpha,\beta)$& $(\alpha,0)$\\
	\bottomrule
	\end{tabular}
	}
\end{table*}

\subsection{Applying the Algorithm to a generalized Whitehead Link -- $2$--bridge link $[2, 2\omega, -2]$}
\label{sec:genaralizedwhiteheadalgorithm}
Let us apply the algorithm to a generalized Whitehead Link, which is a $2$--bridge link $[2, 2\omega, -2]$. 
Recall that a $2$--bridge link $[2, 2\omega, -2]$ is expressed as $\calL_{(4\omega -1)/8\omega}$. 
 
Figure~\ref{fig:DtDiagramL4k-1--8k} shows the portions of the diagrams $D_0=D_\infty$, $D_t$, 
and $D_1$ that carry the minimal edge paths from $1/0$ to 
$(4\omega -1)/8\omega = (4 k-1)/8k$ 
with $k = \omega \ge 1$.  
When $\omega < 0$, 
putting $k = -\omega \ge 1$, 
we have $(4\omega -1)/8\omega = (-4 k-1)/(-8k) = (4k+1)/8k$ with $k = -\omega \ge 1$. 
Figure~\ref{fig:DtL4k+1--8k_omega_minus} shows the portions of the diagrams $D_0=D_\infty$, $D_t$, 
and $D_1$ that carry the minimal edge paths from $1/0$ to $(4 k+1)/8k$ with $k \ge 1$.

We obtain Tables ~\ref{table:4k-1by8kpaths}, \ref{table:4k-1by8kslopes} and \ref{table:4k-1by8ksummary} corresponding to Tables~\ref{table:3/8paths}, \ref{table:3/8slopes} and \ref{table:3/8summary}. 
Note that for each of the paths $\gamma^{\pm}_i$, $i \in \{1,2,3,4,5,6\}$, 
when $\beta=0$ so that $\alpha/\beta=\infty$ the associated surface is disjoint from the second link component.  
A path $\gamma^{+}_i$ is a minimal path from $\frac{1}{0}$ to $\frac{4k-1}{8k}$ (corresponding to the case where $\omega > 0$), 
and a path $\gamma^{-}_i$ is a minimal path from $\frac{1}{0}$ to $\frac{4k+1}{8k}$ (corresponding to the case where $\omega < 0$). 
Indeed, for $i \in \{1,2,3,4,5\}$, 
when $\alpha=1$ and $\beta=0$, 
the associated surface is a once-punctured torus Seifert surface for the first link component.
For $i=6$, when $\alpha=2$ and $\beta=0$, 
the associated surface is a twice punctured torus disjoint from the second component.

\begin{figure}[h]
	\centering
	\includegraphics[width=6in]{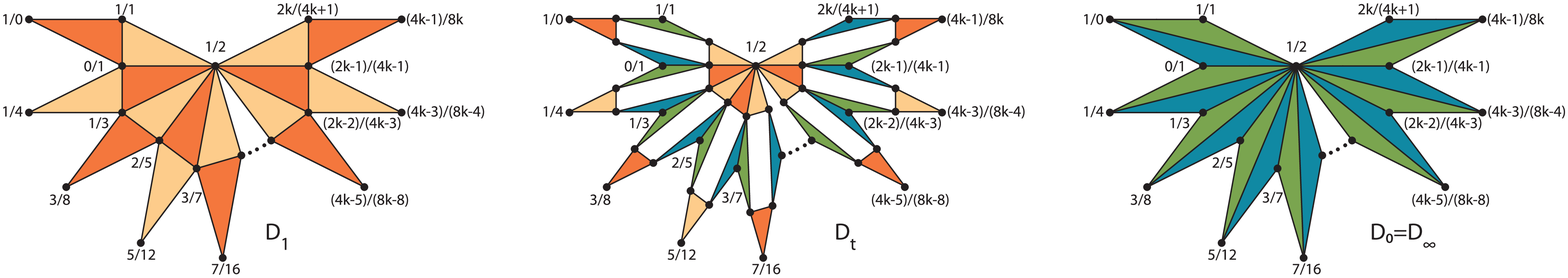}
	\caption{The diagrams $D_1$, $D_t$, and $D_0=D_\infty$ that carry the minimal paths from $\frac{1}{0}$ to $\frac{4k-1}{8k}$ with $k \ge 1$.}
	\label{fig:DtDiagramL4k-1--8k}
\end{figure}

\begin{figure}[h]
	\centering
	\includegraphics[width=6in]{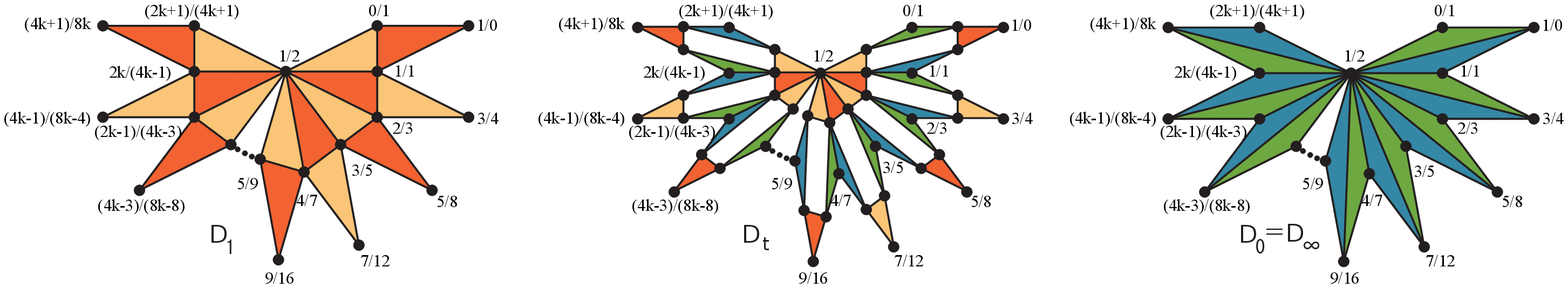}
	\caption{The diagrams $D_1$, $D_t$, and $D_0=D_\infty$ that carry the minimal paths from $\frac{1}{0}$ to $\frac{4k+1}{8k}$ with $k \ge 1$.}
	\label{fig:DtL4k+1--8k_omega_minus}
\end{figure}

\begin{table*}
	\centering
	\caption{
		The minimal edge paths $\gamma^{+}_i$ in $D_t$ from $1/0$ to $(4k-1)/8k$ and 
		the minimal edge paths $\gamma^{-}_i$ in $D_t$ from $1/0$ to $(4k+1)/8k$ for positive integers $k$,
		the branch patterns of their supporting branched surfaces, and the Euler characteristics for the surfaces when $t = \alpha/\beta > 1$ are shown. 
		There is also a path $\gamma'^{\pm 5}$ for $t = \alpha / \beta = 1$, but we omit it.
		The HS path name is established in \cite[Table 2]{HS}.  
		(Compare with Table~\ref{table:3/8paths} for $L_{3/8}$.)
	}
	\label{table:4k-1by8kpaths}
	\begin{tabular}{@{}M{10mm}M{35mm}M{25mm}M{45mm}@{}}
		\toprule
		HS path    & path picture & branch pattern  & $\chi$ \\
		\midrule
		$\gamma^{\pm}_1$  & \includegraphics[width=1in]{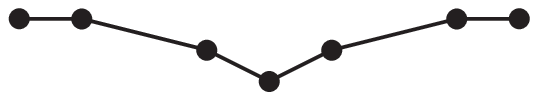} & $ADAADA$ & $-\alpha - \beta$ \\ \addlinespace[0.5em]
		$\gamma^{\pm}_2$  & \includegraphics[width=1in]{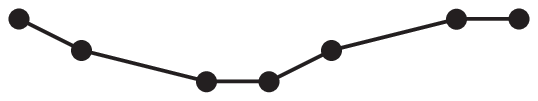} & $ADAADA$ & $-\alpha - \beta$ \\ \addlinespace[0.5em]
		$\gamma^{\pm}_3$  & \includegraphics[width=1in]{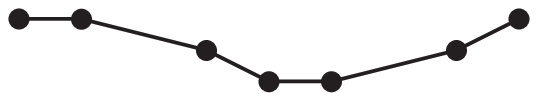} & $ADAADA$ & $-\alpha - \beta$ \\ \addlinespace[0.5em]
		$\gamma^{\pm}_4$  & \includegraphics[width=1in]{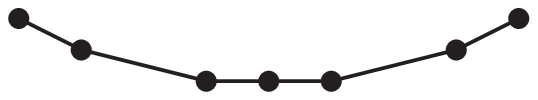} & $ADAADA$ & $-\alpha - \beta$ \\ \addlinespace[0.5em]		
		$\gamma^{\pm}_5$  & \includegraphics[width=1in]{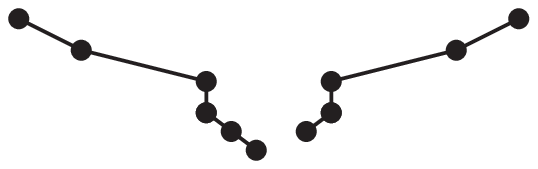} & $ADC^{2k-1}DA$  & $-\alpha+2(1-k)\beta$ \\ \addlinespace[0.5em]
		$\gamma^{\pm}_6$ &  \includegraphics[width=1in]{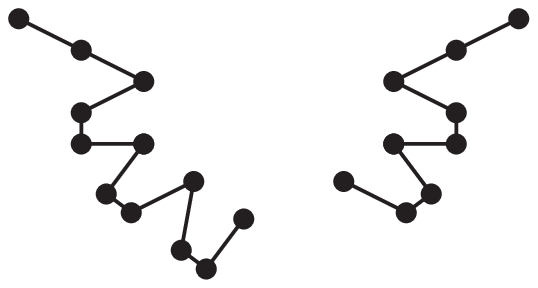} & $AB(BCB)^{2k-1}BA$ & $(1-2k)\alpha$ \\ 
		\bottomrule
	\end{tabular}
\end{table*}

\begin{table}
	\centering
	\caption{
	The two boundary slopes and number of boundary components of each surface carried by a branched surface associated to a minimal edge path from $0/1$ to 
	$(4k-1)/8k$  or $(4k+1)/8k$ for the first and second link component are presented as a pair.  
	The table shows the case $\alpha > \beta > 0$ so that $t = \alpha/\beta \in (1, \infty)$.
	For $t=\infty$ when $\alpha>\beta = 0$, 
	the surface is disjoint from the second component so the second coordinate in the last two columns are $\emptyset$ and $0$ respectively.
	 For $t \in (0,1]$, apply the homeomorphism of the two-bridge link that swaps its two components, 
	 i.e.\ exchange coordinates and swap $\alpha$ and $\beta$.
	(Compare with Table~\ref{table:3/8slopes} for $L_{3/8}$.)}
	\label{table:4k-1by8kslopes}
	{\small
	\begin{tabular}{@{}llll@{}}
		\toprule
		HS & alg.\ int.\ with &  slopes with & number of \\
		path & $(\lambda_1, \mu_1, \lambda_2, \mu_2)$ & canonical framing & boundary components  \\
		\midrule
		$\gamma^{\pm}_1$ &  $(\pm \alpha \pm 2 \beta, \alpha, \mp 2\alpha \mp \beta, \beta)$
		& $(\pm 2 \tfrac{\beta}{\alpha}, \pm 2\tfrac{\alpha}{\beta})$ & $(\gcd(2\beta,\alpha), \gcd(2\alpha,\beta))$ 
		\\
		$\gamma^{\pm}_2$ &  $(\pm \alpha, \alpha, \pm \beta, \beta)$ 
		& $(0,0)$ & $(\alpha,\beta)$ 
		\\
		$\gamma^{\pm}_3$ &  $(\pm \alpha, \alpha, \pm \beta, \beta)$ 
		& $(0,0)$ & $(\alpha,\beta)$ 
		\\
		$\gamma^{\pm}_4$ &  $(\pm \alpha \mp 2\beta, \alpha, \mp 2\alpha \pm \beta, \beta)$ 
		& $(\mp 2 \tfrac{\beta}{\alpha}, \mp 2\tfrac{\alpha}{\beta})$ & $(\gcd(2\beta,\alpha), \gcd(2\alpha,\beta))$  
		\\
		$\gamma^{\pm}_5$ &  $(\pm \alpha \mp 2 \beta, \alpha, \mp 2\alpha\pm (3-4k)\beta, \beta)$  
		& $(\mp 2 \tfrac{\beta}{\alpha}, \mp 2\tfrac{\alpha}{\beta}\pm 2 \mp 4k)$ & $(\gcd(2\beta,\alpha), \gcd(2\alpha,\beta))$ 
		\\
		$\gamma^{\pm}_6$ &  $(\pm (1-4k)\alpha, \alpha, \mp \beta, \beta)$ 
		& $(\mp 4k,\mp 2)$ & $(\alpha,\beta)$ 
		\\
		\bottomrule
	\end{tabular}
	}
\end{table}

\begin{table*}
	\centering
	\caption{
	Summary of data in terms of canonical framing.  
	(Compare with Table~\ref{table:3/8summary} for $L_{3/8}$.)}  
	\label{table:4k-1by8ksummary}
	{\small
	\begin{tabular}{@{}lllllll@{}}
		\toprule
	HS      & $\chi$ & \multicolumn{2}{l}{boundary slopes} & \phantom{x}&  \multicolumn{2}{l}{number of boundary components} \\
		path			&		& $_{\beta>0}$ & $_{\beta=0}$		&	& $_{\beta>0}$ & $_{\beta=0}$\\
	\midrule
	$\gamma^{\pm}_1$ & $-\alpha - \beta$ & $(\pm 2 \tfrac{\beta}{\alpha}, \pm 2\tfrac{\alpha}{\beta})$& $(\pm 2 \tfrac{\beta}{\alpha}, \emptyset)$& & $(\gcd(2\beta, \alpha), \gcd(2\alpha, \beta))$ & $(\gcd(2\beta, \alpha), 0)$ \\
	$\gamma^{\pm}_2$  & $-\alpha - \beta$ & $(0,0)$& $(0,\emptyset)$& & $(\alpha,\beta)$ & $(\alpha,0)$\\
	$\gamma^{\pm}_3$  & $-\alpha - \beta$ & $(0,0)$ & $(0,\emptyset)$ && $(\alpha,\beta)$ & $(\alpha,0)$\\
	$\gamma^{\pm}_4$  & $-\alpha - \beta$ & $(\mp 2 \tfrac{\beta}{\alpha}, \mp 2\tfrac{\alpha}{\beta})$ & $(\mp 2 \tfrac{\beta}{\alpha},\emptyset)$ && $(\gcd(2\beta,\alpha), \gcd(2\alpha,\beta))$& $(\gcd(2\beta,\alpha), 0)$ \\
	$\gamma^{\pm}_5$  & $-\alpha + 2(1-k)\beta$ & $(\mp 2 \tfrac{\beta}{\alpha}, \mp 2\tfrac{\alpha}{\beta} \pm 2 \mp 4k)$ & $(\mp 2 \tfrac{\beta}{\alpha}, \emptyset)$ && $(\gcd(2\beta,\alpha), \gcd(2\alpha,\beta))$& $(\gcd(2\beta,\alpha), 0)$ \\
	$\gamma^{\pm}_6$ & $(1-2k)\alpha$ & $(\mp 4k, \mp 2)$ & $(\mp 4k,\emptyset)$ && $(\alpha,\beta)$& $(\alpha,0)$ \\
	\bottomrule
	\end{tabular}
	}
\end{table*}

\pagebreak
\section{Slope conjecture for twisted generalized Whitehead doubles}
\label{slope_cpnjecture_Wmt}

In this section we prove Theorem~\ref{bothslopeconjecturesWmt}(1), 
which we state here again for convenience. 

\begin{thm_slope conjecture Wmt}
Let $K$ be a knot in $S^3$ and $d_+[J_{K,n}(q)]$ is the quadratic quasi-polynomial 
$\delta_K(n)=a(n) n^2 +b(n) n+c(n)$ for all $n \ge 0$.
We put $a_1:=a(1)$, $b_1:=b(1)$, and $c_1:=c(1)$. 
We assume that the period of $d_+[J_{K,n}(q)]=\delta_K(n)$ is less than or equal to  $2$ and that 
$b_1 \le 0$.
Assume further that if $b_1 = 0$, 
then $a_1 \ne \frac{\tau}{4}$. 
If $K$ satisfies the Slope Conjecture, 
then all of its twisted generalized Whitehead doubles also satisfy the Slope Conjecture.
\end{thm_slope conjecture Wmt}

\begin{proof}
Assume first that $\omega > 0$.
It follows from Proposition~\ref{maxdeg_KT_Wmt-allw} that 
if $r_W \in js_{W_{\omega}^{\tau}(K)}$, 
then there exists $r_K \in js_K$ such that 
\begin{itemize}
\item
$r_W = 4 r_K - 4\tau$ if $a_1 > \frac{\tau}{4}$, or
\item
$r_W  = 0$ if $a_1 \le \frac{\tau}{4}$.   
\end{itemize}

Suppose next that $\omega < 0$.
Then Proposition~\ref{maxdeg_KT_Wmt-allw} shows that 
if $r_W \in js_{W_{\omega}^{\tau}(K)}$, 
then there exists $r_K \in js_K$ such that 
\begin{itemize}
\item
$r_W = 4 r_K -4 \omega -2 -4 \tau$ if $a_1 > \frac{\tau}{4}+\frac{1}{8}$, or
\item
$r_W  = -4 \omega$ if $a_1  \le  \frac{\tau}{4}+\frac{1}{8}$.
\end{itemize}

Let us find essential surfaces in $E(W_{\omega}^{\tau}(K))$ whose boundary slopes are these Jones slopes. 

Recall that $k_1 \cup k_2$ is a $2$--bridge link expressed as $[2, 2\omega, -2]$ ($\omega \ne 0$) depicted in Figure~\ref{fig:2_2m_-2}; 
$(\mu_i, \lambda_i)$ denotes a preferred meridian-longitude pair of $k_i$. 
As in Figure~\ref{fig:2_2m_-2} take a solid torus $V = S^3 - \mathrm{int}N(k_1)$ which contains $k_2$ in its interior; 
let $(\mu_V, \lambda_V)$ be the standard meridian-longitude pair of $V \subset S^3$. 
Performing $-1/\tau$--surgery on $k_1$, 
equivalently $\tau$--twisting along $\mu_V$,  
we obtain $k_{\omega}^{\tau}$ which is the image of $k_2$ (Figure~\ref{fig:satellite}). 
Let $f$ be an orientation preserving embedding 
$f : V \to S^3$ which sends $V$ to $N(K)$ and 
$f(\mu_V) = \mu_K$ and $f(\lambda_V) = \lambda_K$, 
where 
$(\mu_K, \lambda_K)$ is a preferred meridian-longitude pair of $K$.  
Then $W_{\omega}^{\tau}(K) = f(k_{\omega}^{\tau})$ is the $\tau$--twisted generalized Whitehead double of $K$.  
Thus the exterior 
$E(W_{\omega}^{\tau}(K))$ is the union of $E(K)$ and $f(V - \mathrm{int}N(k_{\omega}^{\tau}))$. 
The boundary of $f(V - \mathrm{int}N(k_{\omega}^{\tau}))$ consists of two tori $T_W = \partial N(W_{\omega}^{\tau}(K))$ and $T_K = f(\partial V) = \partial E(K)$. 
Then $(f(\mu_2), f(\lambda_2))$ is a preferred meridian-longitude pair $(\mu_W, \lambda_W)$ of $W_{\omega}^{\tau}(K)$. 

\medskip 

\noindent
\textbf{1. Realization of the Jones slopes when $\omega > 0$.}

We divide into two cases depending upon $a_1 > \frac{\tau}{4}$ or $a_1 \le \frac{\tau}{4}$; 
see Proposition~\ref{maxdeg_KT_Wmt-allw}.  

\smallskip

\noindent 
\textbf{Case 1.\ $a_1 > \frac{\tau}{4}$.} 

Since $K$ satisfies the Slope Conjecture, 
the Jones slope $4a_1$ is realized by a boundary slope of an essential surface $S_K \subset E(K)$.  

\begin{claim}
\label{F_m^t_omega>0}
There exists an essential surface $F_{\omega}^{\tau}$ in $V - \mathrm{int}N(k_{\omega}^{\tau})$ 
such that each component of $F_{\omega}^{\tau} \cap \partial V$ has slope $4a_1$ and 
each component of $F_{\omega}^{\tau} \cap \partial N(k_{\omega}^{\tau})$ has $16a_1 - 4\tau$.  
\end{claim}

\begin{proof}
Let us take an essential surface $F_{\gamma^{+}_1}$ in 
$S^3 - \mathrm{int}N(k_1 \cup k_2) = V - \mathrm{int}N(k_2)$ associated to the minimal edge path $\gamma^{+}_1$ 
described in Section~\ref{essential_surfaces}. 
Then it has a pair of boundary slopes 
$(2\frac{\beta}{\alpha}, 2\frac{\alpha}{\beta})$ on $k_1, k_2$.
Then $F_{\gamma^{+}_1}$ has boundary slopes $\frac{2\beta}{\alpha}$ on $\partial N(k_1)$ and 
$\frac{2\alpha}{\beta}$ on $\partial N(k_2)$.  
Using the preferred meridian-longitude $(\mu_V, \lambda_V)$ of $V$ instead of $(\mu_1, \lambda_1)$ of $k_1$, 
$F_{\gamma^{+}_1} \cap \partial V$ has slope $\frac{\alpha}{2\beta}$. 
Choose $\alpha, \beta$ so that 
$\frac{\alpha}{2\beta} = 4a_1 - \tau > 0$.
Hence $\frac{\alpha}{\beta} = 8a_1 -2\tau > 0$.  
Then $F_{\gamma^{+}_1} \subset V - \mathrm{int}N(k_2)$ has boundary slope $16a_1 -4\tau$ on $\partial N(k_2)$ 
and $4a_1 - \tau$ on $\partial V$. 
Now we apply $\tau$--twisting along $\mu_V$ which changes 
$V - \mathrm{int}N(k_2)$ to $V - \mathrm{int}N(k_{\omega}^{\tau})$; 
we denote the image of $F_{\gamma^{+}_1}$ by $F_{\omega}^{\tau}$. 
By Lemma~\ref{twist} each component of $F_{\omega}^{\tau} \cap \partial V$ has slope $4a_1 (= 4a_1 - \tau + \tau)$,  
and each component of $F_{\omega}^{\tau} \cap \partial N(k_{\omega}^{\tau})$ has slope $16a_1 - 4\tau$ as desired. 
\end{proof}

\medskip

Let us take the image $f(F_{\omega}^{\tau})$ in $f(V - \mathrm{int}N(k_{\omega}^{\tau}))$, 
and denote it by $S_{\omega}^{\tau}$.  
Write $T_K = \partial E(K) = f(\partial V)$ and $T_W = \partial N(W_{\omega}^{\tau}(K)) = f(\partial N(k_{\omega}^{\tau}))$. 
By construction $S_{\omega}^{\tau}$ is essential in $f(V - \mathrm{int}N(k_{\omega}^{\tau}))$ and 
each component of $S_{\omega}^{\tau} \cap T_K$ has slope $4a_1$ and each component of $S_{\omega}^{\tau} \cap T_W$ has slope $16a_1 - 4\tau$. \par

To build a required essential surface $S \subset E(W_{\omega}^{\tau}(K))$ we take $m$ parallel copies 
$m S_{\omega}^{\tau}$ of the essential surface $S_{\omega}^{\tau}$ and $n$ parallel copies $n S_K$ of the essential surface $S_K$, 
and then glue them along their boundaries to obtain a connected surface 
$S = m S_{\omega}^{\tau} \cup n S_K$ in $E(W_{\omega}^{\tau}(K))$.  
Even when both $S_{\omega}^{\tau}$ and $S_K$ are orientable, 
$S$ may not be orientable. 
If $S$ is non-orientable, 
then consider a regular neighborhood of $S$ in $E(W_{\omega}^{\tau}(K))$, 
which is a twisted $I$--bundle of $S$ whose $\partial I$--subbundle is an orientable double cover of $S$. 
We use the same symbol $S$ to denote this $\partial I$--subbundle. 
Note that $S_K$ and $S_{\omega}^{\tau}$ are orientable,  
so $S \cap E(K)$ consists of parallel copies of $S_K$ and similarly 
$S \cap f(V - \mathrm{int}N(k_{\omega}^{\tau}))$ consists of parallel copies of $S_{\omega}^{\tau}$.  
Since $\partial E(K)$ is incompressible in $E(W_{\omega}^{\tau}(K))$ and 
$S_{\omega}^{\tau}$, $S_K$ are essential in $f(V - \mathrm{int}N(k_{\omega}^{\tau}))$ and $E(K)$ respectively, 
$S$ is incompressible in $E(W_{\omega}^{\tau}(K))$.  
If $S$ were boundary-compressible, 
then a component of $S$ would be a boundary-parallel annulus.
However, obviously each component of $S$ is not an annulus, 
and we have a contradiction. 
Hence $S$ is the desired essential surface. 

\smallskip

\noindent 
\textbf{Case 2. $a_1 \le \frac{\tau}{4}$.}

In this case Proposition~\ref{maxdeg_KT_Wmt-allw} shows that the Jones slope is $0$, 
and we explicitly give a desired essential surface. 
Let us take a once punctured torus $F$ bounded by $k_2$ which is contained in 
$V - \mathrm{int}N(k_2)$.
It has the boundary slope $0$ on $k_2$. 
Let $S$ be the image $f(F)$ in $E(W_{\omega}^{\tau}(K))$, 
which is a minimal genus Seifert surface of $W_{\omega}^{\tau}(K)$ and essential in its exterior. 
Thus the jones slope $0$ is a boundary slope of $W_{\omega}^{\tau}(K)$. 
\medskip

\noindent
\textbf{2. Realization of the Jones slopes when $\omega < 0$.}

We divide into two cases depending upon 
$a_1 > \frac{\tau}{4} + \frac{1}{8}$ or $a_1 \le \frac{\tau}{4} +\frac{1}{8}$. 

\smallskip

\noindent 
\textbf{Case 1.\ $a_1 > \frac{\tau}{4} + \frac{1}{8}$.} 

By the assumption, 
$K$ satisfies the Slope Conjecture, 
hence the Jones slope $4a_1$ of $K$ is realized by a boundary slope of an essential surface $S_K \subset E(K)$.

\begin{claim}
\label{F_m^{t}_omega<0}
There exists an essential surface $F_\omega^{\tau}$ in $V - \mathrm{int}N(k_{\omega}^{\tau})$ 
such that each component of $F_\omega^{\tau} \cap \partial V$ has slope $4a_1$ and 
each component of $F_\omega^{\tau} \cap \partial N(k_{\omega}^{\tau})$ has $16a_1 -4\tau - 2- 4\omega$.  
\end{claim}

\begin{proof}
Let us take an essential surface $F_{\gamma^{-}_5}$ with $k = -\omega$ in 
$S^3 - \mathrm{int}N(k_1 \cup k_2) = V - \mathrm{int}N(k_2)$ associated to the minimal edge path $\gamma^{-}_5$ 
described in Section~\ref{essential_surfaces}. 
Then it has a pair of boundary slopes 
$(2\frac{\beta}{\alpha}, 2\frac{\alpha}{\beta} - 2 + 4k)= (2\frac{\beta}{\alpha}, 2\frac{\alpha}{\beta} - 2 -4\omega)$ on $k_1, k_2$ 
($(\alpha, \beta) \le 2$). 
$F_{\gamma^{-}_5}$ has boundary slopes $\frac{2\beta}{\alpha}$ on $\partial N(k_1)$ and 
$2\frac{\alpha}{\beta} - 2 -4\omega$ on $\partial N(k_2)$.  
Using $(\mu_V, \lambda_V)$ instead of $(\mu_1, \lambda_1)$, 
$F_{\gamma^{-}_5} \cap \partial V$ has slope $\frac{\alpha}{2\beta}$. 
Choose $\alpha, \beta$ so that 
$\frac{\alpha}{2\beta} = 4a_1 -\tau$.
Hence $\frac{\alpha}{\beta} = 8a_1 - 2\tau > 1$.   
Then $F_{\gamma^{-}_5} \subset V - \mathrm{int}N(k_2)$ has boundary slope $16a_1 -4\tau - 2 -4\omega$ on $\partial N(k_2)$ and 
$4a_1 - \tau$ on $\partial V$. 
Now we apply $\tau$--twisting along $\mu_V$ which changes 
$V - \mathrm{int}N(k_2)$ to $V - \mathrm{int}N(k_{\omega}^{\tau})$; 
we denote the image of $F_{\gamma^{-}_5}$ by $F_\omega^{\tau}$. 
By Lemma~\ref{twist} each component of $F_\omega^{\tau} \cap \partial V$ has slope $4a_1 (= 4a_1 - \tau + \tau)$,  
and each component of $F_\omega^{\tau} \cap \partial N(k_{\omega}^{\tau})$ has slope $16a_1 - 4\tau - 2 -4\omega$ as desired. 
\end{proof}

Let us take $S_\omega^{\tau} = f(F_{\omega}^{\tau}) \subset f(V - \mathrm{int}N(k_{\omega}^{\tau}))$. 
Then it is essential in $f(V - \mathrm{int}N(k_{\omega}^{\tau}))$ and each component of $S_\omega^{\tau} \cap T_K$ has slope $4a_1$ and 
each component of $S_\omega^{\tau}  \cap T_W$ has slope $16a_1 - 4\tau - 2 -4\omega$. 
Take $m$ parallel copies $m S_\omega^{\tau}$ of $S_\omega^{\tau}$ and $n$ parallel copies $n S_K$ of $S_K$,  
and then glue them along their boundaries to obtain 
a connected surface $S = m S_\omega^{\tau} \cup n S_K$ in $E(W_{\omega}^{\tau}(K))$.  
If $S$ is non-orientable, 
then we re-take $S$ as the $\partial I$--subbundle of the regular neighborhood of $S$ in $E(W_{\omega}^{\tau}(K))$, 
which is an orientable double cover of $S$. 
Note that $S_K$ and $S_\omega^{\tau}$ are orientable, 
so $S \cap E(K)$ consists of parallel copies of $S_K$, 
and similarly $S \cap f(V - \mathrm{int}N(k_{\omega}^{\tau}))$ consists of parallel copies of $S_\omega^{\tau}$. 
Since $\partial E(K)$ is incompressible in $E(W_{\omega}^{\tau}(K))$, 
and $S_\omega^{\tau}$, $S_K$ are essential in $f(V - \mathrm{int}N(k_{\omega}^{\tau}))$ and $E(K)$ respectively, 
$S$ is incompressible in $E(W_{\omega}^{\tau}(K))$. 
Since it cannot be an annulus,  
$S$ is the desired essential surface. 

\smallskip

\noindent 
\textbf{Case 2. $a_1 \le \frac{\tau}{4} +\frac{1}{8}$.}

In this case Proposition~\ref{maxdeg_KT_Wmt-allw} shows that the Jones slope is $-4\omega$. 
Take an (orientable) essential surface $F_{\gamma^{-}_6}$ with $k = -\omega$, $\alpha = 2, \beta = 0$.  
(At the end of the argument in this case, 
we explain why we need to choose $\alpha = 2, \beta = 0$ rather than $\alpha = 1, \beta = 0$.) 
A symmetry of $k_1$ and $k_2$ induces an orientation preserving homeomorphism $\varphi$ of 
$S^3 - \mathrm{int}N(k_1 \cup k_2) = V - \mathrm{int}N(k_2)$  
which exchanges the components $\partial N(k_1) = \partial V$ and $\partial N(k_2)$. 
Let us set $F = \varphi(F_{\gamma^{-}_6}) \subset V- \mathrm{int}N(k_2)$. 
Then it follows from Table~\ref{table:4k-1by8ksummary} that 
$F$ has two boundary components with boundary slopes $(\emptyset, 4k) = (\emptyset, -4\omega)$.  
We apply $\tau$--twisting along $\mu_V$ which changes 
$V - \mathrm{int}N(k_2)$ to $V - \mathrm{int}N(k_{\omega}^{\tau})$, and
we denote the image of $F$ by $F_{\omega}^{\tau}$. 
By Lemma~\ref{twist}, $F_\omega^{\tau}$ has the boundary slope $-4\omega$ on $\partial N(k_{\omega}^{\tau})$. 
Hence, $S = f(F_\omega^{\tau}) \subset f(V - \mathrm{int}N(k_{\omega}^{\tau}))$ is an essential surface such that 
$S \cap T_K = \emptyset$ and each component of $S \cap T_W$ has slope $-4\omega$. 
Thus the Jones slope $-4\omega$ is a boundary slope of $W_{\omega}^{\tau}(K)$.  
Finally we explain why we choose $\alpha = 2, \beta = 0$. 
If we choose $\alpha = 1, \beta = 0$ in the above, 
then $F_{\gamma^{-}_6}$ has a single boundary component on $\partial N(k_1)$. 
Then $S = f(F_\omega^{\tau})$ has a single boundary component on $T_W$. 
If $F_{\gamma^{-}_6}$ (with $\alpha = 1, \beta = 0$), hence $S$, 
is orientable, then its boundary slope would be $0$. 
However $S \cap T_W$ has slope  $-4\omega$, a contradiction. 
Hence $F_{\gamma_6}$ with $\alpha = 1, \beta = 0$ is non-orientable. 
The surface corresponding to $\alpha = 2, \beta = 0$ is an orientable double cover of the surface corresponding to 
$\alpha = 1, \beta = 0$. 

This completes the proof of Theorem~\ref{bothslopeconjecturesWmt}(1). 
\end{proof}

\begin{example}
\label{twisting_vs_Jone surface}

The maximum degree of 
the colored Jones function of a torus knot $K = T_{p, q}$ with relatively prime integers $p,q>0$ is explicitly computed by \cite{Garoufalidis}: 
\[
d_+[J_{K,n}(q)] =\delta_K(n)=\frac {pq} 4 n^2-\frac {pq} 4-(1+(-1)^n)\frac {(p-2)(q-2)}{8}.
\]

Note that  $d_+[J_{K,n}(q)]$ is a quadratic quasi-polynomial for all integers $n$, 
and $a(n) = a_1 = \frac {pq}{4},\ b(n) = b_1 = 0$ and $c(n) = -\frac {pq} 4-(1+(-1)^n)\frac {(p-2)(q-2)}{8}$, 
in particular $c_1 = -\frac{pq}{4}$. 
Therefore, following Proposition~\ref{maxdeg_KT_Wmt-allw} we have: 

If $\omega > 0$, then 
\begin{equation*}
\delta_{W_{\omega}^{\tau}(K)}(n)=
\left\{ \begin{array}{ll} 
(pq - \tau) n^2 +(-pq + \tau - \frac{1}{2}) n + \frac 1 2 & ( \tau < pq ) \\
             -\frac{1}{2} n - \frac{pq}{4}+ \frac{1}{2}  &  ( \tau > pq ).
       \end{array} \right. 
\end{equation*}

If $\omega < 0$, then
\begin{equation*}
\delta_{W_{\omega}^{\tau}(K)}(n)=
\left\{ \begin{array}{ll} 
             (pq- \tau -\omega - \frac{1}{2} )n^2 +(- pq + \omega + \tau + 1)n- \frac{1}{2}
                   & (\tau < pq - \frac{1}{2}) \\
            - \omega n^2 + (\omega + \frac{1}{2})n -  \frac{1}{2} 
                    &  (\tau > pq - \frac{1}{2}).
       \end{array} \right.
\end{equation*}

This shows that 
a Jones surface of a twisted generalized Whitehead double of $K$ is of a different nature depending upon the twisting number $\tau$. 
\end{example}

\section{Strong slope conjecture for twisted generalized Whitehead doubles}
\label{strong Slope_Wmt}

This section is devoted to a proof of Theorem~\ref{bothslopeconjecturesWmt}(2), 
which we state again below.

\begin{thm_strong slope conjecture Wmt}
Let $K$ be a knot in $S^3$ for which $d_+[J_{K,n}(q)]$ is the quadratic quasi-polynomial 
$\delta_K(n)=a(n) n^2 +b(n) n+c(n)$ for all $n \ge 0$.
We put $a_1:=a(1)$, $b_1:=b(1)$, and $c_1:=c(1)$. 
We assume that the period of $d_+[J_{K,n}(q)]=\delta_K(n)$ is less than or equal to  $2$ and that 
$b_1 \le 0$.
Assume further that if $b_1 = 0$, 
then $a_1 \ne \frac{\tau}{4}$. 

If $K$ satisfies the Strong Slope Conjecture with $SS(1)$, 
then all of its twisted generalized Whitehead doubles 
satisfy the Strong Slope Conjecture.
\end{thm_strong slope conjecture Wmt}

\begin{remark}
\label{YSS=Fully}
Even when $\delta_K(n)$ has period $2$, 
$\delta_{W_{\omega}^{\tau}(K)}(n)$ is a usual polynomial rather than quasi-polynomial 
\(Remark~\ref{no_period}\). 
So the Strong Slope Conjecture with $SS(1)$ is equivalent to the Strong Slope Conjecture for $W_{\omega}^{\tau}(K)$ \(Remark~\ref{constant}\). 
\end{remark}

\begin{proof}
Write $\delta_{W_{\omega}^{\tau}(K)}(n)=a_W(n)n^2 + b_W(n)n + c_W(n)$.  
It follows from Remark~\ref{no_period} that coefficients of $\delta_{W_{\omega}^{\tau}(K)}(n)$ are constants 
and so we may write $a_W(n) = a_W$, $b_W(n) = b_W$, $c_W(n) = c_W$. 
Then we show that essential surfaces $S$ in $E(W_{\omega}^{\tau}(K))$ given in the proof of 
Theorem~\ref{bothslopeconjecturesWmt}(1) 
satisfy the condition of the Strong Slope Conjecture: 
\[ S\ \textrm{has boundary slope}\ p/q = 4a_W \quad and \quad \frac{\chi(S)}{|\partial S| q} = 2b_W.\]

\smallskip

It is convenient to note the following general fact.  

\begin{lemma}
\label{double}
Let $F$ be a properly embedded surface in a knot exterior $E$ such that a component of $\partial F$ has slope $p/q$. 
Let $\widetilde{F}$ be the frontier of a tubular neighborhood $N(F)$ in $E$, 
i.e.\  $\widetilde{F}$ is the $\partial I$--subbundle of an $I$--bundle over $F$. 
Then $\partial \widetilde{F}$ has slope $p/q$ and 
\[ \frac{\chi(\widetilde{F})}{|\partial \widetilde{F}| q} =  \frac{\chi(F)}{|\partial F| q}. 
\]
\end{lemma}

\begin{proof}
If $F$ is orientable, 
then $N(F) = F \times I$, 
whose frontier $\widetilde{F}$ consists of two copies of $F$. 
So each component of $\partial F$ has slope $p/q$ and 
$\displaystyle \frac{\chi(\widetilde{F})}{|\partial \widetilde{F}| q} =  \frac{2\chi(F)}{2|\partial F| q} 
=  \frac{\chi(F)}{|\partial F| q}$. 
Assume now that $F$ is non-orientable.
Then $\widetilde{F}$ is the orientable double cover of $F$, 
 and $\bdry \widetilde{F}$ consists of two parallel loops with slope $p/q$. 
Hence 
$\displaystyle \frac{\chi(\widetilde{F})}{|\partial \widetilde{F}| q} 
=  \frac{2\chi(F)}{2|\partial F| q} 
=  \frac{\chi(F)}{|\partial F| q}$. 
\end{proof}

\noindent
\textbf{1. Jones surfaces for $\omega > 0$.}

\noindent 
\textbf{Case 1-1.\ $a_1 > \frac{\tau}{4}$.}

Write $a_1= r/s$ where $r$ and $s$ are coprime integers and $s>0$.   
Then, as a ratio of coprime integers, the denominator of $4a_1$ is $s/\gcd(4,s)$.
Since $K$ satisfies the Strong Slope Conjecture with $SS(1)$, 
there is a properly embedded essential surface $S_K$ in the exterior of $K$ whose boundary slope is $4a_1$ and 
\[\frac{\chi(S_K)}{|\bdry S_K| \cdot \frac{s}{\gcd(4,s)}} = 2 b_1. \]
When addressing the Slope Conjecture for $W^\tau_\omega(K)$ in this case, 
we constructed a properly embedded essential surface $S = m S_K \cup n S^\tau_\omega$ 
in the exterior of $W^\tau_\omega(K)$ 
by joining $m$ copies of $S_K$ in $E(K)$ to $n$ copies of the surface $S^\tau_\omega$  in $V - N(k^\tau_\omega)$.  
This requires that 
\[ m |\bdry S_K| = n |\bdry S^\tau_\omega \cap T_K|. \]
 
The surface $S^\tau_\omega$ is identified with a surface of type $F_{\gamma_1}$ in the exterior of the 
$[2,2\omega,-2]$ two-bridge link, 
where $\frac{\alpha}{\beta} = 8a_1 - 2\tau = \frac{8r -2\tau s}{s} > 0$ so that $S^\tau_\omega$ has boundary slope $4a_1$ on $\bdry V$.  
 We choose $\beta = 2s$, $\alpha = 2(8r - 2\tau s)$ 
so that $F_{\gamma_1} = F_{\gamma_1, \alpha, \beta}$ is orientable; see Subsection~\ref{algorithm}. 
 
Then, using Table~\ref{table:4k-1by8ksummary}, 
we calculate the following:
\begin{itemize}
\item $\chi(S^\tau_\omega) 
= -\alpha -\beta = -2(8r - (2\tau-1)s)$,

\item  slope of $\bdry S^\tau_\omega$  on $T_W$ is 
$2\frac{\alpha}{\beta} 
= 2( \frac{8r-2 \tau s}{s}) 
= \frac{16r-4\tau s}{s}$,  

\item $|\bdry S^\tau_\omega \cap T_K| 
= \gcd(2\beta, \alpha) 
= \gcd(4s,2(8r-2\tau s)) 
= 4 \gcd(4,s)$, and

\item $|\bdry S^\tau_\omega \cap T_W| 
= \gcd(2\alpha, \beta) 
= \gcd(4(8r-2\tau s), 2s) 
= 2\gcd(16,s)$.
\end{itemize}

The boundary of $S$ consists of $n$ copies of the boundary of $S^\tau_\omega$  on $T_W$, 
so we have

\begin{itemize}
\item 
$|\bdry S| = n | \partial S^{\tau}_{\omega} \cap T_W | = 2n \gcd(16,s)$. 
\end{itemize}

Moreover, the boundary slope of $S$ is the  slope of $\bdry S^\tau_\omega$  on $T_W$, 
and so this has denominator $\frac{s}{\gcd(16r-4\tau s,\, s)}  = \frac{s}{\gcd(16,\,s)}$.  
We may now calculate
\begin{align*}
  \frac{\chi(S)}{|\bdry S| \cdot \frac{s}{\gcd(16,\,s)}} 
  &= \frac{m \chi(S_K) + n \chi(S^\tau_\omega)}{ 2n \gcd(16,s) \cdot \frac{s}{\gcd(16,\,s)}} \\
  &= \frac{ 2 b_1 m |\bdry S_K| \cdot \frac{s}{\gcd(4,\,s)} -2n(8r - (2\tau-1)s)}{2n s} \\
  &= \frac{ 8 b_1 n \gcd(4,s) \cdot \frac{s}{\gcd(4,\,s)} -2n(8r - (2\tau-1)s)}{2n s} \\
  &= \frac{ 8 b_1 n s -2n(8r - (2\tau-1)s)}{2n s} \\
    &= 4 b_1 -8r/s+(2\tau-1) = 2(-4a_1 + 2b_1 +  \tau - \frac{1}{2}) = 2b_W
\end{align*}
    as desired.

If the glued surface $S = m S_K \cup n S^\tau_\omega$ is non-orientable, 
then as in the proof of Theorem~\ref{bothslopeconjecturesWmt}(1), 
we replace $S$ by the frontier $\widetilde{S}$ of the tubular neighborhood of $S$, 
but Lemma~\ref{double} shows that $\widetilde{S}$ and $S$ has the same boundary slope and 
$\displaystyle \frac{\chi(\widetilde{S})}{|\bdry \widetilde{S}| \cdot \frac{s}{\gcd(16,\,s)}} 
= \frac{\chi(S)}{|\bdry S| \cdot \frac{s}{\gcd(16,\,s)}}$. 
Thus the essential surface $S$ or $\widetilde{S}$ (when $S$ is non-orientable) is the desired 
essential surface.

\smallskip

\noindent 
\textbf{Case 1-2.\ $a_1 \le \frac{\tau}{4}$.}

In this situation, 
$S$ is a minimal genus Seifert surface of $W_{\omega}^{\tau}(K)$, 
which is a once punctured torus. 
Hence 
\[ \frac{\chi(S)}{| \partial S | q} 
= \frac{\chi(S)}{| \partial S |}  = \frac{-1}{1} = -1 = 2 (-\frac{1}{2}) 
\in jx_{W_{\omega}^{\tau}(K)}.\]

\medskip

\noindent
\textbf{2. Jones surfaces for $\omega < 0$.}

\noindent 
\textbf{Case 2-1.\ $a_1  > \frac{\tau}{4} +\frac{1}{8}$.}

We follow the same argument in Case 1-1. 
Write $a_1= r/s$ for some coprime $r$ and $s > 0$.   
Then, as a ratio of coprime integers, the denominator of $4a_1$ is $s/ \gcd(4,s)$.
Since $K$ satisfies the Strong Slope Conjecture with $SS(1)$,  
there is a properly embedded essential surface $S_K$ in the exterior of $K$ whose boundary slope is $4a_1$ and 
\[
\frac{\chi(S_K)}{|\bdry S_K| \cdot \frac{s}{\gcd(4,\,s)}} = 2 b_1. 
\]
When addressing the Slope Conjecture for $W^\tau_\omega(K)$ in this case, 
we constructed a properly embedded essential surface $S = m S_K \cup n S^{\tau}_\omega$
in the exterior of $W^\tau_\omega(K)$ 
by joining $m$ copies of $S_K$ in $E(K)$ to $n$ copies of the surface $S^{\tau}_\omega$  in $V - N(k^\tau_\omega)$.  
This requires that 
\[ m |\bdry S_K| = n |\bdry S^{\tau}_\omega \cap T_K|. \]
 
The surface $S^{\tau}_\omega$ is identified with a surface of type $F_{\gamma^{-}_5}$ (with $k = -\omega$) in the exterior of the $[2,2\omega,-2]$ two-bridge link where 
$\frac{\alpha}{\beta} = 8a_1 - 2\tau = \frac{8r- 2\tau s}{s} > 1$ 
so that $S^{\tau}_\omega$ has boundary slope $4a_1$ on $\bdry V$.  
We choose $\beta = 2s$ and $\alpha = 2(8r - 2\tau s)$, 
so that $F_{\gamma^{-}_5} = F_{\gamma^{-}_5, \alpha, \beta}$ is orientable; see Subsection~\ref{algorithm}.
Then, using Table~\ref{table:4k-1by8ksummary}, we calculate the following:
\begin{itemize}
\item 
$\chi(S^{\tau}_\omega) 
= -\alpha + 2(1-k)\beta 
= -\alpha + 2(1+\omega)\beta 
= -16r +4(\tau +\omega +1)s$, 

\item  
slope of $\bdry S^{\tau}_\omega$  on $T_W$ is 
$2\frac{\alpha}{\beta} - 2 +4k 
= 2\frac{\alpha}{\beta} - 2 -4 \omega
= 2( \frac{8r-2 \tau s}{s}) -2-4\omega
= \frac{16r-4\tau s -2s -4\omega s}{s}$,  

\item 
$|\bdry S^{\tau}_\omega \cap T_K| 
= \gcd(2\beta, \alpha) 
= \gcd(4s, 16r-4\tau s) 
= 4 \gcd(s, 4)$, and

\item 
$|\bdry S^{\tau}_\omega \cap T_W| 
= \gcd(2\alpha, \beta) 
= \gcd(32r-8\tau s, 2s) 
= 2\gcd(16, s)$.
\end{itemize}
The boundary of $S$ consists of $n$ copies of the boundary of $S^{\tau}_\omega$ on $T_W$, 
so we have 
\begin{itemize}
\item 
$|\bdry S| = n | \partial S^{\tau}_{\omega} \cap T_W | = 2n \gcd(16,s)$. 
\end{itemize}
Moreover, the boundary slope of $S$ is the slope of $\bdry S^{\tau}_\omega$ on $T_W$, 
and so this has denominator 
$\frac{s}{\gcd(16r -4\tau s -2s-4\omega s,\, s)}  
= \frac{s}{\gcd(16,\,s)}$.  
We may now calculate
\begin{align*}
  \frac{\chi(S)}{|\bdry S| \cdot \frac{s}{\gcd(16,\,s)}} 
  &= \frac{m \chi(S_K) + n \chi(S^{\tau}_\omega)}{ 2n \gcd(16,s) \cdot \frac{s}{\gcd(16,\,s)}} \\
  &= \frac{ 2 b_1 m |\bdry S_K| \cdot \frac{s}{\gcd(4,\,s)} + n (-16r +4(\tau +\omega +1)s)}{2n s} \\
  &= \frac{ 8 b_1 n \gcd(4,s) \cdot \frac{s}{\gcd(4,\,s)} + n (-16r +4(\tau +\omega +1)s)}{2n s} \\
  &= \frac{ 8 b_1 n s+ n (-16r +4(\tau +\omega +1)s)}{2n s} \\
  &= 4 b_1 - 8r/s +2(\tau + \omega +1) \\
  &= -8a_1 +4 b_1 +2\tau + 2\omega +2 \\
  &= 2(-4a_1 + 2b_1 +  \tau + \omega +1) \\
  &= 2b_W
\end{align*}
as desired.

If the glued surface $S = m S_K \cup n S^{\tau}_\omega$ is non-orientable, 
then we replace $S$ by the frontier $\widetilde{S}$ of the tubular neighborhood of $S$. 
By Lemma~\ref{double} $\widetilde{S}$ and $S$ has the same boundary slope and 
$\displaystyle \frac{\chi(\widetilde{S})}{|\bdry \widetilde{S}| \cdot \frac{s}{\gcd(16,\,s)}} 
= \frac{\chi(S)}{|\bdry S| \cdot \frac{s}{\gcd(16,\,s)}}$. 
Thus the essential surface $S$ or $\widetilde{S}$ (when $S$ is non-orientable) is the desired 
essential surface.

\medskip 

\noindent 
\textbf{Case 2-2.\ $a_1 \le \frac{\tau}{4}+\frac{1}{8}$.}

In this case, 
the argument in Case $2$ in the proof of Theorem~\ref{bothslopeconjecturesWmt}(1) and 
Table~\ref{table:4k-1by8ksummary} show that 
$|\partial S| = 2$ and $\chi(S) = (1-2k)\cdot 2 = (1+2\omega)\cdot 2 = 4\omega + 2$.

Note that $\partial S \cap T_W$ consists of two simple closed curves each of which has slope $-4\omega$
($\partial S \cap T_K = \emptyset$). 
Hence 
\[
 \frac{\chi(S)}{| \partial S | q} = \frac{\chi(S)}{| \partial S |}  
= \frac{4\omega+2}{2} = 2\omega+1 = 2 (\omega + \frac{1}{2}) = 2 b_W.
\]
Thus $S$ satisfies the required condition. 

This completes the proof of Theorem~\ref{bothslopeconjecturesWmt}(2). 
\end{proof}

\medskip

\begin{proof}[Proof of Theorem~\ref{slopeconjectureWmtSignCond}]
Apply Proposition~\ref{maxdeg-signcond} and the argument in the proof of Theorem~\ref{bothslopeconjecturesWmt}.
\end{proof}

\section{Proof of Corollary~\ref{finite-sequence}}
\label{corollary}

The aim of this section is to prove Corollary~\ref{finite-sequence}. 

First we review notions of adequacy for knots.
Let $D$ be a diagram of a knot in $S^3$. 
A \textit{state} for $D$ is a choice of replacing every crossing of $D$ by \textit{$A$--resolution} or 
\textit{$B$--resolution} as in Figure~\ref{fig:resolution} with the dotted segment recording the location of 
the crossing before the replacement. 
The state $\sigma_{+}$ ( resp. $\sigma_{-}$ ) denotes the choice of $A$ (resp. $B$) 
-resolution at each crossing of $D$. 
Applying a state to $D$, we obtain a set of disjoint circles called \textit{state circles}. 
We form \textit{$\sigma$-state graph} $G_{\sigma}(D)$ for a state $\sigma$ 
by letting the resulting circles be vertices and the dotted segments be edges. 

\quad

\begin{figure}[!ht]
\includegraphics[width=0.5\linewidth]{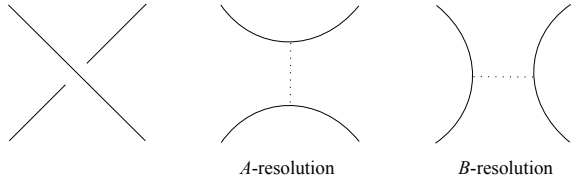}
\caption{$A$--resolution and $B$--resolution}
\label{fig:resolution}
\end{figure}

\quad

\begin{definition} 
A diagram $D$ is \textit{$A$--adequate} (resp. \textit{$B$--adequate}) 
if the graph $G_{\sigma_+}(D)$ ( resp. $G_{\sigma_-}(D)$ ) has no one-edged loop. 
If $D$ is both $A$--adequate and  $B$--adequate,  then $D$ is called \textit{adequate}.  
A knot is \textit{$A$--adequate} (resp. \textit{$B$--adequate}) 
if it has a $A$--adequate diagram (resp. $B$--adequate diagram).
A knot  is \textit{adequate} if it has an adequate diagram. 
\end{definition}

\begin{cor_finite-sequence}
Any knot obtained by a finite sequence of cabling, 
untwisted $\omega$--generalized Whitehead doublings with $\omega > 0$ 
and connected sums of $B$--adequate knots or torus knots satisfies the Strong Slope Conjecture. 
\end{cor_finite-sequence}

This result follows from a more general proposition below (Proposition~\ref{induction}), 
for which we introduce the following technical condition. 

\begin{definition}
\label{condition}
We say that $K$ satisfies \textit{Condition $\delta$} if 
\begin{enumerate}
\item
$\delta_K(n)=an^2 + bn + c(n)$ has period at most $2$, 
\item
$b \le 0$, 
\item $4a \in \mathbb{Z}$, and 
\item $b=0 \implies a \neq 0$.
\end{enumerate}
\end{definition}

\begin{remark}
This Condition $\delta$ is slightly stronger than the Condition $\delta$ in \cite{BMT_graph}.  
They are the same except for the addition of item (4).
\end{remark}

\begin{definition}
\label{kappa}
Let $\mathcal{K}$ be the maximal set of knots in $S^3$ of which each is either the trivial knot or satisfies the Sign Condition, Condition $\delta$, 
and the Strong Slope Conjecture.
\end{definition}

\begin{example} \label{Kexamples}
Torus knots and $B$--adequate knots belong to $\mathcal{K}$.
\begin{itemize}
\item 
The trivial knot is in $\mathcal{K}$ by definition.
\item
Any nontrivial torus knot satisfies the Strong Slope Conjecture and Condition $\delta$ via \cite[Theorem 3.9]{KT}.   
More precisely, 
if $K$ is a positive torus knot, then $0 < 4a \in \mathbb{Z}$ and $b = 0$. 
If $K$ is a negative torus knot, 
then $a = 0$ and $b < 0$. 
Furthermore, it follows from \cite[Proposition~4.3]{BMT_graph} that any torus knot satisfies the Sign Condition.  
Hence torus knots are in $\mathcal{K}$. 

\item
Any nontrivial $B$--adequate knot satisfies the Strong Slope Conjecture and Condition $\delta$ via \cite{FKP, FKP2} and \cite[Lemma 3.6, 3.8]{KT}.  
More precisely, 
$0 \le 4a \in \mathbb{Z},  b \le 0$. 
See \cite[Lemma 3.6]{KT}. 
If $b = 0$, then $K$ is a torus knot and $a > 0$ (\cite[Lemma~3.8]{KT}).
Furthermore, it follows from \cite[Proposition~4.4]{BMT_graph} that any $B$--adequate knot satisfies the Sign Condition.  
Hence $B$--adequate knots are in $\mathcal{K}$. 
\end{itemize}
\end{example}

\begin{proposition}
\label{induction}
The set $\mathcal{K}$ is closed under connected sum, 
cabling and untwisted $\omega$--generalized Whitehead doubling with $\omega > 0$.  
\end{proposition}

To prove this proposition, we prepare some lemmas. 

\begin{lemma}
\label{connected sum}
Assume that $K_i \in \mathcal{K}$. 
Then $K_1 \sharp K_2 \in \mathcal{K}$. 
\end{lemma}

\begin{proof}
Since the trivial knot is the identity for the connected sum operation, we may assume neither $K_1$ nor $K_2$ is trivial.  
Then by \cite[Lemma~4.1]{BMT_graph} we only need to see that (4) in Condition $\delta$ holds. 
Write $\delta_{K_i}(n) = a_i n^2 + b_i n + c_i(n)$.  
Then 
\begin{align*}
\delta_{K_1 \sharp K_2}(n) 
&= \delta_{K_1}(n) + \delta_{K_2}(n) - \frac{1}{2}n + \frac{1}{2} \\
&= (a_1 + a_2)n^2 + (b_1 + b_2 -\frac{1}{2}) n + (c_1(n) + c_2(n) + \frac{1}{2}). 
\end{align*}
Since $b_i \le 0$,  we have
$b_1 + b_2 -\frac{1}{2} < 0$.

Hence (4) obviously holds for $K_1 \sharp K_2$.  
\end{proof}

\begin{lemma}
\label{cabling}
Assume that $K \in \mathcal{K}$. 
Then its $(p, q)$--cable $K_{p, q}$ \($q > 1$\) belongs to $\mathcal{K}$. 
\end{lemma}

\begin{proof}
If $K$ is trivial, then its cables $K_{p,q}$ are torus knots. 
By Example~\ref{Kexamples},  
these belong to $\mathcal{K}$.  
Thus we may assume $K$ is non-trivial. 

By \cite[Lemma~4.2]{BMT_graph} we only need to see that (4) in Condition $\delta$ holds. 
If $\delta_{K}(n) = a n^2 + b n + c(n)$, 
then we have \cite[Proposition 3.1]{BMT_graph}: 
\begin{equation*}
  \delta_{K_{p,q}}(n)=
   \left\{ \begin{array}{ll} 
           q^2 a n^2 +\left(q b+\frac {(q-1)(p-4qa)}{2} \right) n & \\
          \quad\quad +\left(a(q-1)^2-(b+\frac p 2)(q-1)+c(i)\right) & \mbox{for } \frac p q< 4a, \\
          \frac {pq(n^2-1)} 4+C_j(K_{p,q})  &  \mbox{for }\frac p q \geq 4a.
          \end{array} \right. 
\end{equation*}
where $i \equiv_{(2)} q(n-1)+1$, $j \equiv_{(2)} n$, 
and $C_j(K_{p,q})$ 
is a number that only depends on the knot $K$ and the numbers $p$ and $q$.

Assume first that $\frac{p}{q} < 4a$. 
If $q^2a = 0$, then $a=0$.  Thus $p/q < 0$ and $p < 0$ (because $q > 1$).
Then
$q b+\frac {(q-1)(p-4qa)}{2} = qb + \frac{p(q-1)}{2} < 0$ (because $b \le 0$). 
Hence, 
if $q b+\frac {(q-1)(p-4qa)}{2} = 0$, then $q^2a \ne 0$. 
In the case where $\frac{p}{q} \ge 4a$, 
the linear therm is $0$, but the quadratic term $\frac{pq}{4}$ is not $0$. 

Therefore $K_{p, q}$ also satisfies Condition $\delta$. 
\end{proof}

\begin{lemma}
\label{doubling}
Assume that $K \in \mathcal{K}$. 
Then its untwisted $\omega$--generalized Whitehead double $W_{\omega}^0(K)$ 
with $\omega >0$ also belongs to $\mathcal{K}$. 
\end{lemma}

\begin{proof}
Note that an untwisted $\omega$--generalized Whitehead double of the trivial knot is again a trivial knot.  
So we may assume $K$ is non-trivial. 

Write $\delta_{K}(n) = a n^2 + b n + c(n)$. 
Since $K$ satisfies Condition $\delta$, 
$\delta_{K}(n)$ has period at most $2$ and 
we have $b \le 0$, $4a \in \mathbb{Z}$, 
and $a \ne 0$ if $b = 0$.

Since $\omega \in \Z$,   
we may assume $\omega \geq 1$. 
Since $\tau = 0$, 
Propositions~\ref{maxdeg-signcond} show that 

\begin{equation*}
  \delta_{W_{\omega}^{0}(K)}(n)=
   \left\{ \begin{array}{ll} 
            4a n^2 +(- 4a+2b - \frac{1}{2}) n+a - b + c_1+\frac{1}{2} & (a > 0), \\
             -\frac 1 2 n+C_+(K, 0)+\frac{1}{2}  &  (a < 0), \\
             -\frac{1}{2}n + C_+'(K, 0) + \frac{1}{2} &  (a = 0, b \ne 0),
         \end{array} \right. 
\end{equation*}

We now check Condition $\delta$ for $W_{\omega}^{0}(K)$. 
Obviously (1) in Condition $\delta$ holds; actually $ \delta_{W_{\omega}^{0}(K)}(n)$ is a usual polynomial rather than quasi-polynomial. 
Since $4(4a)$, $0$ are integers, 
we have (3) in Condition $\delta$. 

To address (2) and (4) in Condition $\delta$ we  first examine the coefficients of $\delta_{W_{\omega}^{0}(K)}(n)$. 
When $a > 0$,  
since $b \le 0$ (by Condition $\delta$ for $K$), we have
$-4a + 2b - \frac{1}{2} < 0$. 
Thus the coefficient of linear term of $\delta_{W_{\omega}^{0}(K)}(n)$ is negative; 
in particular it is not $0$. 
When $a \le 0$, 
the coefficient of linear term of $\delta_{W_{\omega}^{0}(K)}(n)$ is $-\frac{1}{2} < 0$. 
Hence (2) and (4) in Condition $\delta$ hold. 

Since $K$ satisfies the Sign Condition, the last assertion in Proposition~\ref{maxdeg-signcond} shows that 
$W_{\omega}^{0}(K)$ also satisfies the Sign Condition. 
So it remains to show that $W_{\omega}^{0}(K)$ enjoys the Strong Slope Conjecture, 
which is equivalent in this case to the Strong Slope Conjecture with $SS(1)$ by Remark~\ref{constant}. 
Note that since $K$ satisfies the Strong Slope Conjecture and Condition $\delta$, 
it also satisfies the Strong Slope Conjecture with $SS(1)$; again,
see Remark~\ref{constant}. 
Since $\tau = 0$, 
to apply Theorem~\ref{slopeconjectureWmtSignCond} we need to check extra conditions: 
if $b = 0$, 
then $a \ne \frac{\tau}{4} = 0$.  

If $b = 0$, then by Condition $\delta$ for $K$, 
we have $a \ne 0$.  
Hence, Theorem~\ref{slopeconjectureWmtSignCond} shows that $W_{\omega}^{0}(K)$ satisfies the Strong Slope Conjecture. 
\end{proof}

\begin{proof}[Proof of Proposition~\ref{induction}]
The proof follows from Lemmas~\ref{connected sum}, \ref{cabling} and \ref{doubling}. 
\end{proof}

\begin{proof}[Proof of Corollary~\ref{finite-sequence}]
Since the trivial knot, torus knots, and $B$--adequate knots belong to $\mathcal{K}$ by Example~\ref{Kexamples}, 
the proof follows from Proposition~\ref{induction}. 
\end{proof}

As we mentioned in Section~\ref{section:Introduction} 
adequate knots satisfy the assumption in Theorem~\ref{bothslopeconjecturesWmt}. 
Recall that a positive torus knot is $A$--adequate and a negative torus knot is $B$--adequate \cite[Example 9]{FKP}.  
We say that a knot $K$ is \textit{inadequate} if it is neither $A$--adequate nor $B$--adequate.
We close this section by observing that $\mathcal{K}$ contains infinitely many inadequate knots.

\begin{example}
\label{inadequate}
Let $T_{p, q}$ and $T_{p', -q'}$ be torus knots with odd integers $p, q, p', q' > 1$. 
Then a connected sum $T_{p, q}\, \sharp\, T_{p', -q'}$ is inadequate. 
Following Corollary~\ref{finite-sequence} $T_{p, q} \,\sharp\, T_{p', -q'}$ belongs to $\mathcal{K}$. 
\end{example}

\begin{proof}
Let us denote the minimum degree of $J'_{K,n}(v)$ by $d_-[J'_{K,n}(v)]$. 
(Here we use the letter ``$v$'' instead of ``$q$'' for variables of normalized colored Jones functions to avoid notational confusion.)
Assume for a contradiction that $T_{p, q}\, \sharp\, T_{p', -q'}$ is $A$--adequate or $B$--adequate. 
Thus $T_{p, q}\, \sharp\, T_{p', -q'}$ admits a diagram $D$ with $c_{+}$ positive crossings and $c_{-}$ negative crossings, 
which is $A$--adequate or $B$--adequate. 
Then \cite[Lemma~6]{FKP} shows that 
if $D$ is $A$--adequate, then $d_{-}[J'_{T_{p, q}\, \sharp\, T_{p', -q'}, n}(v)] = - \frac{c_{-}}{2} n^2 + O(n)$, 
and if $D$ is $B$--adequate, 
then $d_{+}[J'_{T_{p, q}\, \sharp\, T_{p', -q'}, n}(v)] =  \frac{c_{+}}{2} n^2 + O(n)$, 
where $O(n)$ denotes a term that is at most linear in $n$. 
To apply this result let us compute $d_{+}[J'_{T_{p,q}\, \sharp\, T_{p', q'}, n}(v)]$ and $d_{-}[J'_{T_{p,q}\, \sharp\, T_{p', q'}, n}(v)]$.
Garoufalidis \cite[Section~4.8]{Garoufalidis} computes normalized colored Jones functions of $T_{p, q}$ $(p, q > 0)$ explicitly: 
\begin{eqnarray*}
 && 
d_{+}[J'_{T_{p,q}, n}(v)] = \frac{pq}{4}n^2 - \frac{1}{2}n - \frac{pq-2}{4} - (1+(-1)^n)\frac{(p-2)(q-2)}{8},\\
&& d_{-}[J'_{T_{p,q}, n}(v)] = \frac{(p-1)(q-1)}{2}n - \frac{(p-1)(q-1)}{2}. 
\end{eqnarray*}
Since $T_{p', -q'}=\overline{T_{p', q'}}$, we have
\begin{eqnarray*}
 && 
d_{+}[J'_{T_{p', -q'}, n}(v)] 
= - d_{-}[J'_{T_{p', q'}, n}(v)] 
= -\frac{(p'-1)(q'-1)}{2}n + \frac{(p'-1)(q'-1)}{2},\\
&& d_{-}[J'_{T_{p', -q'}, n}(v)] 
= - d_{+}[J'_{T_{p', q'}, n}(v)] 
=  -\frac{p'q'}{4}n^2 + \frac{1}{2}n + \frac{p'q'-2}{4} + (1+(-1)^n)\frac{(p'-2)(q'-2)}{8}. 
\end{eqnarray*}

Hence 
\[
d_{+}[J'_{T_{p,q}\, \sharp\, T_{p', q'}, n}(v)] 
= d_{+}[J'_{T_{p,q}, n}(v)] + d_{+}[J'_{T_{p' ,q'}, n}(v)]
= \frac{pq}{4} n^2 + O(n)
\]
and
\[
d_{-}[J'_{T_{p,q}\, \sharp\, T_{p', q'}, n}(v)] 
= d_{-}[J'_{T_{p,q}, n}(v)] + d_{-}[J'_{T_{p' ,q'}, n}(v)]
= -\frac{p'q'}{4} n^2 + O(n)
\]

Since $p, q, p', q' > 1$ are odd integers, 
$\frac{pq}{4},  -\frac{p'q'}{4} \not \in \frac{1}{2}\mathbb{Z}$. 
This is a contradiction. 
\end{proof}

\begin{question}
Is a twisted generalized Whitehead double of an inadequate knot also inadequate?
\end{question}

\section{Non-adequate Whitehead doubles}
\label{non-adequate}

As shown by Kalfagianni-Tran \cite{KT} and Futer-Kalfagianni-Purcell \cite{FKP, FKP2} (cf.\ \cite{Ozawa}), 
adequate knots and their iterated cables satisfy the Strong Slope Conjecture. 
In some cases a Whitehead double is a $B$--adequate knot. 

\begin{proposition}
\label{minus_adequate_W}
Let $K$ be a knot which has a $B$--adequate diagram with non-negative writhe. 
The its untwisted negative Whitehead double is also $B$--adequate.
\end{proposition}

\begin{proof}
Let $K$ be a knot which has a $B$--adequate diagram $\mathcal{D}(K)$ whose writhe is not negative; 
$\mathcal{D}_-(K)$ denotes a diagram obtained from $\mathcal{D}(K)$ by applying $B$--resolution at each crossing. 
Then we have a diagram $\mathcal{D}(W_{1}^{0}(K))$ of $W_{1}^{0}(K)$ as in the Figure~\ref{fig:adequate_tgW}. 
Apply $B$--resolution at each crossing to $\mathcal{D}(W_{1}^{0}(K))$ to obtain a diagram $\mathcal{D}^2_-(K) \cup C$,  
where $\mathcal{D}^2(K)$ denotes a diagram obtained from $\mathcal{D}(K)$ by replacing each of its component by $2$ parallels and $\mathcal{D}^2_-(K)$ denotes a diagram obtained from $\mathcal{D}^2(K)$ by applying $B$--resolution at each crossing. 
Since $\mathcal{D}(K)$ is $B$--adequate, 
\cite[Lemma 2.17]{Le-lec} implies that $\mathcal{D}^2(K)$ is also $B$--adequate, 
and hence the $\sigma_{-}$--state graph associated to $\mathcal{D}_{-}^2(K)$ has no loop edge. 
(In \cite{Le-lec} Le use the terminology \textit{plus}-adequate (resp.\ \textit{minus}-adequate) to mean $A$--adequate 
(resp. $B$--adequate).)
Then it is easy to see that the graph associated to $\mathcal{D}^2_-(K) \cup C$ has no loop edge, 
and hence $\mathcal{D}(W_{1}^{0}(K))$ is $B$--adequate as well. 
Thus $W_{1}^{0}(K)$ is $B$--adequate. 
\end{proof}

\begin{figure}[h]
	\centering
	\includegraphics[width=4in]{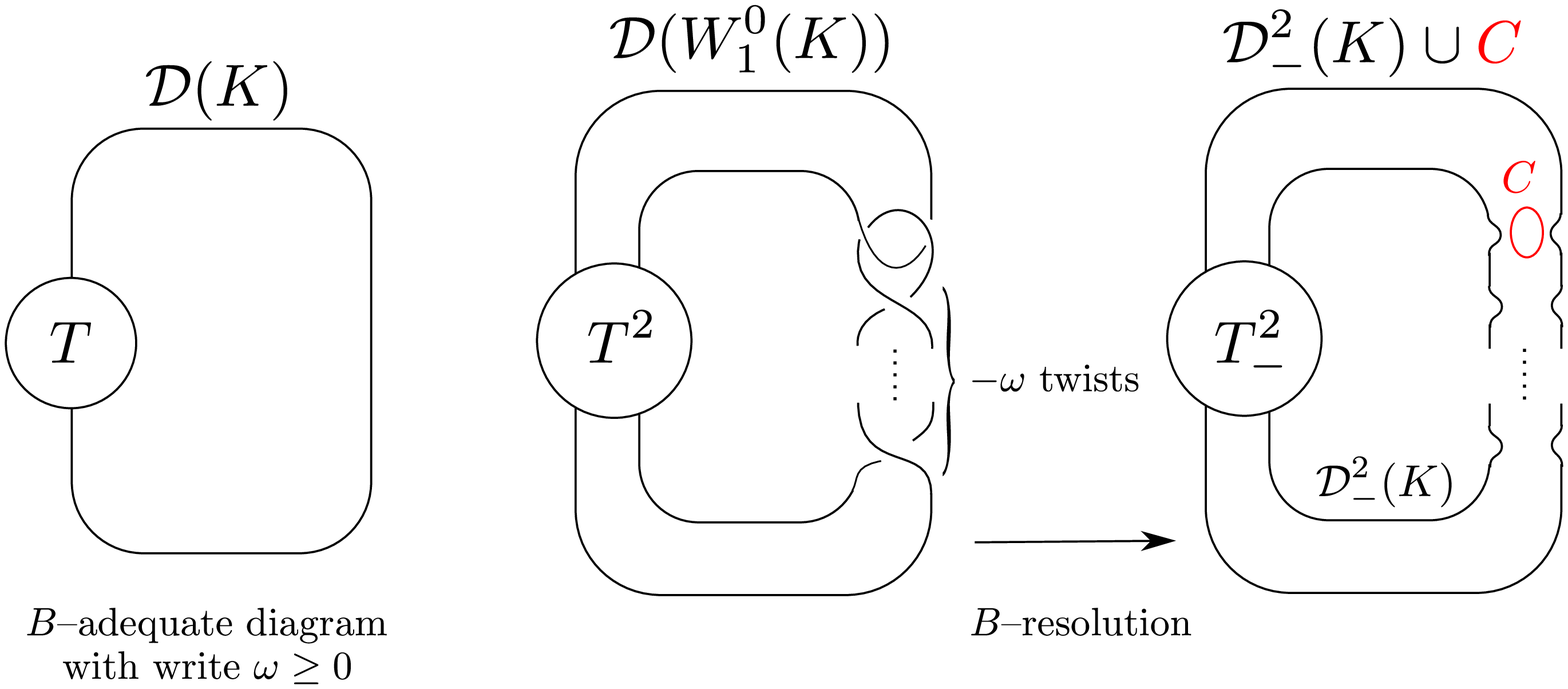}
	\caption{}
	\label{fig:adequate_tgW}
\end{figure}
 
We may expect that most Whitehead doubles are not adequate. 
However, to the best our knowledge, 
there are no explicit such examples. 
So for completeness we give explicit family of Whitehead doubles which are not adequate. 
Recall that $W_1^0(K)$ is the (untwisted) negative Whitehead double of $K$.
In the following, for notational simplicity, 
we denote $W_1^0(K)$ by $W_{-}(K)$. 
We also denote the (untwisted) positive Whitehead double of $K$ by $W_{+}(K)$, which may be written as $W_{-1}^0(K)$.

\begin{theorem}
\label{non-adequate_W}
Let $K$ be the  torus knot $T_{2,-(2m+1)}$ for $m \ge 2$.  
Then $W_{-}(K)$ is not adequate.  
\end{theorem}

Since $K = T_{2,-(2m+1)}$ is alternating, it is adequate. 
So this result shows that even when $K$ is adequate, 
its Whitehead double $W_{-}(K)$ may not be adequate.  

To prove Theorem~\ref{non-adequate_W}, 
we prepare two lemmas below. 
Let us denote the Turaev genus of $K$ by $g_T(K)$,  
which was introduced by Turaev \cite{Turaev}, 
and denote the minimal crossing number of $K$ by $c(K)$. 

\begin{lemma}
\label{Turaev_genus_1}
If $W_{-}(T_{p, -q})$ is adequate for $p, q > 0$, 
then $g_T(W_{-}(T_{p, -q})) = 1$. 
\end{lemma}

\begin{proof}
Since $T_{p, -q}=\overline{T_{p, q}}$, 
\begin{equation*}
 \delta_{T_{p, -q}}(n) = - \delta^{*}_{T_{p, q}}(n) =-\frac {pq-p-q} 2 n+\frac {pq-p-q} 2,
\end{equation*}
and then by Proposition~\ref{maxdeg_KT_Wmt-allw}, we obtain: 
\begin{equation*}
  \delta_{W_-(T_{p, -q})}(n)=-\frac{1}{2} n + \frac {pq-p-q} 2 + \frac{1}{2}. 
\end{equation*}
Moreover, since 
\begin{equation*}
 \delta_{T_{p, q}}(n) =\frac {pq} 4 n^2-\frac {pq} 4-(1+(-1)^n)\frac{(p-2)(q-2)}{8},
\end{equation*}
 by Propositions~\ref{maxdeg_KT_Wmt-allw} and 
 $\delta^{*}_{W_-(T_{p, -q})}(n)=-\delta_{\overline{W_-(T_{p, -q})}}=-\delta_{W_{-1}^0(T_{p, q})}(n)$, 
we obtain: 
\begin{equation*}
  \delta^{*}_{W_-(T_{p, -q})}(n)=-(pq+\frac{1}{2})n^2+pq n+ \frac{1}{2}. 
\end{equation*}
Then, it follows that 
\begin{equation}
\delta_{W_-(T_{p, -q})}(n)- \delta^{*}_{W_-(T_{p, -q})}(n)
=(pq+\frac{1}{2})n^2-(pq+\frac{1}{2})n +  \frac {pq-p-q} 2.
\end{equation}

Then \cite[Theorem 1.1]{Ka} asserts: 
\begin{equation}
c(W_-(T_{p, -q})) = 2(pq+\frac{1}{2}) = 2pq + 1,\quad \mathrm{and}
\end{equation}
\begin{equation}
-(pq+\frac 1 2) = 1-g_T(W_-(T_{p, -q}))-c(W_-(T_{p, -q}))/2  =1-g_T(W_-(T_{p, -q}))-(pq+\frac 1 2).
\end{equation}
Hence, we have $g_T(W_-(T_{p, -q}))=1$ as desired. 
\end{proof} 

\begin{lemma}
\label{Turaev_genus_not_1}
$g_T(W_-(T_{2, -(2m+1)})) \ne 1$.
\end{lemma}

\begin{proof}
It follows from \cite[Lemmas 4.1 and 4.5]{DFKLS} that a Turaev genus one knot admits an alternating projection on a standardly embedded torus in $S^3$. 
Recall that since any Whitehead double has (Seifert) genus one, 
it is prime. 
Then Adams \cite{Adams} has shown that such knots are either torus knots or hyperbolic knots. 
Thus the Whitehead double of any non-trivial knot cannot have Turaev genus one. 
\end{proof}

\begin{proof}[Proof of Theorem~\ref{non-adequate_W}]
Assume for a contradiction that $W_-(T_{2, -(2m+1)})$ is adequate. 
Following Lemma~\ref{Turaev_genus_1} $g_T(W_-(T_{2, -(2m+1)})) = 1$. 
On the contrary, Lemma~\ref{Turaev_genus_not_1} shows that 
$g_T(W_-(T_{2, -(2m+1)})) \ne 1$, a contradiction. 
\end{proof}

Theorem~\ref{non-adequate_W} says that $W_{-}(K)$ is not adequate, 
and we can see that a modification of the diagram of $W_{-}(K)$ in Figure~\ref{fig:Wwtwrithe0} is $A$--adequate. 
So it is not $B$--adequate. 
However Theorem~\ref{bothslopeconjecturesWmt} (or Theorem~\ref{slopeconjectureWmtSignCond}) 
shows that $W_{-}(K)$ satisfies the Strong Slope Conjecture. 

\bigskip

\textbf{Acknowledgements} --
We would like to thank the referee for careful reading and valuable suggestions that have enabled us to improve the article. 
In particular, the proof of Lemma~\ref{Turaev_genus_not_1} was simplified by the referee's suggestion.

\end{document}